\documentclass{amsart}
\usepackage[utf8]{inputenc}
\usepackage{bbold}
\usepackage{amsmath, amssymb, hyperref}
\usepackage{graphicx}
\usepackage{wrapfig}
\usepackage{amsthm}
\usepackage{color}
\usepackage{float}
\usepackage{tikz}
\usetikzlibrary{cd}
\usetikzlibrary{arrows.meta}

\graphicspath{{figures/}}

\usepackage[utf8]{inputenc}
\usepackage{amsmath, color}
\usepackage{graphicx}
\usepackage{wrapfig}
\usepackage{amsthm}
\usepackage[margin=1in]{geometry}
\usepackage[capitalize, nameinlink, noabbrev, nosort]{cleveref}
\usepackage{comment}

\newtheorem{maintheorem}{Theorem}

\newtheorem*{corollary*}{Corollary}
\newtheorem{theorem}{Theorem}[section]

\newtheorem{lemma}[theorem]{Lemma}
\newtheorem{proposition}[theorem]{Proposition}

\newtheorem{corollary}[theorem]{Corollary}

\theoremstyle{definition}

\newtheorem{definition}[theorem]{Definition}
\newtheorem{remark}[theorem]{Remark}

\newtheorem{assumption}[theorem]{Assumption}

\DeclareMathOperator{\mrk}{mrk}
\DeclareMathOperator{\wt}{wt}
\DeclareMathOperator{\height}{ht}

\DeclareMathOperator{\Edg}{E}

\DeclareMathOperator{\faces}{Face}
\DeclareMathOperator{\Gr}{Gr}
\DeclareMathOperator{\bd}{bd}
\DeclareMathOperator{\fr}{fr}
\DeclareMathOperator{\pr}{pr}
\DeclareMathOperator{\ext}{ext}

\DeclareMathOperator{\Newton}{{\rm Newton}}
\def\extQ{Q^{\ext}}
\def\frzQ{Q^{\fr}}
\def\Qprin{Q^{\pr}}
\def\prinSig{\Sigma^{\pr}}

\def\hvec^#1_#2{\mathbf{h}^{#1}\!(#2)}

\title{Dimer face polynomials in knot theory and cluster algebras}
\author{Karola M\'esz\'aros, Gregg Musiker, Melissa Sherman-Bennett, Alexander Vidinas}
\date{August 2024}

\begin{document}

 \begin{abstract} 

The set of perfect matchings of a connected bipartite plane graph $G$ has the structure of a distributive lattice, as shown by Propp, where the partial order is induced by the \emph{height} of a matching. In this article, our focus is the \emph{dimer face polynomial} of $G$, which is the height generating function of all perfect matchings of $G$.  We connect the dimer face polynomial on the one hand to knot theory, and on the other to cluster algebras. We show that certain dimer face polynomials are multivariate generalizations of Alexander polynomials of links, highlighting another combinatorial view of the Alexander polynomial. We also show that an arbitrary dimer face polynomial is an $F$-polynomial in the cluster algebra whose initial quiver is dual to the graph $G$. As a result, we recover a recent representation theoretic result of Bazier-Matte and Schiffler that connects $F$-polynomials and Alexander polynomials, albeit from a very different, dimer-based perspective. As another application of our results, we also show that all nonvanishing Pl\"ucker coordinates on open positroid varieties are cluster monomials.

 \end{abstract}

\maketitle

\setcounter{tocdepth}{1}
	\tableofcontents

\section{Introduction}

Perfect matchings, also called \emph{dimers}, are ubiquitous in mathematics. They appear in statistical mechanics \cite{cohn1996local, kenyon2010dimer}, the theory of cluster algebras \cite{di2014t,GK}, the study of string theory \cite{ eager2012colored, franco2006brane}, and in knot theory \cite{CohenThesis, TwistedDimer,  ClockLattice}, to name a few. 

Propp \cite{P02} showed that the set of dimers of a connected plane graph $G$ can be endowed with the structure of a distributive lattice, as long as every edge of $G$ is in some perfect matching (see Theorem \ref{thm:dimer-lattice} below). One way to phrase this lattice structure, which is also due to Propp, is to give each matching $M$ of $G$ a height vector $\height(M) \in \mathbb{Z}^{\faces(G)}$. The partial order on matchings is then given by coordinate-wise comparison of height vectors.

\medskip

The main object of study throughout this paper is the height generating function of all perfect matchings
\[D_G= \sum_{M \text{ dimer}} \mathbf{y}^{\height(M)}\]
which we call the \textbf{dimer face polynomial}. The height function has been widely used and studied in literature on perfect matchings, for example in \cite{EKLP,P02}. For special graphs called \emph{snake graphs}, the dimer face polynomial was studied in \cite{mussch,MSW,MSW2}. The dimer face polynomial also is implicit in work of \cite{vichitkunakorn2016solutions} and \cite{jeong2013gale} in the context of the Octahedron and Gale-Robinson recurrences, respectively. See also \cite[Section 6]{eager2012colored} where such polynomials are studied for brane tilings in string theory, and are referred to as \emph{colored partition functions}. 
The dimer face polynomial is related to the edge generating function of perfect matchings, also sometimes called the \emph{partition function}, by a change of variables followed by a rescaling (see \cref{rem:dimer-polyomial-edges}). 

Our main results relate the dimer face polynomial on the one hand to knot theory, and on the other hand to cluster algebras. In the knot theory direction, building on results of Cohen, Dasbach, Russell and Teicher  \cite{CohenThesis, TwistedDimer, ClockLattice}, we express the Alexander polynomials of links as specializations of dimer face polynomials:

\begin{maintheorem}\label{thm:intro-Alex-poly} For a link diagram $L$ and a segment $i$, let $G_{L,i}$ be the associated   truncated face-crossing incidence graph. Then, the dimer face polynomial $D_{G_{L,i}}$ equals the  Alexander polynomial $\Delta_L(t)$ after we specialize its variables to $-1, -t, -t^{-1}$ (as determined by $L$).
   \end{maintheorem} 

Theorem \ref{thm:intro-Alex-poly} appears as \cref{thm:alexdimer} in the text. There are previous formulations of the Alexander polynomial in terms of dimers, using edges rather than faces \cite{CohenThesis, TwistedDimer}, which are closely related to Theorem ~\ref{thm:intro-Alex-poly} as we explain in Section \ref{sec:connection-twisted-dimers}. Indeed, there is a vast pool of combinatorial interpretations of the Alexander polynomial \cite{crowell1959genus, K06, murasugi2003alexander}, and Theorem ~\ref{thm:intro-Alex-poly} adds  another way of thinking about the Alexander polynomial by weighting the faces of planar bipartite graphs. 

\medskip
In the cluster algebra direction, we show the following result, which appears as \cref{thm:dimer-poly-is-F-poly}.

\begin{maintheorem}\label{thm:intro-dimer-F-poly} Let $G$ be a connected plane graph and suppose every edge of $G$ is in some perfect matching. Let $Q_G$ denote the dual quiver of $G$.
The dimer face polynomial of $G$ is an $F$-polynomial in the cluster algebra $\mathcal{A}(Q_G)$.
\end{maintheorem}

Further, we give an explicit mutation sequence giving this $F$-polynomial (which can be read off of $G$), determine the corresponding $g$-vector, and give a cluster expansion formula for the corresponding cluster variable $x$ in terms of the dimers of $G$. Interestingly, the denominator vector of $x$ is the all-1's vector.

Perfect matchings of plane graphs have appeared previously in cluster algebras and related fields. The $F$-polynomials in cluster algebras from surfaces can be obtained by specializing (attaching the same variable to multiple faces)
dimer face polynomials of various \emph{snake graphs} \cite{MSW}. If the surface is a polygon, so the cluster algebra is type $A$, no specialization is necessary and the $F$-polynomials are dimer face polynomials. \cref{thm:intro-dimer-F-poly} recovers these dimer formulas for type $A$ $F$-polynomials when $G$ is chosen to be a snake graph (see \cref{prop:TypeAF}). 
For special graphs coming from link diagrams, \cref{thm:intro-dimer-F-poly} has the following corollary, which appears in the text as \cref{cor:alex-poly-specialization-F-poly}.

\begin{corollary*}
  The Alexander polynomial of any link $L$ is a specialization of an $F$-polynomial in $\mathcal{A}(Q_L)$.  
\end{corollary*}
 This insight was first conjectured in the beautiful work of Bazier-Matte and Schiffler \cite{B21}, which served as an inspiration for our work. It was also proved independently by Bazier-Matte and Schiffler in their recent preprint \cite{BMS24}, which appeared during the final stages of the preparation of this manuscript. 

We also give two applications of \cref{thm:intro-dimer-F-poly}. The first gives a deeper connection between the knot-theoretic and cluster algebraic results on the dimer face polynomial. Each link diagram for a link $L$ gives rise to many dimer face polynomials, all of which specialize to the Alexander polynomial of $L$. In \cref{thm:cluster-stuff-for-link-diag}, we show that all of these dimer face polynomials are $F$-polynomials in a single cluster algebra, confirming part of a conjecture of \cite{B21}. A similar result independently appeared in \cite{BMS24}.

There is a resemblance between the combinatorics of dimers on plane graphs and the study of almost-perfect matchings, see \cref{def:almost-perfect}, on planar bicolored graphs on a disk (plabic graphs).  Almost-perfect matchings on plabic graphs have been used in the combinatorial study of total positivity as initiated by Postnikov \cite{Postnikov}, and in particular in the works \cite{marsh2016twists,MS-twist}.  

Using this connection, the second application of \cref{thm:intro-dimer-F-poly} involves \emph{open positroid varieties} $\Pi_G^{\circ}$. Open positroid varieties $\Pi_G^{\circ}$ are subvarieties of the Grassmannian, introduced in \cite{KLS}. 
The coordinate ring $\mathbb{C}[\Pi_G^{\circ}]$ is generated as an algebra by Pl\"ucker coordinates. Additionally, $\mathbb{C}[\Pi_G^{\circ}]$ is a cluster algebra \cite{GL-positroid}, and so is generated as an algebra by the cluster variables. We determine the relationship between these two generating sets in the result below, which appears as \cref{thm:pluckers-cluster-monomials}.

\begin{maintheorem}\label{thm:intro-pluckers-cluster-mono}
    All nonvanishing Pl\"ucker coordinates on $\Pi_G^{\circ}$ are cluster monomials.
\end{maintheorem}

At first glance, this result has little to do with dimers. However, we prove \cref{thm:intro-pluckers-cluster-mono} by relating the \emph{almost perfect matching} formulas of Muller--Speyer for twisted Pl\"ucker coordinates $P_J \circ \tau$ \cite{MS-twist} to the dimer formulation for $F$-polynomials in \cref{thm:intro-dimer-F-poly}.

We begin our paper with a focus on the combinatorics of graphs and lattices, followed by a presentation of topological and algebraic applications. In \cref{sec:dimer-lattice-and-poly}, we discuss the dimer lattice and the dimer face polynomial, as well as some of its properties.  In \cref{sec:newton-polytope} we study the Newton polytope of the dimer face polynomial $D_G$, show it is affinely isomorphic to the perfect matching polytope of $G$, and determine its face lattice. 
\cref{sec:alexander-poly-and-dimers} covers background on Kauffmann's lattice and Alexander polynomial, followed by a proof of Theorem \ref{thm:intro-Alex-poly}. The remaining three sections focus on the dimer face polynomial and its relation to cluster algebras.  In \cref{sec:background-cluster}, we review the relevant theory of cluster algebras and then prove \cref{thm:intro-dimer-F-poly} in \cref{sec:dimer_f_poly} after building up the theory even further. In \cref{sec:further-applications-to-links}, we restrict our attention to graphs coming from links, and show that many different $F$-polynomials specializing to the Alexander polynomial of $L$ can be found in one cluster algebra. We also examine the case of $2$-bridge links and their relation to snake graphs. In \cref{sec:positroid-var-applications}, we discuss open positroid varieties and prove \cref{thm:intro-pluckers-cluster-mono}.

\section{Dimer lattices}\label{sec:dimer-lattice-and-poly}

In this section, we introduce a distributive lattice of perfect matchings on plane bipartite graphs, following the work of Propp \cite{P02}. Using this lattice for a plane bipartite graph, we define the \textit{dimer face polynomial} in terms of both \emph{multivariate ranks} and \emph{heights} of matchings. As we will discuss in Section \ref{sec:alexander-poly-and-dimers}, the dimer face polynomial generalizes the Alexander polynomial of a link, and as we will discuss in Section \ref{sec:dimer_f_poly}, it is an $F$-polynomial in a distinguished cluster algebra.

\subsection{The dimer lattice of a plane graph} \label{sec:dimer lattice}
Recall that a plane graph is a planar graph with a choice of embedding in the plane $\mathbb{R}^2$.  We call the connected regions of the complement $\mathbb{R}^2 \setminus G$ the \textbf{faces} of $G$. We often identify the faces of plane graphs with the set of edges in their boundary. A \textbf{perfect matching} of a graph is a subset of the edges incident to each vertex exactly once. We often say ``matching" or ``dimer" instead of ``perfect matching." 

\begin{definition} 
    A graph $G$ \textbf{has property $(*)$} if it is a finite connected bipartite plane graph, with vertices properly colored black and white, and every edge is in some perfect matching. We view the vertex coloring as fixed, just as the embedding of $G$ in the plane is fixed.

    We denote the collection of all perfect matchings on a plane graph $G$ by $\mathcal{D}_G$. Additionally, we use $\faces(G)$ to denote the set of all non-infinite faces of $G$.
\end{definition}

\begin{definition}
    Let $G$ be a bipartite plane graph. Choose a non-infinite face $f$ of $G$ whose boundary is a cycle. An edge $e$ in the boundary of $f$ is \textbf{black-white} in $f$ if, going around $f$ clockwise, we see first the black vertex of $e$ and then the white vertex. The \textbf{white-black} edges of $f$ are defined similarly.
\end{definition}

\begin{definition}\label{def:dimer-poset}
    Let $G$ be a connected bipartite plane graph. Let $M$ be a perfect matching and let $f\in\faces(G)$ be a face whose boundary is a cycle. If $M$ contains all of the black-white edges of $f$, then we may obtain a new matching $M'$ by removing the black-white edges of $f$ from $M$ and replacing them with the white-black edges of $f$. We call this operation the \textbf{down-flip} at $f$. Similarly, if $M$ contains all of the white-black edges of $f$, we may replace them with the black-white edges to obtain a new matching $M'$; we call this operation the \textbf{up-flip} at $f$. In either of these situations, we say $M, M'$ are related by a flip at $f$.
\end{definition}

\begin{remark}\label{rmk:prop*-no-dangling-edges}
    If $G$ has property $(*)$, then the boundary of any face is a cycle (otherwise, some edge would have the same face on both sides.  We later show that this would lead to a contradiction, see \cref{lem:edges-dif-face-each-side}). Thus, when $G$ has property $(*)$, flips are well-defined at any face $f \in \faces(G)$.
\end{remark}

Propp showed that the dimer sets of graphs with property $(*)$ are very structured.

\begin{theorem}[{\cite[Theorem 2]{P02}}]\label{thm:dimer-lattice}
    Let $G$ be a graph with property $(*)$. We define a binary relation on the set $\mathcal{D}_G$ of dimers by declaring $M \le M'$ if $M$ is obtained from $M'$ by applying a sequence of down-flips.
    
    Then $(\mathcal{D}_G, \le)$ is a distributive lattice, which we call the \textbf{dimer lattice of $G$}. 
\end{theorem}

See \cref{fig:general_weights} for an illustration of \cref{thm:dimer-lattice}.

\begin{remark}\label{rmk:dimer-lattice-arbitrary-graph}
    One can place a distributive lattice structure on $\mathcal{D}_{G}$ for an arbitrary bipartite plane graph as follows. If some edges of $G$ are not in any perfect matching, delete them to obtain a new plane graph $G'=G_1 \sqcup \cdots \sqcup G_r$, possibly with multiple connected components. Each connected component $G_i$ has property $(*)$, so $(\mathcal{D}_{G_i}, \le)$ is a lattice by \cref{thm:dimer-lattice}. There is a natural bijection between $\mathcal{D}_G$, and the product $\prod \mathcal{D}_{G_i}$, so one may endow $\mathcal{D}_G$ with a lattice structure via this bijection. However, if any non-infinite face of $G_i$ is not a face of $G$, some covering relations in this lattice structure do not correspond to flips on faces of $G$.  See an example of this in \Cref{fig:example-delete-edges-lattice}.
\end{remark}

\begin{figure}
    \centering
    \includegraphics[width=0.9\linewidth]{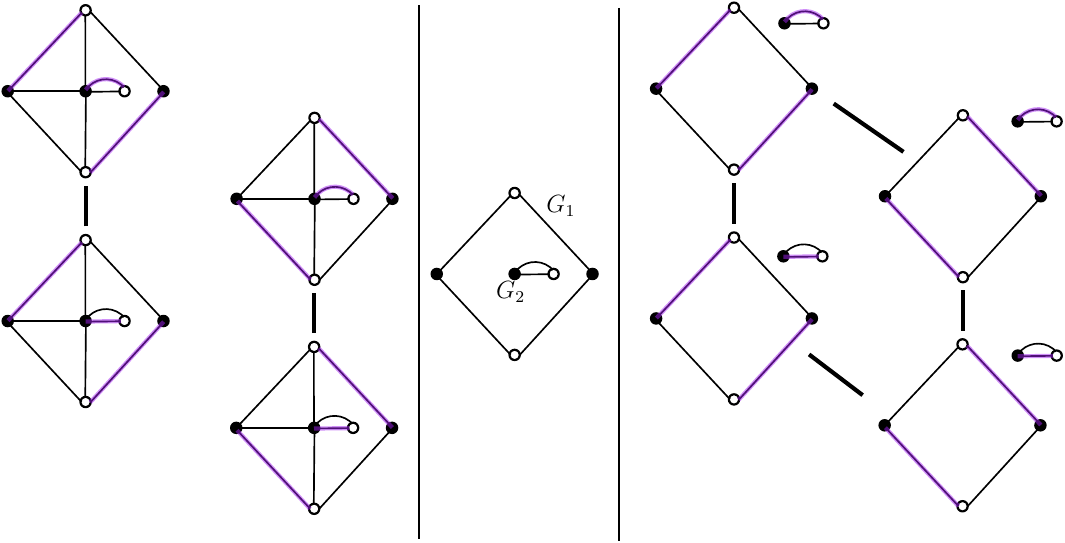}
    \caption{An illustration of \cref{rmk:dimer-lattice-arbitrary-graph}. Left: the poset $(\mathcal{D}_G, \le)$ for a graph $G$ which does \emph{not} have property $(*)$. Center: the graph $G'=G_1 \sqcup G_2$ obtained from $G$ by deleting edges not in any matching. Right: The lattice $\mathcal{D}_{G_1} \times \mathcal{D}_{G_2}$. Note that there are two cover relations that are not present in $(\mathcal{D}_G, \le)$, since the unique face of $G_1$  is not a face of $G$. 
    }
    \label{fig:example-delete-edges-lattice}
\end{figure}

 \subsection{The dimer face polynomial}

 In this section, building on \cite[Section 2]{P02} and subsequent work (see \cite[Section 5]{P02} for a detailed discussion), we associate a multivariate polynomial to a graph $G$ with property $(*)$, using the dimer lattice. In particular, we associate to each dimer $M$ a monomial \emph{multivariate rank} $\mrk(M)$, and the dimer face polynomial is the sum of these multivariate ranks. At the end of the section, we rephrase the dimer face polynomial using heights.
 
 \begin{definition} 
    Let $G$ be a graph with property $(*)$. Let $C$ be a saturated chain $\hat{0}=M_0\lessdot M_1\lessdot \dots \lessdot M_\ell$ in the dimer lattice and suppose that $M_{i-1}, M_{i}$ are related by a flip at face $f_i$. The \textbf{weight of the saturated chain} $C$ is 
    \[\wt(C):= y_{f_1} \cdots y_{f_\ell}.\]
\end{definition}

\begin{proposition}\label{prop:satchains}
     Let $G$ be a graph with property $(*)$. Choose $M$ a dimer on $G$ and let $C, C'$ be two saturated chains from $\hat{0}$ to $M$. Then $\wt(C)=\wt(C')$, or, in other words, the weight of a saturated chain depends only on the largest element in the chain.
\end{proposition}

We first prove a straightforward lemma.

\begin{lemma}\label{lem:commuting}
   Let $G$ be a graph with property $(*)$, and let $M$ be a dimer on $G$. Suppose $M'$ is obtained from $M$ by a down-flip at face $a$ and $M''$ is obtained from $M$ by a down-flip at face $b$, where $a \neq b$. Then there exists a dimer $M'''$ of $G$ which is obtained from $M'$ by a down-flip at face $b$ and is obtained from $M''$ by a down-flip at face $a$.
\end{lemma}

\begin{proof}
By the assumption that we may do a down-flip at both $a$ and $b$, the black-white edges of $a$ and $b$ are contained in $M$. This implies that $a$ and $b$ are disjoint. Indeed, suppose for contradiction that there is a vertex $v$ in $a \cap b$. As $M$ is a perfect matching, $v$ is in a single edge $e$ of $M$. Since we may perform a down-flip at both $a$ and $b$, $e$ must, in fact, be in both $a$ and $b$. But then $e \in M$ is black-white in exactly one of $a$ and $b$, a contradiction.

Since $a$ and $b$ are disjoint, flips at $a$ and $b$ commute. The desired matching $M'''$ is the matching obtained from $M'$ by a down-flip at $b$, which is the same as the matching obtained from $M''$ by a down-flip at $a$.
\end{proof}

We now proceed to the proof of Proposition \ref{prop:satchains}.

\noindent \textit{Proof of Proposition \ref{prop:satchains}.} We proceed by induction on the length of the saturated chain from $\hat{0}$ to $M$. The base case is $M=\hat{0}$, for which the statement is trivially true.

Suppose the length of any saturated chain from $\hat{0}$ to $M$ is at least two. Let $C'=\hat{0} \lessdot M_1\lessdot \cdots \lessdot M' \lessdot M$ and $C''=\hat{0} \lessdot N_1 \cdots \lessdot M'' \lessdot M$ be saturated chains from $\hat{0}$ to $M$. If $M'=M''$, then we have $\wt(C)=\wt(C')$ by the inductive hypothesis. If $M'$ and $M''$ are distinct, say they are obtained from $M$ by a down-flip at faces $a, b$ respectively. By Lemma \ref{lem:commuting}, there exists $M'''$ covered by both $M'$ and $M''$ such that $\wt(M''' \lessdot M')=y_b$ and $\wt(M''' \lessdot M'')=y_a$.

Now, fix any saturated chain $S=\hat{0}\lessdot\cdots\lessdot M'''$ and let $S'$, $S''$ be the chains obtained by adding $M'$ and $M''$ to the end of this chain. We have 
\begin{align*}
    \wt(C')= \wt(\hat{0} \lessdot M_1 \cdots \lessdot M') y_a
    = \wt(S') \cdot  y_a
    = \wt(S) \cdot y_a y_b
\end{align*}
where the second equality is by the inductive hypothesis and the third equality is by the choice of $M'''$. 
We have a similar string of inequalities for $C''$:
\begin{align*}
    \wt(C'')= \wt(\hat{0} \lessdot N_1 \cdots \lessdot M'') y_b
    = \wt(S'') \cdot  y_b
    = \wt(S) \cdot y_a y_b.
\end{align*}
This shows $\wt(C')=\wt(C'')$ as desired.
\qed

Using \cref{prop:satchains}, we now define the ``multivariate rank" of a dimer and define the dimer face polynomial.

\begin{definition}\label{def:dimer-poly} 
    Let $G$ be a graph with property $(*)$. For $M$ a dimer on $G$ and $C$ any saturated chain from $\hat{0}$ to $M$, define the \textbf{multivariate rank} of $M$ to be 
    \[\mrk(M)=\wt(C).\]
   The \textbf{dimer face polynomial} of $G$ is
    \[
    D_G(\mathbf{y}) = \sum_{M \in \mathcal{D}_G}\mrk(M).
    \]
\end{definition}

\begin{figure}
    \centering
    \includegraphics[height=13cm]{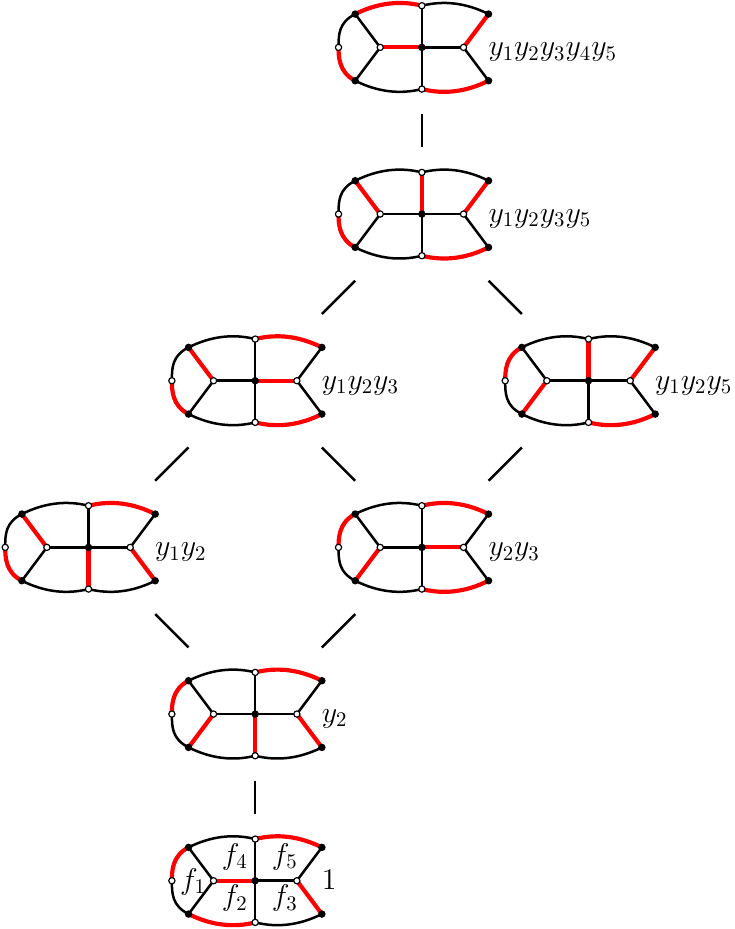}
    \caption{The dimer lattice $\mathcal{D}_G$ depicted with the multivariate rank of each dimer.}
    \label{fig:general_weights}
\end{figure}

In the remainder of this section, we give an alternate definition of the dimer face polynomial using heights.  Height functions, as we use them, are first defined in Elkies, Larsen, Kuperberg, Propp \cite[Section 2]{EKLP} as inspired by earlier work enumerating tilings including work of Conway-Lagarias \cite{conway1990tiling} and Thurston \cite{thurston1990conway}.  The second author and Schiffler originated the use of height functions in the context of cluster algebras in \cite[Section 5]{mussch} in the special case of unpunctured surfaces and certain graphs known as snake graphs (see \cref{def:snake_graphs} for more details).   More broadly, the height of a matching was used by Propp to define the dimer lattice in Section 3 of \cite{P02}, and so many of the properties of heights appear in \cite{P02}.  For the reader's convenience, we provide self-contained proofs here.  

We first need an elementary lemma, whose proof we omit.

\begin{lemma} \label{lem:symm}  Let $G$ be a graph with property $(*)$ and $M, M' \in \mathcal{D}_G$ two dimers. The symmetric difference $M \triangle M'$ is a disjoint union of cycles and isolated vertices. \end{lemma}

To define heights, we first need to orient the cycles of $M \triangle \hat{0}$.

\begin{definition}\label{def:oriented-cycles}
     Let $G$ be a graph with property $(*)$, and let $M \in \mathcal{D}_G$ be a dimer. Orient the edges of $M \triangle \hat{0}$ as follows: each edge of $M \setminus \hat{0}$ is oriented from black vertex to white and each edge in $\hat{0} \setminus M$ is oriented from white vertex to black. We denote the resulting collection of oriented cycles $\overrightarrow{M \triangle 0}$. 
\end{definition}
Note that any cycle in $\overrightarrow{M \triangle 0}$ is either ``clockwise," meaning all edges are oriented clockwise, or ``counterclockwise." See \cref{fig:height} for an example.

\begin{definition}\label{def:height}
   Let $G$ be a graph with property $(*)$, and let $M \in \mathcal{D}_G$ be a dimer.
    The \textbf{height} of $M$ is the vector $\height(M) \in \mathbb{Z}^{\faces(G)}$ where the coordinate indexed by $f \in \faces(G)$ is
    \[\height(M)_f= \#\{\text{clockwise cycles in }\overrightarrow{M \triangle 0} \text{ encircling}\footnote{By the Schoenflies theorem, the complement of any cycle $C$ in $M\triangle\hat{0}$ has a bounded component and an unbounded component. The cycle $C$ encircles $f$ if $f$ is in the bounded component.}f\}- \#\{\text{counterclockwise cycles in }\overrightarrow{M \triangle 0} \text{ encircling }f\}.\]
 Equivalently, if $p$ is any point in the face $f$, 
    $$\height(M)_f = \sum_{C\in\overrightarrow{M \triangle 0}}-\text{wind}(C,p),$$
    where $\text{wind}(C,p)$ is the winding number of $C$ around $p$.
\end{definition}
For convenience, we also allow $\height(M)$ to have a coordinate indexed by the infinite face. Following the definition above, this coordinate is $0$ for all matchings.

\begin{remark} The above notion of height is defined with respect to the minimal element $\hat{0}$ of the dimer lattice. However, we could have equally defined the height $\height_{\widetilde{M}}(M)_f$ with respect to a fixed matching $\widetilde{M}$ in the dimer lattice  \cite{P02}. Such a definition would satisfy $\height(M)_f=\height_{\widetilde{M}}(M)_f-\height_{\widetilde{M}}(\hat{0})_f.$

\end{remark}

\begin{figure}
    \centering
    \includegraphics[height=4cm]{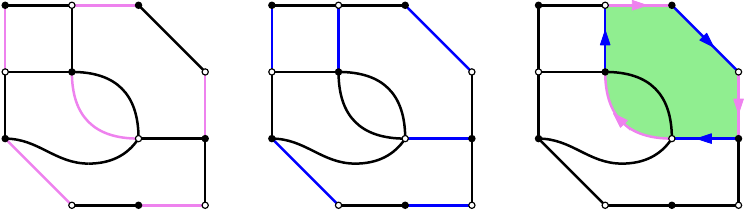}
    \caption{Left: the dimer $\hat{0}$ for the graph $G$ depicted. Center: a dimer $M$. Right: the edges in $M\triangle\hat{0}$, which form a single clockwise cycle encircling the highlighted faces.}
    \label{fig:height}
\end{figure}

The height of a face can also be computed using walks on the dual graph $G^*$ of $G$ as in \cite{P02}, which will be useful in proofs.

\begin{proposition}\label{prop:ht_altitude}
    Let $G$ be a graph with property $(*)$, let $M$ be a dimer on $G$ and let $f\in\faces(G)$. Let $P=(p_1,\dots,p_n)$ be a walk in $G^*$ from $f$ to the infinite face. Starting at $f$ and traveling along $P$, let $L_M(P)$ be the number of edges $p_i$ which pass through an edge in $\overrightarrow{M \triangle 0}$ directed left-to-right, and let $R_M(P)$ be the number of edges $p_i$ which pass through an edge in $\overrightarrow{M \triangle 0}$ directed right-to-left. We have
    $$\height(M)_f=L_M(P)-R_M(P).$$
\end{proposition}

\begin{proof}
    We observe that the walk $P$ must cross each cycle encircling $f$ at least once. Moreover, if $C$ encircles $f$, $P$ must cross $C$ an odd number of times to reach the infinite face. If $C$ is oriented clockwise around $f$, each time $P$ exits the region encircled by $C$, it crosses an edge directed left-to-right, and every time $P$ re-enters, it crosses an edge directed right-to-left. 
    As such, after performing all possible cancellations, $C$ contributes $+1$ to the difference $L_M(P)-R_M(P)$. Similarly, if $C$ is oriented counterclockwise around $f$, it contributes $-1$ to the difference (see Figure \ref{fig:altitude_path}). 
On the other hand, if $C$ does not encircle $f$, then $P$ must cross $C$ an even number of times, and the edges of $C$ that $P$ crosses have alternating orientations.
\end{proof}

\begin{figure}
        \centering
        \includegraphics[height=4cm]{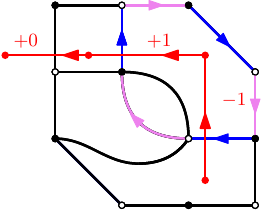}
        \caption{A walk from a height-$0$ face to the exterior, along with the contribution of each oriented edge, continuing the example in Figure \ref{fig:height}.}
    \label{fig:altitude_path}
\end{figure}

As we now show, the height of a matching is simply a reformulation of its multivariate rank (see also Figure 11 in \cite[Section 3]{P02} for a description in terms of contour lines).  

\begin{theorem}\label{cor:ht}
    Let $G$ be a graph with property $(*)$ and $M \in \mathcal{D}_G$ a dimer. Then 
    \[\mrk(M)= \mathbf{y}^{\height(M)}\]
    and in particular, the coordinates of $\height(M)$ are nonnegative. 
    
    Moreover,
    \[D_G(\mathbf{y}) = \sum_{M \in \mathcal{D}_G} \mathbf{y}^{\height(M)}.\]
\end{theorem}

To prove \cref{cor:ht}, we need a few preparatory statements. 

\begin{lemma} \label{lem:ht_switch}
If $M'$ is obtained from $M$ by an up-flip or a down-flip over $f$, then for any edge $e$ bounding $f$, $e\in M'\triangle\hat{0}$ if and only if $e\notin M\triangle\hat{0}$.  
\end{lemma}

\proof 
Let $e$ be an edge bounding $f$. The edge $e$ is in $M'\triangle\hat{0}$ if and only if $e \in M' \setminus \hat{0}$ or $e \in \hat{0} \setminus M'$. The former condition is equivalent to $e \notin M \cup \hat{0}$, and the latter is equivalent to $e \in \hat{0} \cap M$, since on the boundary of $f$, $M$ and $M'$ are complements. So $e \in M'\triangle\hat{0}$ if and only if $e \notin M \cup \hat{0}$ or $e \in \hat{0} \cap M$, which is equivalent to $e \notin M \triangle \hat{0}$.
\qed 

\begin{proposition}\label{lem:ht_flip} 
    Let $G$ be a graph with property $(*)$ and let $M,M' \in \mathcal{D}_G$. Suppose that $M$ is obtained from $M'$ by a down-flip on face $f$. Then,
    $$\height(M) = \height(M')-\mathbf{e}_f.$$
\end{proposition}

\begin{proof} We first compare $\height(M)_f$ and $\height(M')_f$.
    Let $P$ be a walk in $G^*$ from $f$ to the infinite face and let $p_1$ be its first edge, which crosses edge $e$ of $f$. Suppose $e \in M' \triangle \hat{0}$. Since $M'$ contains exactly the black-white edges of $f$, if $e$ is black-white then $e \in M'$. If $e$ is white-black in $f$, then $e \in \hat{0}$. In either case, according to the definition of $\overrightarrow{M' \triangle 0}$, $e$ is oriented left-to-right with respect to $p_1$. As such, $L_{M'}(P)=L_M(P)+1$. If $e \notin M' \triangle \hat{0}$, then by \cref{lem:ht_switch}, $e \in M \triangle \hat{0}$. A similar argument shows that $e$ is oriented right-to-left with respect to $p_1$, and so $R_{M'}(P)=R_M(P)-1$. In either case, \cref{prop:ht_altitude} implies $\height(M')_f = \height(M)_f+1$.
    
    Next, fix a face $f' \neq f$. We will show that $\height_{f'}(M)=\height_{f'}(M')$, which will complete the proof. Let $P$ be a walk in $G^*$ from $f'$ to the infinite face. If $P$ does not travel through the interior of $f$, then the equality is clear. Otherwise, some edge $p$ of $P$ passes into $f$ through an edge $e$, then is followed by an edge $q$ exiting $f$ through some edge $d$. Note that each of $e$ and $d$ contribute to exactly one of $L_M(P), R_M(P), L_{M'}(P),$ and $R_{M'}(P)$. If $d \in M' \triangle \hat{0}$, the previous paragraph shows that $d$ contributes 1 to $L_{M'}(P)$; otherwise, it contributes $1$ to $R_{M}(P)$. Similar statements hold for $e$, switching $M$ and $M'$. One can check that in all cases, the desired equality holds.
    \end{proof}

\begin{proof}[Proof of \cref{cor:ht}]
    The first sentence follows from \Cref{lem:ht_flip} by induction on the rank of $M$. The base case is $\mrk(\hat{0})=1=\mathbf{y}^{\height(\hat{0})}$. The second sentence follows from the first, by the definition of $D_G$.
\end{proof}

Finally, we deduce some corollaries from the formulation of $D_G$ in terms of height. 

\begin{lemma} \label{lem:mon} Let $G$ be a graph with property $(*)$. If ${\bf y}^{\height(M)}= {\bf y}^{\height(M')}$ for dimers $M, M' \in \mathcal{D}_G$, then $M=M'$.
\end{lemma}

\proof We claim that the height function ${\height(M)}$ uniquely determines the cycles in $M \triangle \hat{0}$, which in turn uniquely determines $M$. Let $e$ be an edge of $G$, and let $f$ and $f'$ denote the faces on either side of $e$. Let $p_0$ denote the edge of $G^*$ dual to $e$, and let $P'=(p_1,\dots,p_n)$ be a path from $f'$ to the exterior face. Observe, $P=(p_0,\dots,p_n)$ is a path from $f$ to the exterior. As such, $\height(M)_f$ and $\height(M)_{f'}$ differ if and only if $e$ is in $M\triangle\hat{0}$.
\qed

\begin{corollary} \label{cor:01} The coefficients of $D_G(\mathbf{y})$ are all $0$ and $1$. 
\end{corollary}

\begin{corollary}\label{cor:prop*=all-faces-flipped}
    Let $G$ be a connected plane graph. Then $G$ has property $(*)$ if and only if the boundary of the infinite face is a cycle and for every non-infinite face of $f$, there exists a matching $M$ of $G$ which admits an up-flip at $f$.
\end{corollary}
\begin{proof}

    Suppose $G$ has property $(*)$ and choose a non-infinite face $f$ of $G$. The face $f$ has at least one black-white edge $e$ which is not in the bottom matching $\hat{0}$ (otherwise, $\hat{0}$ would contain all black-white edges of $f$ and we could perform a down-flip at $f$ in $\hat{0}$, which is impossible). Since $G$ has property $(*)$, there is a matching $N$ which contains $e$. We have $e \in N \triangle \hat{0}$. Let $f'$ be the other face adjacent to $e$; by \cref{lem:edges-dif-face-each-side}, $f \neq f'$. Note that in $\overrightarrow{N\Delta\hat{0}}$, $e$ is directed black-to-white. Traveling from $f$ to $f'$ across $e$, $e$ is oriented left-to-right. Using \cref{prop:ht_altitude}, we conclude that $\height(N)_{f'}=\height(N)_f-1$. Since all components of $\height(N)$ are nonnegative, this implies $\height(N)_f \geq 1$. In other words, $y_f$ divides $\mathbf{y}^{\height(N)}$. By \cref{cor:ht} and the definition of $\mrk(N)$, this implies that there is some matching $\hat{0}\leq M \leq N$ which admits an up-flip at $f$. Property $(*)$ also implies the condition on the boundary of the infinite face (see \cref{rmk:prop*-no-dangling-edges}).

    Suppose that for every non-infinite face of $f$, there is a matching $M$ which admits an up-flip at $f$. Suppose also that the boundary of the infinite face is a cycle. 
    The first assumption implies that the boundary of every non-infinite face is a cycle, and so every edge in a non-infinite face $f$ is either white-black or black-white in $f$. 
    Since matching $M$ of $G$ admits an up-flip at $f$, there is a matching that contains every white-black edge of $f$.  After applying this up-flip, we obtain a matching $M'$ of $G$ that contains every black-white edge of $f$. The assumption that the boundary of the infinite face is a cycle implies that every edge $e$ is in a non-infinite face of $G$. In this face, $e$ is either a white-black edge or a black-white edge, so the previous arguments imply that $G$ is in fact bipartite and that every edge of $G$ is in some matching, and thus $G$ has property $(*)$.
\end{proof}

\section{Newton polytopes of dimer face polynomials} \label{sec:newton-polytope}

In this section, we investigate the combinatorics of the Newton polytope of the dimer face polynomial $D_G$. We show the Newton polytope is affinely isomorphic to the perfect matching polytope of $G$. From this, we determine the face lattice of the Newton polytope using results of \cite{PSW} and show that all lattice points are vertices.

\begin{definition} If $p$ is a polynomial in a polynomial ring whose variables are indexed by some set $I$,
the \textbf{support} of $p$ is the integer point set in $\mathbb N^I$ consisting of the exponent vectors of monomials with nonzero coefficient in~$p$. 
The \textbf{Newton polytope} $\Newton(p)\subseteq\mathbb R^I$ is the convex hull of the support of~$p$. 
\end{definition}

In particular, if $G$ is a graph with property $(*)$ and $D_G$ its dimer face polynomial, then 
\[\Newton(D_G)= {\rm ConvHull}(\height(M) \colon M \text{ is a perfect matching of }G) \subset \mathbb{R}^{\faces(G)}.\]

Note that $\Newton(D_G) \subset \mathbb{R}^{\faces(G)}$ is a full dimensional polytope, as can be concluded based on its definition and Corollary \ref{cor:prop*=all-faces-flipped}. 

\medskip

\begin{definition} The \textbf{perfect matching polytope} $PM(G)$ of a bipartite graph $G$ is $$PM(G)={\rm ConvHull}(\chi_M\mid M \text{ is a perfect matching of } G),$$ where $\chi_M \in \mathbb{R}^{E(G)}$ is the indicator vector of the perfect matching $M$. 
\end{definition}

Note that  $PM(G)$ is a $0/1$-polytope. As such, the integer points $PM(G) \cap \mathbb{Z}^{E(G)}$ of $PM(G)$ and the vertex set of $PM(G)$ are both exactly the set of the incidence vectors $\chi_M$.

We will show that $\Newton(D_G)$ and $PM(G)$ are affinely isomorphic.

\begin{definition} \label{def:aff-isom} Let $P\subset \mathbb{R}^d$ and $Q \subset \mathbb{R}^l$ be polytopes. Fix $A\in \mathbb{R}^{l\times d}$, ${\bf x}\in \mathbb{R}^{d}$ and ${\bf v}\in \mathbb{R}^{l}$. The affine map $\pi:\mathbb{R}^d\rightarrow \mathbb{R}^l$ defined by $$\pi({\bf x})=A{\bf x}+{\bf v}$$ is an \textbf{affine isomorphism} between $P$ and $Q$ if $\pi|_P$ is a bijection whose image is $Q$.
\end{definition}

Assume that $G$ has property $(*)$. We define the affine transformation $\psi$ between ${\Newton}(D_G) \subset \mathbb{R}^{\faces(G)}$ and $PM(G)\subset \mathbb{R}^{E(G)}$ as follows. 
\medskip

To each non-infinite face $f \in \faces(G)$, we associate a vector ${\bf v}_f \in \mathbb{R}^{E(G)}$ as follows. For an edge $e \in E(G)$, we set
\[(\mathbf{v}_f)_e= \begin{cases}
    1 & \text{if } e\text{ a black-white edge of }f\\
    -1 & \text{if } e\text{ a white-black edge of }f\\
    0 & \text{if }e \text{ is not in the boundary of }f.
\end{cases}\]

Let $A_G \in \mathbb{R}^{E(G)} \times \mathbb{R}^{\faces(G)}$ be the matrix in whose column vectors are ${\bf v}_{f}$. Moreover, recall that  $\chi_{\hat{0}} \in \{0, 1\}^{E(G)}$ is the indicator vector of the minimal element of $\mathcal{D}_G$.

\begin{proposition} \label{prop:image} Suppose $G$ has property $(*)$. Let $\psi: \mathbb{R}^{\faces(G)}\rightarrow \mathbb{R}^{E(G)}$ be the affine map defined by $\psi({\bf x})=A_G {\bf x} +\chi_{\hat{0}}.$ We have that $$\psi(\Newton(D_G))=PM(G).$$
\end{proposition}

\proof
Recall that by \cref{cor:ht}, $\Newton(D_G)$ is the convex hull of $\{\height(M)\}_{M \in \mathcal{D}_G}$. In particular, each vertex of $\Newton(D_G)$ is a height vector of some matching.

We will show that $\psi$ maps the height vector $\height(M)$ to the indicator vector $\chi_M$ and in particular is a bijection between $\{\height(M)\}_{M \in \mathcal{D}_G}$ and $\{\chi_M\}_{M \in \mathcal{M}_G}$. We then use this to show that 
$\psi(\Newton(D_G))=PM(G).$ 

To show that $\psi$ is a bijection between $\{\height(M)\}_{M \in \mathcal{D}_G}$ and $\{\chi_M\}_{M \in \mathcal{D}_G}$, we proceed by induction on rank in $\mathcal{D}_G$.  

\medskip

\textit{Base case:} the unique rank $0$ perfect matching in $\mathcal{D}_G$ corresponds to the integer point $(0,\ldots, 0)=\height(\hat{0})\in \Newton(D_G)$, whereas it corresponds to $\chi_{\hat{0}}\in PM(G)$. Note that $$\psi((0,\ldots, 0))=\chi_{\hat{0}}$$ as desired.

\medskip

\textit{Inductive hypothesis:} Given a matching $M \in \mathcal{D}_G$ of rank $k$ for some fixed $k \geq 0$, the height $\height(M)$ is mapped to $\chi_M$ under $\psi$. That is: 
\begin{equation} \label{ind-hyp}\psi(\height(M)))=A_G (\height(M)))+\chi_{\hat{0}}=
\chi_M.
\end{equation}

\textit{Inductive step:} Consider a matching $M'$ of rank $k+1$ in  $\mathcal{D}_G$. Let $M$ be a matching of rank $k$ that $M'$ covers and let $f$ be the face of $G$ at which we perform a down-flip to get from $M'$ to $M$. 

 By \cref{lem:ht_flip}, 
\begin{equation} \label{exp} \height(M')=\height(M)+{\bf e}_f\end{equation} where ${\bf e}_f \in \mathbb{R}^{\faces{(G)}}$ is the coordinate vector with $1$ in the coordinate indexed by $f$ and $0$s elsewhere. 

By definition, \begin{equation} \label{chi} \chi_{M'}=\chi_M+{\bf v}_f.\end{equation}

Then, putting equations \eqref{ind-hyp}, \eqref{exp}, \eqref{chi} together, we obtain:
\begin{align*}
    \psi(\height(M'))&=\psi(\height(M)+{\bf e}_f)\\&=A_G(\height(M)+{\bf e}_f)+\chi_{\hat{0}} \\&=(A_G(\height(M))+\chi_{\hat{0}})+A_G({\bf e}_f)\\
   & =\chi_M+{\bf v}_f=\chi_{M'}
\end{align*}
as desired.

Now, since $\{\chi_M\}_{M \in \mathcal{D}_G}$ is the set of vertices of $PM(G)$, $\psi$ maps $\{\height(M)\}_{M \in \mathcal{D}_G}$ to $\{\chi_M\}_{M \in \mathcal{D}_G}$, and affine maps preserve convex combinations, the set of vertices of $\Newton(D_G)$ is $\{\height(M)\}_{M \in \mathcal{D}_G}$. Thus $\psi$ is a bijection between the vertices of $\Newton(D_G)$ and $PM(G)$. This implies that 
$\psi(\Newton(D_G))=PM(G).$
\qed

\begin{remark}\label{rem:dimer-polyomial-edges} 
    There is another polynomial associated to the dimers of a planar bipartite graph $G$, called the \emph{partition function} $Z_G$. The partition function is in ``edge variables" $\{z_e\}_{e \in E(G)}$ rather than ``face variables" $\{y_f\}_{f \in \faces(G)}$, and is defined as 
    \[Z_G:=\sum_{M \in \mathcal{D}_G} \prod_{e \in M} z_e.\]
    The polytope $PM(G)$ is the Newton polytope $\Newton(Z_G)$. The above proposition shows that to go from the dimer face polynomial $D_G$ to the partition function $Z_G$, one substitutes
    \begin{equation}\label{eq:face-to-edge-substitution}
    y_f \mapsto \prod_{\substack{e \in f\\ \text{black-white}}} z_e\prod_{\substack{e \in f\\ \text{white-black}}} z_e^{-1},\end{equation}
    and then multiplies by $\prod_{e \in \hat{0}} z_e$. Note that the alternating product of edge weights as in \eqref{eq:face-to-edge-substitution} is important in the dimer theory literature, cf. \cite[Sec. 3.2]{Dimers}.
\end{remark}

\begin{theorem} \label{thm:affine} The affine map 
\begin{align*}
    \psi: \mathbb{R}^{\faces(G)}&\rightarrow \mathbb{R}^{E(G)}\\
    \mathbf{x} &\mapsto A_G {\bf x} +\chi_{\hat{0}}
\end{align*}
is an affine isomorphism of the polytopes $\Newton(D_G)$ and $PM(G)$.
\end{theorem}

\proof We already know from \cref{prop:image} that $\psi(\Newton(D_G))=PM(G)$. So by Definition \ref{def:aff-isom}, it suffices to show that $A_G$ is injective. Since $|E(G)|> |\faces(G)|$, we will prove this by showing that the columns of $A_G$, the vectors $\{{\bf v}_f\}_{f \in \faces(G)}$, are linearly independent. 

Let $$\sum_{f \in \faces(G)} c_f {\bf v}_f={\bf 0},$$ where $c_f \in \mathbb{R}$ for $f \in \faces(G).$ We will show that $c_f=0$ for all $f \in \faces(G)$.

Note that any edge $e\in E(G)$ in the boundary of the infinite face, there is exactly one face $f \in \faces(G)$ that contains $e$, and in particular, exactly one vector among $\{{\bf v}_f\}_{f \in \faces(G)}$ with a nonzero coordinate in position $e$. We thus conclude that for all faces $f \in \faces(G)$ neighboring the infinite face of $G$, we have $c_f=0$. With similar reasoning, we can conclude that for all faces $f \in \faces(G)$ that neighbor any of the faces neighboring the infinite face of $G$, we have $c_f=0$. Continuing this way, we obtain that $c_f=0$ for all $f \in \faces(G)$, as desired. \qed

Recall that affine isomorphism of polytopes implies combinatorial isomorphism; that is, affinely isomorphic polytopes have isomorphic face lattices. We will use this to describe the face lattice of $\Newton(D_G)$. We first need the following definition.

\begin{definition}
    A subgraph $H$ of $G$ is \emph{elementary} if it contains every vertex of $G$ and every edge of $H$ is used in some perfect matching. Equivalently, $H$ is a union of several perfect matchings of $G$. 
\end{definition}

The next proposition follows directly from \cite[Theorem 7.3]{PSW}, which establishes the face lattice of $PM(G)$ once conventions are translated appropriately.

\begin{proposition}\label{prop:face-lattice-newton}
    Suppose $G$ has property $(*)$. The face lattice of $\Newton(D_G)$ is isomorphic to the lattice of all elementary subgraphs of $G$, ordered by inclusion.
\end{proposition}

\begin{proof}
    Since $\Newton(D_G)$ and $PM(G)$ are affinely isomorphic, their face lattices are isomorphic.
    It follows from \cite[Theorem 7.3]{PSW} that the face lattice of $PM(G)$ is isomorphic to the lattice of elementary subgraphs of $G$, ordered by inclusion. In more detail, in \cite{PSW}, they deal with planar graphs embedded in a disk, perhaps with some vertices on the boundary. So to apply \cite[Theorem 7.3]{PSW}, one should draw $G$ in a disk to obtain the graph $G'$.
    The graph $G'$ satisfies \cite[Definition 2.2]{PSW} (with $n=0$), a perfect matching of $G$ is an \emph{almost-perfect matching} of $G'$, and $PM(G)$ is equal to the polytope $P(G')$ from \cite[Definition 4.1]{PSW}.
\end{proof}

We can also determine the integer points of $\Newton(D_G)$.

\begin{corollary}\label{cor:newton-lattice-points}
    Suppose $G$ has property $(*)$. Then every integer point of $\Newton(D_G)$ is a vertex.
\end{corollary}

\proof 
By Theorem \ref{thm:affine}, the map $\psi$ is an affine isomorphism from $\Newton(D_G)$ to $PM(G)$. Observe that by definition, the map $\psi$ sends integer points to integer points, so the integer points $\mathbf{x}$ of $\Newton(D_G)$ inject into the integer points of $PM(G)$. The polytope $PM(G)$ is a $0,1$-polytope with all integer points being vertices. As $\psi$ sends vertices bijectively to vertices, this implies that in the Newton polytope of $D_G$, every integer point is a vertex. 
\qed

\medskip

A special case of this corollary appeared in \cite[Theorem 1.10]{BMS24}, when $G=G_{L,i}$ is a truncated face-crossing incidence graph of a prime link diagram $L$ with no curls (see \cref{def:truncated}), and $D_G=F_{T(i)}$ (see \cref{prop:BMS-poly-and-dimer-poly}). 

\section{On dimers and links}\label{sec:alexander-poly-and-dimers}

In this section, we describe connections between knot theory and dimers. We first recall Kauffman's lattice of states and how it can be used to compute the Alexander polynomial of a link. We describe how to associate a bipartite plane graph $G_L$ to each (oriented) link diagram $L$. We show that Kauffmann's lattice of states is isomorphic to the dimer lattice of a related graph $G_{L,i}$ (see \cref{lem:state-dimer-bijection}). The latter isomorphism is also stated in the work of Cohen and Teicher \cite{ClockLattice}, and this section is in part inspired by their work. We conclude this section by a proof of Theorem \ref{thm:intro-Alex-poly}, which show that the Alexander polynomial of $L$ can be obtained by specializing the dimer face polynomial $D_{G_{L,i}}$.

\subsection{Kauffman's clock lattice and Alexander polynomials}
In this section, we review the \emph{state summation} formula for the Alexander polynomial of an oriented link, given by Kauffman \cite{K06}. We assume some familiarity with knot theory; see e.g. \cite{burde2002knots} for details.

If $L$ is an oriented link diagram, we write $\Delta_L(t) \in \mathbb{Z}[t,t^{-1}]$ for its Alexander polynomial. The Alexander polynomial is defined up to multiplication by a signed power of $t$. We use the notation $P \sim Q$ to indicate that $P,Q \in \mathbb{Z}[t,t^{-1}]$ are related by multiplication by $\pm t^k$ for some integer $k$.

Fix an oriented link diagram $L$ and a distinguished \textbf{segment} $i$ of $L$ (that is, a curve between two crossings). The two planar regions adjacent to $i$ are called \textbf{absent}; the remaining regions are \textbf{present}. Weight the present regions around each crossing as in \cref{fig:weighting}, left.

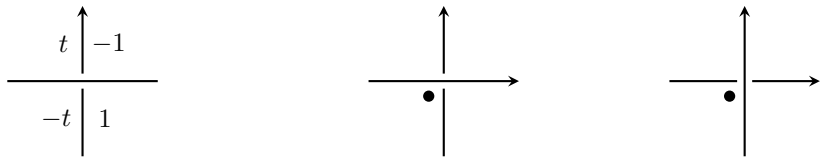
\begin{figure}[h]
\centering
\begin{tikzpicture}
\draw[black, thick] (5,0) -- (5,.9);
\draw[-stealth, black, thick] (5,1.1) -- (5,2);
\draw[black, thick] (4,1) -- (6,1);
\node at (4.75,1.5) {$t$};
\node at (4.65,.5) {$-t$};
\node at (5.35,1.5) {$-1$};
\node at (5.3,.5) {$1$};
\end{tikzpicture}
\hspace{1in}
\begin{tikzpicture}
\draw[black, thick] (1,0) -- (1,.9);
\draw[-stealth, black, thick] (1,1.1) -- (1,2);
\draw[-stealth, black, thick] (0,1) -- (2,1);
\draw[black, thick] (4,1) -- (4.9,1);
\draw[-stealth, black, thick] (5.1,1) -- (6,1);
\draw[-stealth, black, thick] (5,0) -- (5,2);
\node at (.8,.8) [circle, fill, inner sep=1.5pt]{};
\node at (4.8,.8) [circle, fill, inner sep=1.5pt]{};
\end{tikzpicture}
\caption{Left: the weighting of present regions around each crossing in the definition of the state sum. Note that the orientation of the horizontal strand does not matter. Center and right: black holes.}
\label{fig:weighting}
\end{figure}

A \textbf{state} is a bijection between crossings in the diagram and the present regions so that each crossing is mapped to one of the four regions it meets. This bijection is represented diagrammatically by adding a \textbf{state marker} to the region each corner is mapped to; see \cref{fig:clock_lattice} for examples. We denote by $\mathcal{S}_{L,i}$ the set of states of the link diagram $L$ with distinguished segment $i$. The \textbf{weight} of a state $S$, denoted $\langle L|S\rangle$, is the product of the weights corresponding to the marked regions at each crossing. See \cref{fig:clock_lattice} for examples.

A state marker is called a \textbf{black hole} if it is in the configuration depicted in \cref{fig:weighting}, center or right. Let $b(S)$ be the number of black holes in a fixed state $S$. 

\begin{theorem}[\cite{K06} pg. 176]\label{thm:state-sum-alexander-poly}
    Let $L$ be an oriented link diagram and $i$ a segment. Then, \[\Delta_L(t)\sim\sum_{S \in \mathcal{S}_{L,i}} (-1)^{b(S)}\langle L|S\rangle.\]
\end{theorem} 

We now define the clock lattice.

\begin{definition}
    A \textbf{clock move} on a state is the following transposition of state markers at two adjacent crossings (regardless of the orientation of strands and the crossing information).
    \begin{center}
        \includegraphics[height=1cm]{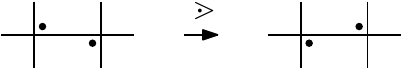}
    \end{center}
\end{definition}
Note that the state markers move \emph{clockwise} in the clock move, hence the name.

\begin{theorem}[{\cite[Theorem 2.5] {K06}}]\label{tmh:clockthm}
Define $S\leq S'$ whenever $S$ is obtained from $S'$ by a sequence of clock moves. The set of states $\mathcal{S}_{L,i}$ equipped with the binary relation $\leq$ is a lattice.
\end{theorem}

We call the lattice of \cref{tmh:clockthm} the \textbf{clock lattice of the link diagram $L$}, and in an abuse of notation denote it by $\mathcal{S}_{L,i}$. Figure \ref{fig:clock_lattice} depicts an example of $\mathcal{S}_{L,i}$.

\begin{figure}
    \centering
    \includegraphics[height=13cm]{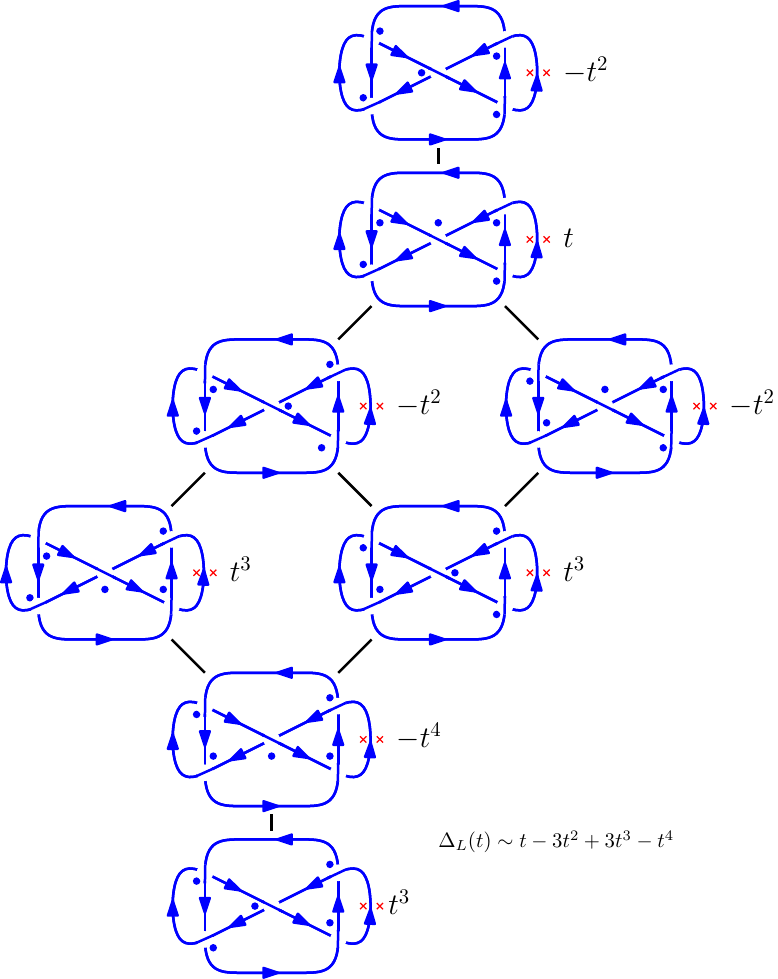}
    \caption{The clock lattice for the Whitehead link, where the distinguished segment is the one adjacent to the red crosses. To the right of each state is its weight. }
    \label{fig:clock_lattice}
\end{figure}

\subsection{Bipartite plane graphs for link diagrams}

As in \cite{TwistedDimer,ClockLattice}, we associate a bipartite plane graph $G_L$ to each link diagram\footnote{In fact, $G_L$ does not depend on the crossing information of the link diagram.} $L$. In other references the graph $G_L$ is also referred to as  the \textit{overlaid Tait graph} of a link.

\begin{definition}
    Let $L$ be a link diagram. The \textbf{face-crossing incidence graph}, $G_L$, is a plane graph defined as follows. Place a black vertex $b_c$ on each crossing $c$ of $L$ and a white vertex $w_r$ in each of its planar regions $r$. There is an edge $(b_c, w_r)$ in $G_L$ if and only if crossing $c$ touches region $r$. The faces of $G_L$ are in bijection with the segments of $L$, and in particular, each segment is contained in a unique face. We use $f_j$ to denote the face of $G_L$ containing segment $j$.
\end{definition}  

\begin{figure}
    \centering
    \includegraphics[height=4cm]{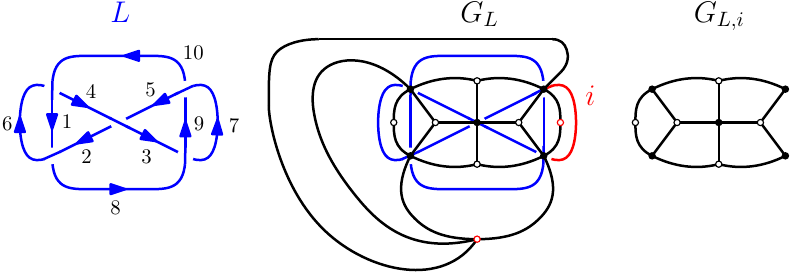}
    \caption{Left: In blue, a link diagram for the Whitehead link. Center: The face-crossing incidence graph $G_L$, with $f_7$ as the infinite face. Right: The truncated face-crossing incidence graph $G_{L,i}$ for $i=7$.}
    \label{fig:whitehead_graph}
\end{figure}

Note that any link diagram $L$ and graph $G_L$ may be viewed on the sphere by compactifying $\mathbb{R}^2$. For any fixed segment $i$ of $L$, stereographically projecting through any point in a region adjacent to $i$ yields a drawing of $L$ in the plane for which $i$ bounds the exterior region, and a drawing of $G_L$ in which $f_i$ is in the infinite face. See Figure \ref{fig:whitehead_graph}.

The graph $G_L$ has two more white vertices than black vertices. In order to make connections to dimer lattices, we will delete white vertices of $G_L$ corresponding to the choice of a segment in the definition of Kauffman's state lattice, see Definition \ref{def:truncated}. Such a construction also appears under the name of \textit{balanced overlaid Tait graph} in other references. 

\begin{definition} \label{def:truncated}
    Let $L$ be a link diagram, and fix a segment $i$ of $L$. Changing the steroegraphic projection of $L$ if necessary, we may assume that $f_i$ is the infinite face of $G_L$. The \textbf{truncated face-crossing incidence graph} is the plane graph $G_{L,i}$ obtained by deleting the two white vertices of $f_i$. 
\end{definition}

We will impose some constraints on the link diagram $L$ to ensure that the graph $G_{L,i}$ has property $(*)$. The   ``prime-like" condition for links introduced in  \cite{ClockLattice} turns out to be just what we need.
 \begin{definition} A link diagram $L$ is
 \begin{itemize}
     \item \textbf{connected} if it has a single connected component;
     \item \textbf{prime-like} if there does not exist a simple closed curve $\gamma$ which intersects $L$ in exactly $2$ points such that $L$ has crossings both inside and outside of $\gamma$, or equivalently, if $L$ is not the connect sum of two link diagrams that both have crossings (see \cref{fig:connect_sum}, left).
 \end{itemize}
 A crossing $c$ of $L$ is \textbf{nugatory} if there exists a simple closed curve $\gamma$ which intersects $L$ only at $c$.
\end{definition}

The following theorem gives a sufficient condition for $G_{L,i}$ to have property $(*)$. It is proved in \cite[Theorems 4.6 and 4.7]{ClockLattice} for prime-like knot diagrams with no nugatory crossings. A careful reading shows that the proof there also holds if ``knot diagram" is replaced with ``connected link diagram." We state the version for link diagrams below.

\begin{theorem}[{cf. \cite[Theorems 4.6 and 4.7]{ClockLattice}}]
Let $L$ be a connected prime-like link diagram without nugatory crossings. Then for any segment $i$, the plane graph $G_{L,i}$ satisfies property $(*)$.
\end{theorem}

\begin{assumption}\label{assump:connected-prime-like-no-nugatory}
    We assume throughout the remainder of the paper that $L$ is a connected prime-like link diagram with no nugatory crossings. Every link has a connected prime-like diagram without nugatory crossings. Indeed, given any diagram for the link, one can always uncross any nugatory crossings. Then one can repeatedly apply the reverse direction of the second Reidemeister move to make the diagram connected and prime-like (see Figure \ref{fig:prime_like}). 
\end{assumption}

\begin{figure}
    \centering
    \includegraphics[width=0.5\textwidth]{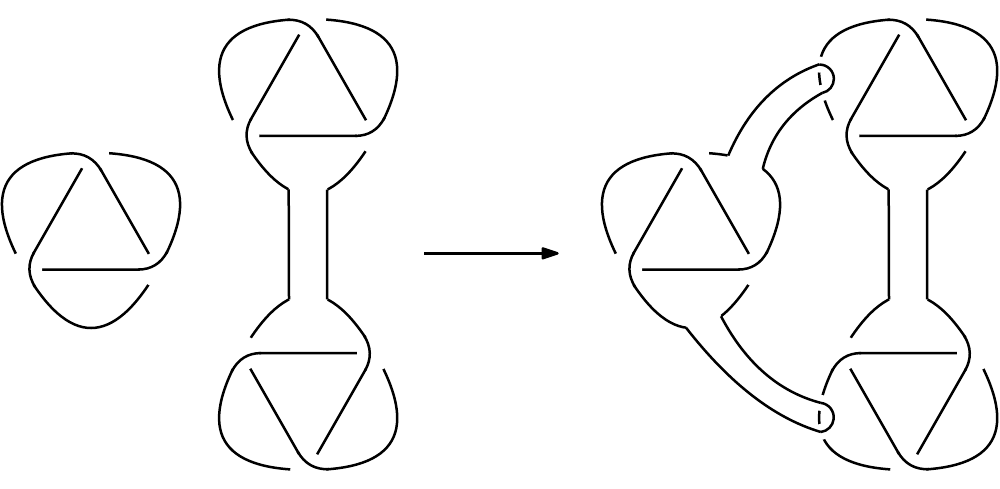}
    \caption{To the left, a non-prime-like and disconnected link diagram, and to the right, an isotopic prime-like and connected diagram.}
    \label{fig:prime_like}
\end{figure}

Recall that, for $L$ a link diagram and $i$ a segment of $L$, a state $S$ in Kauffman's model can be viewed as a pairing of the present regions of $L$ and the crossings of the diagram, where each crossing is paired with a region adjacent to it. Using this, we obtain the following lattice isomorphism between the state lattice of $L$ with respect to $i$ and the dimer lattice of $G_{L,i}$. This isomorphism is also stated in \cite{ClockLattice}.

\begin{lemma} \label{lem:state-dimer-bijection}
Let $L$ be a link diagram and let $i$ be a segment of $L$. Fix an embedding of $L$ for which $i$ is on the boundary of the exterior region. Let $S\in\mathcal{S}_{L,i}$, and define the dimer $D_S\in\mathcal{D}_{G_{L,i}}$ by $e=(b_c,w_r)\in D_S$ if the state marker near crossing $c$ in $S$ is in region $r$. The map 
\begin{align*}
    \mathcal{S}_{L,i}&\rightarrow\mathcal{D}_{L,i}\\
    S &\mapsto D_S
\end{align*}
is a lattice isomorphism.
\end{lemma}

\begin{proof}
    The map $S\mapsto D_S$ is a bijection between states in $\mathcal{S}_{L,i}$ and dimers in $\mathcal{D}_{G_{L,i}}$, and it remains to show that this map respects cover relations in both lattices. Clock moves relate to flips as shown in \cref{fig:clock-vs-flip}. Because a clock move is never performed on the segment $i$, each clock move corresponds to a flip of an internal face of $G_{L,i}$. 
 \end{proof}
    \begin{figure}[h]
    \centering
        \includegraphics[width=0.4\textwidth]{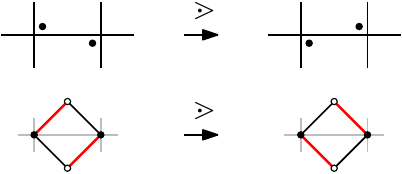}
        \caption{The relationship between clock moves on $L$ and flips on $G_L$.}
        \label{fig:clock-vs-flip}
    \end{figure}

\begin{theorem}\label{thm:alexdimer} 
    Let $L$ be a link diagram, and fix a segment $i$ of $L$. Let $\Delta_L(t)$ denote the Alexander polynomial of $L$. Let $D_{G_{L,i}} (t)$ be obtained from $D_{G_{L,i}}(\mathbf{y})$ by specializing
    \begin{equation}\label{eq:alex-specialization}y_j \mapsto \begin{cases}
       -t & \text{if segment } j \text{ exits an undercrossing and enters an overcrossing,} \\
    -t^{-1} & \text{if segment } j \text{ exits an overcrossing and enters an undercrossing, and} \\
    -1 & \text{otherwise.} \\ 
    \end{cases}\end{equation}
    Then we have
     \begin{equation}\label{eq:alexdimer}
     D_{G_{L,i}}(t)=\left( (-1)^{b(\hat{0})} \langle L|\hat{0}\rangle^{-1} \right) \sum_{S \in \mathcal{S}_{L,i}} (-1)^{b(S)} \langle L|S\rangle
     \end{equation}
     and in particular, $$D_{G_{L,i}}(t)\sim \Delta_L(t).$$ 
   \end{theorem}

Figure \ref{fig:whitehead_dimer} depicts an example of this calculation. Motivated by \cref{thm:alexdimer}, we sometimes call $D_{G_{L,i}}({\bf y})$ a \textbf{multivariate Alexander polynomial} for $L$.

\begin{figure}
    \centering
    \includegraphics[height=13cm]{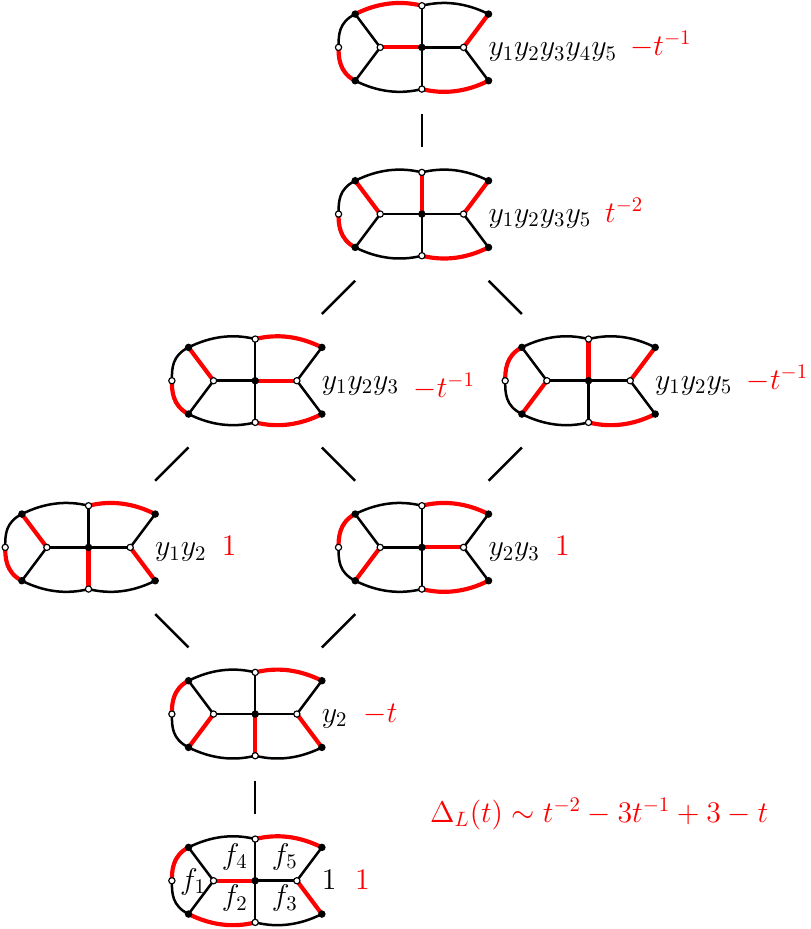}
    \caption{The Alexander polynomial of the Whitehead link in Figure \ref{fig:whitehead_graph}, computed as a weighted sum of dimers. The original monomials are in black and their specializations are in red. Notice that the result is the same as that in Figure \ref{fig:clock_lattice} up to $\sim$.}
    \label{fig:whitehead_dimer}
\end{figure}

\begin{proof}
In both sides of \eqref{eq:alexdimer}, the term indexed by $\hat{0}$ is 1. So to show the equality in \eqref{eq:alexdimer}, it suffices to show that if $S$ is a Kauffman state obtained from state $S'$ by performing a clock move on segment $j$, then
    $$\frac{(-1)^{b(S)}\langle L | S\rangle}{(-1)^{b(S')}\langle L | S'\rangle} = \begin{cases}
    -t & \text{if segment } j \text{ exits an undercrossing and enters an overcrossing,} \\
    -t^{-1} & \text{if segment } j \text{ exits an overcrossing and enters an undercrossing, and} \\
    -1 & \text{otherwise.} \\
\end{cases}$$

In all cases, performing a clock move causes $S$ to either lose or gain a black hole. Suppose first that segment $j$ exits an overcrossing and enters an undercrossing. Then, the Kauffman weights incident to the crossings are the following. Below, $\alpha$ is $\pm 1$ or $\pm t$.

\begin{center}
\includegraphics[width=0.4\textwidth]{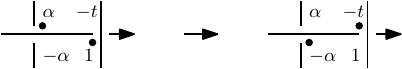}
\end{center}

The clock move increases the degree of $\langle L | S' \rangle$ by one. Additionally, $S'$ either loses or gains a black hole, meaning the clock move changes the sign of $(-1)^{b(S')}\langle L | S' \rangle$. And so, $(-1)^{b(S)}\langle L | S \rangle = (-t)(-1)^{b(S')}\langle L | S' \rangle$.

Suppose that segment $j$ exits an undercrossing and enters an overcrossing. We have the following.

\begin{center}
\includegraphics[width=0.4\textwidth]{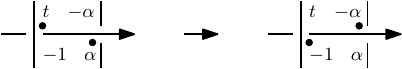}
\end{center}

So, the clock move decreases the degree of $\langle L | S' \rangle$ by one. As a result, $(-1)^{b(S)}\langle L | S\rangle = (-t^{-1})(-1)^{b(S')}\langle L | S' \rangle$. Finally, if $i$ either exits and enters an undercrossing or either exits and enters an overcrossing, performing a clock move leaves $\langle L | S' \rangle$ unchanged. Hence, $(-1)^{b(S)}\langle L | S\rangle = (-1)(-1)^{b(S')}\langle L | S' \rangle$. The two cases are depicted below.

\begin{center}\begin{tabular}{l|r}
  \includegraphics[width=0.4\textwidth]{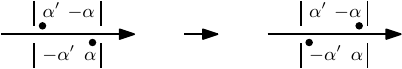} \hspace{12pt}   & \hspace{12pt}  \includegraphics[width=0.4\textwidth]{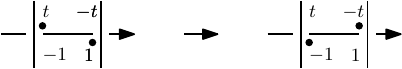} \\
\end{tabular}
\end{center}
Finally, the fact that $D_{G_{L_i}}(t) \sim \Delta_L(t)$ is immediate from \eqref{eq:alexdimer} and \cref{thm:state-sum-alexander-poly}, since $(-1)^{b(\hat{0})} \langle L | \hat{0} \rangle ^{-1}$ is equal to $\pm t^k$ for some $k$.
\end{proof}

\subsection{Connections to partition function specializations}\label{sec:connection-twisted-dimers}
    A formula for the Alexander polynomial of $L$ as a weighted sum over dimers of $G_{L,i}$ was given in \cite[Proposition 3.4]{TwistedDimer}. The authors of that work used edges rather than faces, and so wrote the Alexander polynomial as a specialization of the partition function of $G_{L,i}$ (see \cref{rem:dimer-polyomial-edges}) rather than as a specialization of the dimer face polynomial. The proof of \cite[Proposition 3.4]{TwistedDimer} utilized Kasteleyn matrices, so is rather different from the proof of \cref{thm:alexdimer}. Here, we briefly discuss the connection between their specialization of the partition function and our specialization of the dimer face polynomial in \cref{thm:alexdimer}. 

    In Cohen's work, at each crossing, for each present region incident to the crossing, one assigns two weights: a \textit{local weight} and a specific \textit{Kastelyn weight} informed by Kauffman's state summation model. 

    \begin{definition}[{cf. \cite[Algorithm 3.2, (D2) and (D3)]{TwistedDimer}}]
    Let $L$ be a link diagram and let $i$ be a segment of $L$. For $e=(b_c,w_r)\in\Edg(G_{L,i})$, define $\alpha(e)$ to be the product of the local weight and Kastelyn weight region $r$ receives near crossing $c$ as in Figure \ref{fig:cohen_weighting}. 
    \end{definition}

    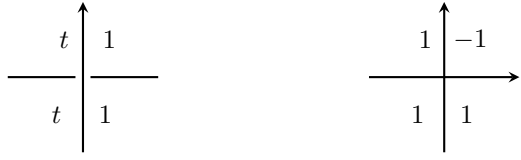
\begin{figure}[h]
    \centering
    \begin{tikzpicture}
    \draw[-stealth, black, thick] (5,0) -- (5,2);
    \draw[black, thick] (4,1) -- (4.9,1);
    \draw[black, thick] (5.1,1) -- (6,1);
    \node at (4.75,1.5) {$t$};
    \node at (4.65,.5) {$t$};
    \node at (5.35,1.5) {$1$};
    \node at (5.3,.5) {$1$};
    \end{tikzpicture}
    \hspace{1in}
    \begin{tikzpicture}
    \draw[black, thick] (1,0) -- (1,1);
    \draw[-stealth, black, thick] (1,1) -- (1,2);
    \draw[-stealth, black, thick] (0,1) -- (2,1);
    \node at (.75,1.5) {$1$};
    \node at (.65,.5) {$1$};
    \node at (1.35,1.5) {$-1$};
    \node at (1.3,.5) {$1$};
    \end{tikzpicture}
    \caption{Left: the local weighting of present regions around each crossing in \cite{TwistedDimer}. 
    As before, the orientation of the horizontal strand does not matter. Right: the Kasteleyn weighting from \cite{TwistedDimer}. Note the under- and overcrossing information does not change the weighting.}
    \label{fig:cohen_weighting}
    \end{figure}

The next proposition explains the relationship between the specialization $z_e \mapsto \alpha(e)$ and the specialization of \eqref{eq:alex-specialization}. Essentially, they intertwine with the change of variables from the dimer face polynomial to the partition function.

    \begin{proposition}
    Let $L$ be a link diagram, $\bar{L}$ the diagram obtained from $L$ by switching overcrossings and undercrossings, and let $i$ be a segment of $L$. Throughout, we use $\alpha(e)$ to denote the weight of $e$ an edge of $G_{\bar{L}, i}$, that is, using the crossing information of $\bar{L}$.
    Let $M$ be a matching of $G_{L,i}=G_{\bar{L}, i}$, $S$ the corresponding state (cf. Lemma \ref{lem:state-dimer-bijection}), and let $\mathbf{z}_M$ denote the corresponding term of the partition function $Z_{G_{\bar{L},i}}$. Let $p:= (-1)^{b(\hat{0})} \langle {L} | \hat{0}\rangle$. Then $\mathbf{z}_{\hat{0}}|_{z_e \mapsto \alpha(e)} = \sigma \cdot p$ where $\sigma \in \{\pm 1\}.$ Further,
    we have
    \[\begin{tikzcd}
    \mathbf{y}^{\height(M)} \arrow{rrr}{\text{\eqref{eq:face-to-edge-substitution} and multiply by } \mathbf{z}_{\hat{0}}} \arrow[d,"\text{\eqref{eq:alex-specialization}}"'] & && \mathbf{z}_M \arrow{d}{z_e \mapsto \alpha(e)}\\
    (-1)^{b(S)} \langle {L}|S\rangle \cdot p^{-1}\arrow[rrr,"\text{multiply by }\sigma\cdot p"] & & & \sigma \cdot (-1)^{b(S)} \langle {L}|S\rangle
    \end{tikzcd}
    \]
where in the specialization \eqref{eq:alex-specialization} we use the crossing information from ${L}$. 
  \end{proposition}

    \begin{proof} 
    Comparing \cref{fig:cohen_weighting} applied to $\bar{L}$ with \cref{fig:weighting} applied to $L$, we see that if we exchange over-crossings and under-crossings, the present regions around each crossing are weighted in the same way, up to sign. In \cref{fig:cohen_weighting}, these weightings give the specialization $z_e \mapsto \alpha(e)$. In \cref{fig:weighting}, these weightings are used to compute $(-1)^{b(S)}\langle L | S \rangle$. Using the bijection between matchings of $G_{\bar{L},i}= G_{L,i}$ and states of $L$ (cf. \cref{lem:state-dimer-bijection}), we have that $\mathbf{z}_{\hat{0}}|_{z_e \mapsto \alpha(e)} = \sigma \cdot p$ where $\sigma \in \{\pm 1\}.$
    
    We now turn to the commutative diagram. The top arrow is \cref{rem:dimer-polyomial-edges}, the left vertical arrow is \cref{thm:alexdimer}, and the bottom arrow is clear. The right arrow may be deduced from \cite[Algorithm 3.2, (D4) and Algorithm 3.6]{TwistedDimer}. 
    
    Alternately, one can check using case analysis that under the specialization $z_e \mapsto \alpha(e)$, we have
    \[\prod_{\substack{e \in f \\ \text{black-white}}} \alpha(e)\prod_{\substack{e \in f \\ \text{white-black}}} \alpha(e)^{-1} = \begin{cases}
       -t & \text{if segment } j \text{ exits an overcrossing and enters an undercrossing} \\
       & \text{in }\bar{L} \\
    -t^{-1} & \text{if segment } j \text{ exits an undercrossing and enters an overcrossing}\\
    & \text{in } \bar{L}, \text{ and} \\
    -1 & \text{otherwise.} \\ 
    \end{cases}\]
    This means that the change of variables in \eqref{eq:face-to-edge-substitution} followed by the specialization $z_e \mapsto \alpha(e)$ has exactly the same effect on $y_f$ as \eqref{eq:alex-specialization}. So the top arrow followed by the right arrow has the effect of performing \eqref{eq:alex-specialization} on $\mathbf{y}^{\height(M)}$ and then multiplying by the specialization of $\mathbf{z}_{\hat{0}}$, which is $\sigma\cdot p$. This is the same as the effect of the left arrow followed by the bottom arrow.
    \end{proof}

\subsection{Multivariate Alexander polynomials and submodule polynomials.}
 In the work \cite{B21}, for a curl-free diagram $L$ of a prime link and any fixed segment $i$ of $L$, the authors define a polynomial $F_{T(i)}(\mathbf{y})$ which specializes to the Alexander polynomial via the same specialization as in \cref{thm:alexdimer}. In this section, we explain the relationship between $F_{T(i)}$ and the dimer face polynomial $D_{G_{L,i}}$.

We briefly review the setup of \cite{B21}. Fix $L$ a curl-free diagram of a prime link. The authors define a quiver\footnote{which is $\extQ_{G_L}$ (see \cref{def:extended-quiver}) up to reversing all arrows and adding two-cycles} $Q_L$ whose vertices are the segments, and a potential $W$ on $Q_L$. For each choice of segment $i$, they define a module $T(i)$ over the Jacobean algebra associated to $(Q_L, W)$. The module $T(i)$ consists of a vector space for each segment of $L$ (equivalently each face of $G_{L,i}$, since segment $i$ is assigned the zero-dimensional vector space), as well as maps between these vector spaces. The polynomial $F_{T(i)}$ is the \emph{submodule polynomial}\footnote{In \cite{B21}, $F_{T(i)}$ is called the ``$F$-polynomial" of $T(i)$, because submodule polynomials are often $F$-polynomials in a related cluster algebra. In the recent work \cite{BMS24}, Bazier-Matte and Schiffler show that $F_{T(i)}$ is indeed an $F$-polynomial of a cluster algebra; this can also be deduced from \cref{thm:dimer-poly-is-F-poly} of this work.} of $T(i)$ and by \cite[Corollary 6.8 (b)]{B21} is equal to
\begin{equation*}F_{T(i)}=\sum_{M \subset T(i)} \mathbf{y}^{\text{dim} M}. \end{equation*}
The sum is over submodules $M$ of $T(i)$ and $\mathbf{y}^{\text{dim} M} \in \mathbb{Z}^{\faces(G_{L,i})}$ records the dimension of the vector space of $M$ sitting in each face of $G_{L,i}$.

By \cite[Theorem 1.2]{B21}, the submodule lattice of $T(i)$ is isomorphic to Kauffman's clock lattice $\mathcal{S}_{L,i}$, and thus by \cref{lem:state-dimer-bijection}, to the dimer lattice $\mathcal{D}_{G_{L,i}}$. This suggests that $F_{T(i)}$ is closely related to the dimer face polynomial $D_{G_{L,i}}$, as the next proposition verifies.

\begin{proposition}\label{prop:BMS-poly-and-dimer-poly}
	Let $L$ be a curl-free diagram of a prime link, and let $i$ be a segment of $L$. The polynomial $F_{T(i)}$ is equal to $D_{G_{L,i}}$. 
\end{proposition}

\begin{proof}
We prove equality term by term, going up the Kauffman lattice. The submodule of $T(i)$ corresponding to the clocked state $\hat{0}$ is the zero module, so the corresponding term of $F_{T(i)}$ is 1. This is the same as the term of $D_{G_{L,i}}$ corresponding to the clocked state.

    Let $S, S' \in \mathcal{S}_{L, i}$ be states in the Kauffman lattice and suppose $S$ is obtained from $S'$ by a clock move at segment $a$, so $S \lessdot S'$. Let $N, N'$ the corresponding submodules of $T(i)$, and $M, M'$ the corresponding matchings of $G_{L,i}$. Then by \cite[Lemma 6.4]{B21}, the module $N'$ is obtained from $N$ by increasing the dimension at segment $a$ by 1 (and changing the maps). That is, if segment $a$ is in face $f_a$ of $G_{L,i}$, then $\mathbf{y}^{\text{dim}N'}= y_{f_a} \mathbf{y}^{\text{dim}N}$. By \cref{lem:state-dimer-bijection}, we also have $\mrk(M')= y_{f_a} \mrk(M)$. So if $\mathbf{y}^{\text{dim}N} =\mrk(M) $, we also have $\mathbf{y}^{\text{dim}N'} =\mrk(M') $.
\end{proof}

\subsection{On multivariate Alexander polynomials of composite links}

In this section, we show factorization for (certain) multivariate Alexander polynomials of connect sums of link diagrams. This can be viewed as a generalization of the well-known fact that the Alexander polynomial of a composite link factors.

To form the connect sum of two link diagrams $L_1$ and $L_2$, one breaks segments $i_1$ of $L_1$ and $i_2$ of $L_2$ bounding the exterior region and adjoins them in a way which is consistent with orientation and does not introduce any additional crossings (see Figure \ref{fig:connect_sum}).

\begin{figure}
    \centering
    \includegraphics[height=9cm]{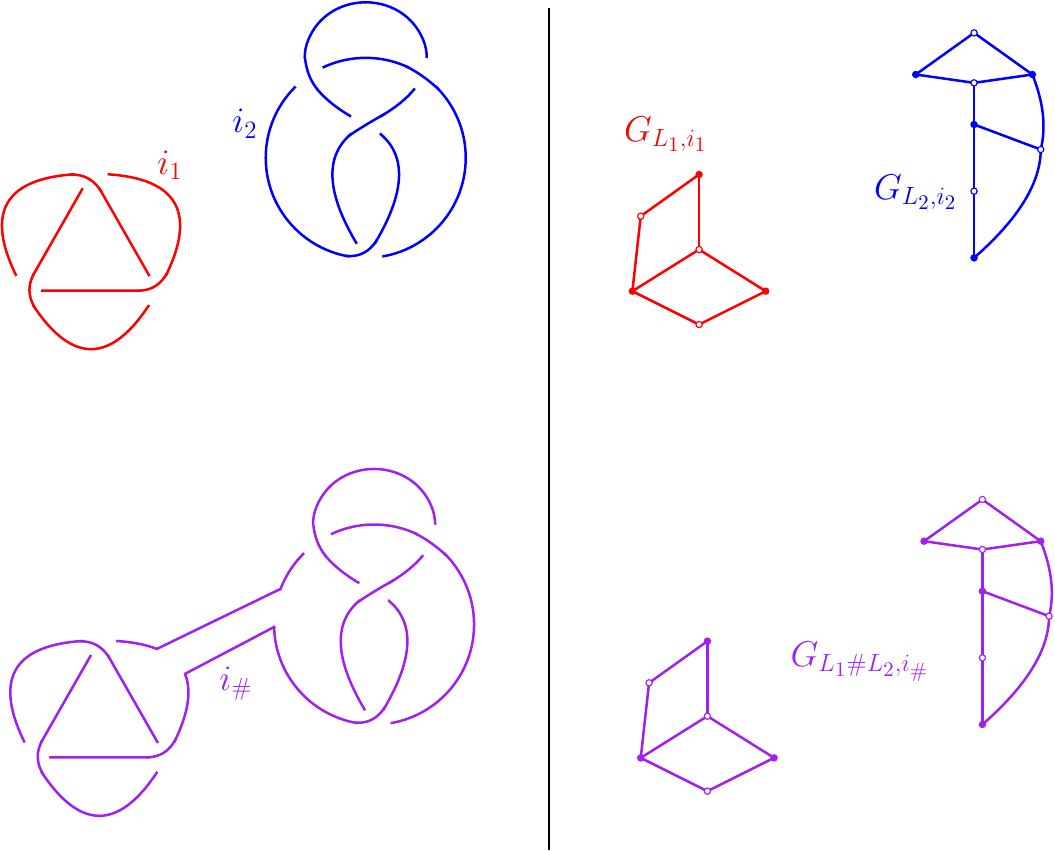}
    \caption{Left, top: link diagram $L_1$ in red and link diagram $L_2$ in blue. Left, bottom: the connect sum $L_1 \# L_2$ along segments $i_1$ and $i_2$. Right: the corresponding truncated face-crossing incidence graphs.}
    \label{fig:connect_sum}
\end{figure}

\begin{proposition} \label{prop:dimer-poly-connect-sum}
    Let $L_1$ and $L_2$ be link diagrams. Label the segments of $L_1$ by $1,\dots,m$, and label the segments of $L_2$ by $m+1,\dots, n$. Fix segments $i_1$ of $L_1$ and $i_2$ of $L_2$ bounding the exterior region. Choose one of the segments of $L_1 \# L_2$ formed by joining $i_1$ and $i_2$, and denote it $i_\#$. Then, 
    $$G_{L_1 \# L_2, i_\#} = G_{L_1, i_1}\sqcup G_{L_2, i_2}$$
    and 
    $$D_{G_{L_1\# L_2, i_\#}}(y_1,\dots, y_n) = D_{G_{L_1, i_1}}(y_1,\dots, y_m)\cdot D_{G_{L_2, i_2}}(y_{m+1},\dots, y_n).$$
\end{proposition}

\begin{proof} Let $R_1$ and $R_2$ denote the non-exterior regions of $L_1$ and $L_2$ bounded by $i_1$ and $i_2$, respectively. Let $R$ denote the region of $L_1\#L_2$ formed by joining $i_1$ and $i_2$.

The connect sum $L_1 \# L_2$ has the crossing-region adjacency relations of $L_1$ and $L_2$, except that any crossing incident to $R_1$ in $L_1$ is incident to $R$ in $L_1\# L_2$. Similarly, any crossing incident to $R_2$ in $L_2$ is incident to $R$ in $L_1\# L_2$. This shows the statement about $G_{L_1 \# L_2, i_\#}$.

The statement about the dimer face polynomials follows in a straightforward way from the statement about the graphs.

\end{proof}

\section{Background on cluster algebras}\label{sec:background-cluster}
In this section, we review skew-symmetric cluster algebras of geometric type, as well as their $F$-polynomials, $g$-vectors, and $d$-vectors. See e.g. \cite{FZ07}  for additional details. 

\begin{definition}[Quiver and quiver mutation]
A \textbf{quiver} is a directed graph $Q$ with no loops or directed 2-cycles. Edges of $Q$ are called \textbf{arrows}. Each vertex of $Q$ is declared either \textbf{mutable} or \textbf{frozen}. The \textbf{mutable part} of $Q$ is the induced subquiver of $Q$ on the mutable vertices. If $k$ is a mutable vertex of $Q$, \textbf{mutating $Q$ at $k$} produces a new quiver $\mu_k(Q)$, which is obtained from $Q$ by
      \begin{enumerate}
            \item adding an arrow $i\rightarrow j$ for every path $i\rightarrow k\rightarrow j$;
            \item reversing all arrows incident to $k$;
            \item deleting all $2$-cycles, one by one.
        \end{enumerate}
\end{definition}

\begin{definition}[Seeds]
    Let $\mathcal{F}$ be a field of rational functions in $n$ algebraically independent variables over $\mathbb{C}$. A \textbf{seed} of rank $r \leq n$ in $\mathcal{F}$ is a pair $\Sigma=(\mathbf{x}, Q)$ where $\mathbf{x}=(x_1, \dots, x_r, \dots, x_n)$ is a free generating set for $\mathcal{F}$ and $Q$ is a quiver on $[n]$ where vertex $i$ is mutable if $i \leq r$ and is frozen otherwise. The tuple $\mathbf{x}$ is a \textbf{cluster}, the elements $x_i$ are \textbf{cluster variables}, and the cluster variable $x_i$ is \textbf{mutable} if $i \leq r$ and \textbf{frozen} otherwise.
    \end{definition}

\begin{definition}[Seed mutation]
    Let $\Sigma=(\mathbf{x}, Q)$ be a seed in $\mathcal{F}$ and let $k$ be a mutable vertex of $Q$. \textbf{Mutating $\Sigma$ at $k$} produces a new seed $\mu_k(\Sigma)=(\mathbf{x}', \mu_k(Q))$. The cluster $\mathbf{x}'$ is defined by $\mathbf{x}'= \mathbf{x} \setminus \{x_k\} \cup \{x_k'\}$ where
        \[x_k x_k'= \prod_{i \to k} x_i + \prod_{k \to i} x_i.\]
\end{definition}

Seed mutation is an involution. That is, $\mu_k(\mu_k(\Sigma))=\Sigma$.

\begin{definition}[Cluster algebra]
        Given a seed $\Sigma$ in field $\mathcal{F}$, let $\mathcal{C}$ denote the set of cluster variables obtained by performing arbitrary sequences of mutations to $\Sigma$. The cluster algebra $\mathcal{A}(\Sigma)$ is the $\mathbb{C}$-subalgebra of $\mathcal{F}$ generated by the elements of $\mathcal{C}$, together with the inverses of the frozen variables\footnote{Other common conventions include taking $\mathcal{A}(\Sigma)$ to be the $\mathbb{Z}$-subalgebra with this generating set, or to omit the inverses of frozen variables from the generating set.}.
\end{definition}

Notice that $\mathcal{A}(\Sigma)$ is completely determined by any seed which can be obtained from $\Sigma$ by mutation.

The following theorem summarizes a number of central results in the theory of cluster algebras. The \emph{Laurent phenomenon}, or the Laurent polynomial expression for each cluster variable in terms of an initial cluster, is due to \cite{FZ02}. The sharper \emph{positive Laurent phenomenon}, which asserts the coefficients of this Laurent polynomial are positive, is due to \cite{LS15} for our definition of cluster algebras. The nonnegativity of the denominator vector and its relation to compatibility are due to \cite{CL20-denom-vec}.

\begin{theorem}\label{thm:cluster-Laurent-positivity-denom}
    Let $\Sigma=((x_1, \dots, x_n), Q)$ be a seed and let $z$ be a cluster variable of the cluster algebra $\mathcal{A}(\Sigma)$ which is not in $\Sigma$. Then $z$ has a positive Laurent expression in terms of the initial cluster variables $x_1, \dots, x_n$. More precisely,
    \begin{enumerate}
    \item there is a polynomial $P^{\Sigma}_{z}(x_1, \dots, x_n) \in \mathbb{Z}_{\ge 0}[x_1, \dots, x_n]$ which is not divisible by any $x_i$ and a vector $\mathbf{d}_{z}^{\Sigma}=(d_j) \in (\mathbb{Z}_{\ge 0})^n$ such that
    \[z= \frac{P^{\Sigma}_{z}(x_1, \dots, x_n)}{x_1^{d_1} \cdots x_n^{d_n}}.\]
    \item We have $d_j=0$ if and only if $z$ and the initial cluster variable $x_j$ are \textbf{compatible}, meaning that they appear together in some seed.
    \end{enumerate}
\end{theorem}

\begin{remark}\label{rmk:denom-vec}
    The vector $\mathbf{d}_{z}^{\Sigma}$ of \cref{thm:cluster-Laurent-positivity-denom} is called the \emph{denominator vector} of the cluster variable $z$. It is conjectured that different cluster variables have different denominator vectors, see \cite[Conjecture 7.6]{FZ07}.
\end{remark}

We now turn to $F$-polynomials and $g$-vectors, which are another way to encode cluster variables. First, we define a special choice for the frozen parts of a quiver and seed. This special choice turns out to encode the cluster variables for arbitrary frozen variables. 

\begin{definition}[Principal coefficients]
    Let $Q$ be a quiver with mutable vertices $[r]$. The \textbf{framed quiver} $\Qprin$ is the quiver obtained from $Q$ by deleting all frozen vertices, adding $r$ additional frozen vertices $1^{\bullet}, \dots, r^{\bullet}$ and one arrow $i^{\bullet} \to i$ for each $i \in [r]$. For a seed $\Sigma=((x_1, \dots, x_n), Q)$, the \emph{framing} of $\Sigma$ is the seed $\prinSig=((x_1, \dots, x_r, y_1, \dots, y_r), \Qprin)$. We say that $\prinSig$ has \textbf{principal coefficients}. 
\end{definition}

Because the mutable parts of $Q$ and $\Qprin$ are the same, \cite[Theorem 4.8]{CKLP} implies that there is a one-to-one correspondence between mutable variables of $\mathcal{A}(\Sigma)$ and $\mathcal{A}(\prinSig)$. Concretely, if the seed $\mu_{j_q} \circ \cdots \circ \mu_{j_1}(\Sigma)$ has cluster $(z_1, \dots, z_r, x_{r+1}, \dots, x_n)$ and the seed $\mu_{j_q} \circ \cdots \circ \mu_{j_1}(\prinSig)$ has cluster $(z_1', \dots, z_r', y_{1}, \dots, y_r)$, then $z_i$ corresponds to $z_i'$.

\begin{definition}[$F$-polynomials]
Let $\Sigma$ be a seed. Let $z$ be a mutable cluster variable in $\mathcal{A}(\Sigma)$, and let $z'$ be the corresponding cluster variable in $\mathcal{A}(\prinSig)$. We define the \textbf{$F$-polynomial} of $z$ with respect to $\Sigma$ as
\[F_{z}^{\Sigma}(y_1, \dots, y_r):=P_{z'}^{\prinSig}(1, \dots, 1, y_1, \dots, y_r)\]
where $P_{z'}^{\prinSig}$ is as in \cref{thm:cluster-Laurent-positivity-denom}.
If $Q$ is the quiver of $\Sigma$, we may also write $F_{z}^Q$ for $F_{z}^{\Sigma}$.
\end{definition}

Notice that the $F$-polynomial $F_z^{\Sigma}$ is computed using $\mathcal{A}(\prinSig)$, so it may seem strange to associate the $F$-polynomial to a cluster variable of $\mathcal{A}(\Sigma)$. As we will now review, from the $F$-polynomial $F_z^{\Sigma}$ and the seed $\Sigma$, one can in fact recover the cluster variable $z$ of $\mathcal{A}(\Sigma)$. 

We will need the following notion which is based on \cite[Prop. 3.9]{FZ07} but first appeared in this form in \cite{fraser2016quasi}. 
(See also \cite[Section 11]{Postnikov} and \cite[Remark 7.2]{marsh2016twists} which define the same quantity, called a \emph{face weight} and \emph{shear weight} respectively, in the special case that the $x_i$'s label faces of a plabic graph.)

\begin{definition}\label{def:exchange-ratio}
    Let $\Sigma=(\mathbf{x}, Q)$ be a seed. For $j$ a mutable vertex, the \textbf{exchange ratio} is
    \[\hat{y}_j:= \frac{\prod_{i \to j}x_i}{\prod_{j \to i}x_i}.\]
\end{definition}

The next theorem involves the \emph{$g$-vector} of a cluster variable, which can be viewed as the degree of the corresponding variable in $\mathcal{A}(\prinSig)$ with respect to a particular $\mathbb{Z}^r$-grading \cite[Section 6]{FZ07}. There is also a recursive definition for $g$-vectors (see e.g. \cite{M-green}). We will not need the explicit definition, so do not recall it here.

\begin{theorem}[{\cite[Corollary 6.3]{FZ07}}]\label{thm:g-F-cluster-expansion}
    Let $\Sigma$ be a seed with cluster $(x_1, \dots, x_n)$ and let $z$ be a mutable cluster variable of $\mathcal{A}(\Sigma)$. Then there is a vector $\mathbf{g}_{z}^{\Sigma}= (g_j) \in \mathbb{Z}^r$, called the \textbf{$g$-vector} of $z$ such that 
    \[z=x_1^{g_1} \cdots x_r^{g_r} \cdot F_{z}^{\Sigma}(\hat{y}_1, \dots, \hat{y}_r).\]
\end{theorem}

We note that again, the $g$-vector $\mathbf{g}_{z}^{\Sigma}$ depends only on the cluster variable $z'$ in $\mathcal{A}(\prinSig)$ corresponding to $z$. We will sometimes use the notation $\mathbf{g}_{z}^{Q}$ instead of $\mathbf{g}_{z}^{\Sigma}$. We also note that, due to the algebraic independence of the initial cluster variables, once we write $x_1^{g_1} \cdots x_r^{g_r} \cdot F_{z}^{\Sigma}(\hat{y}_1, \dots, \hat{y}_r)$ in lowest terms over a common denominator, we obtain ${P^{\Sigma}_{z}(x_1, \dots, x_n)}/({x_1^{d_1} \cdots x_n^{d_n}})$.

We next describe how $F$-polynomials and $g$-vectors behave under the addition of mutable vertices to the quiver $Q$.

\begin{proposition}\label{prop:F-g-adding-mutable}
    Let $Q$ be a quiver with mutable vertices $[r]$ and $Q'$ be a quiver with mutable vertices $[r']$ with $r \leq r'$ such that the induced subgraph of $Q'$ on $[r]$ is the mutable part of $Q$. Fix a sequence of mutations $\mu_{\mathbf{p}}$ of $Q$. Consider the mutable cluster variable $w:=w_i$ in $\mu_{\mathbf{p}}(\mathbf{x}, Q)$ and the mutable cluster variable $w':=w_i'$ in $\mu_{\mathbf{p}}(\mathbf{x}', Q')$. Denote their respective $g$-vectors by $\mathbf{g}_{w}^Q:=(g_i) \in \mathbb{Z}^r$ and $\mathbf{g}_{w'}^{Q'}:=(g'_i) \in \mathbb{Z}^{r'}$. Then 
    \[F_{w}^{Q}(y_1, \dots, y_r)= F_{w'}^{Q'}(y_1, \dots, y_{r'}) \qquad \text{and} \qquad g_j= g'_j \text{ for }j \in [r].\]
\end{proposition}
\begin{proof}
We may assume that both $Q$ and $Q'$ are framed quivers, since the $F$-polynomials and $g$-vectors are determined by the framings of $Q$ and $Q'$.

We will use induction on the length of the mutation sequence. The base case is when the length of the mutation sequence is $0$; in this case, all $F$-polynomials in both seeds are $1$, and the $g$-vector of the initial cluster variable $x_i$ is the standard basis vector $\mathbf{e}_i$. 

Now suppose the length of the mutation sequence is at least 1. It is straightforward to check, again by induction, that the induced subgraph of $\mu_{\mathbf{p}}(Q')$ on $[r]$ is exactly $\mu_{\mathbf{p}}(Q)$. Moreover, in $\mu_{\mathbf{p}}(Q')$, there are no arrows from the frozen vertices $(r+1)^{\bullet}, \dots, (r')^{\bullet}$ to the mutable vertices $1, \dots, r$. 

We first deal with $F$-polynomials. We use \cite[(5.3)]{FZ07}, which details how $F$-polynomials change under one mutation. If the final mutation in $\mathbf{p}$ is not at $i$, then we have $F_w^Q = F_{w'}^{Q'}$ by the inductive hypothesis. If it is at $i$, then using the inductive hypothesis, the formulas of \cite[(5.3)]{FZ07} for $F_w^Q$ and $F_{w'}^{Q'}$ are nearly the same. The formula for $F_{w'}^{Q'}$ can be obtained from that for $F_w^Q$ by multiplying each term by products of $F_{x_q}^{Q'}$, where $q>r$ and $x_q$ is an initial cluster variable. Since we never mutate at $q$ in the mutation sequence $\mathbf{p}$, these $F$-polynomials are all $1$, so $F_w^Q = F_{w'}^{Q'}$.

We now turn to $g$-vectors. We will use \cite[pg. 3]{M-green}, which details how $g$-vectors change under one mutation. Again, if the final mutation in $\mathbf{p}$ is not at $i$, then the $g$-vectors of $w$ and $w'$ are not affected by the final mutation and we have the desired equality by the inductive hypothesis. If the final mutation is at $i$, then \cite[pg. 3]{M-green} together with the inductive hypothesis show that, for $j \in [r]$, $g_j$ and $g_j'$ differ only by the $j$th coordinate of $g$-vectors $\mathbf{g}_{x_q}^{Q'}$, where $q>r$ and $x_q$ is an initial cluster variable. In this case, $\mathbf{g}_{x_q}^{Q'}= \mathbf{e}_q$ as we never mutate at $q$ in the mutation sequence $\mathbf{p}$, so the $j$th coordinate of this vector is 0. Thus we have $g_j = g_j'$ as desired.
\end{proof}

\section{Dimer face polynomials as $F$-polynomials}\label{sec:dimer_f_poly}

In this section, we show that if $G$ has property $(*)$, the dimer face polynomial $D_G$ is an $F$-polynomial for a particular cluster algebra. We also give an explicit mutation sequence to reach this $F$-polynomial. The main result is stated in \cref{thm:dimer-poly-is-F-poly}.  

\subsection{Reduction sequences for bipartite plane graphs}

We now define the dual quiver of a bipartite plane graph.

\begin{definition}\label{def:dual-quiver} Given a bipartite plane graph $G$, the dual quiver $Q_G$ is constructed as follows:
\begin{itemize}
    \item Place a mutable vertex in each non-infinite face $f$ of $G$.
    \item For each edge $e$ of $G$ which separates two distinct non-infinite faces $f, f'$, draw an arrow across $e$ so that the white vertex is on the right.
    \item Delete oriented 2-cycles, one by one, in any order. 
\end{itemize}
\end{definition}

See \cref{fig:quiver-ex} for an example of $Q_G$.

\begin{figure}
    \centering
    \includegraphics[width=0.5\linewidth]{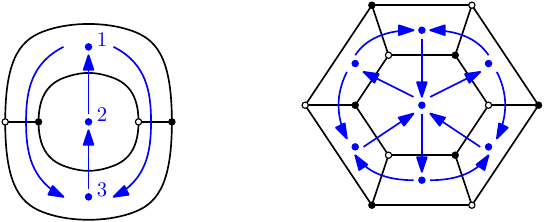}
    \caption{Two examples of $Q_G$, shown superimposed on $G$. Note that on the left, there are two arrows from $1$ to $3$.}
    \label{fig:quiver-ex}
\end{figure}

\begin{remark} 
In the above definition of dual quivers for bipartite plane graphs, an oriented $2$-cycle may arise for instance from a $2$-valent vertex $v$ or from two faces that border each other twice with an even number of edges between them. In the former case, deleting such a $2$-cycle is equivalent to an edge contraction move removing $v$. In the latter case, this configuration cannot occur for a quadrilateral face in a graph with property $(*)$ (see the proof of \cref{lem:moves-effect-on-Q}).  Consequently, deleting 2-cycles in $Q_G$ will not affect the correspondence between square moves of bipartite plane graphs and quiver mutations.   
\end{remark}

\begin{remark}
    Note that all vertices of the dual quiver $Q_G$ defined above are mutable. 
    This is slightly different than the convention taken when defining the dual quiver of a \emph{plabic} graph (see e.g. \cite[Definition 7.1.4]{IntroCA7}).
\end{remark}

We will show that the dimer face polynomial of $G$ is an $F$-polynomial for $\mathcal{A}(Q_G)$. To specify which $F$-polynomial, we need some additional terminology.

\begin{definition}\label{def:moves}
We introduce the following operations on bipartite plane graphs: edge contraction/uncontraction, square move, and bigon removal, respectively.

\begin{center}
\includegraphics[width=\textwidth]{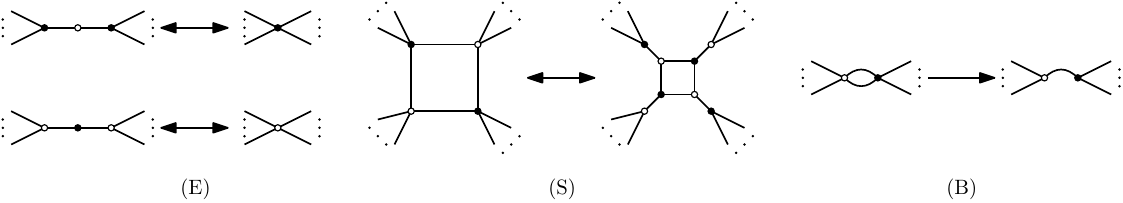} 
\end{center}
On the left-hand side of (E), we assume that the two vertices of the same color are distinct and are degree at least 2. On the left-hand side of (S), we assume all vertices are degree at least 3. On the left side of (B), we assume both vertices are degree at least 2.
\end{definition}

\begin{definition}
    Let $G$ be a graph with property $(*)$. A \emph{reduction sequence} for $G$  is a sequence of the moves (E), (S), (B), which turn $G$ into a graph with a single edge. Numbering the non-infinite faces of $G$, we represent a reduction sequence by the list $\mathbf{r}=(f_1, f_2, \dots, f_q)$ of faces at which (S) or (B) were performed. That is, $f_1$ is the first face at which (S) or (B) is performed, $f_2$ is the second, etc.
\end{definition}

See \cref{fig:reduction-seq} for an example of a reduction sequence.

\begin{figure}
    \centering
    \includegraphics[width=0.8\linewidth]{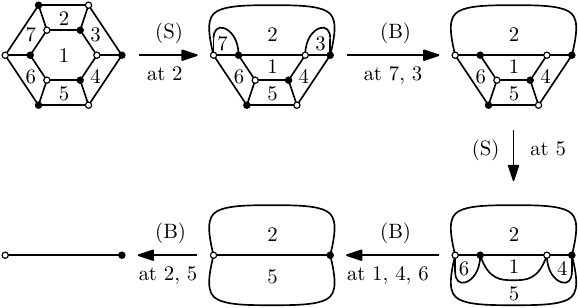}
    \caption{An illustration of the reduction sequence $\mathbf{r}=(2,7,3,5,1,4,6,2,5)$ for the graph $G$ on the upper left. Arrows are labeled by the moves relating successive graphs; (E) moves are omitted.} 
    \label{fig:reduction-seq}
\end{figure}

\begin{proposition}\label{prop:reduc-seq-exists}
    Let $G$ be a graph with property $(*)$. Then, $G$ has at least one reduction sequence.
\end{proposition}
The proof of \cref{prop:reduc-seq-exists}, which appears in the next subsection, follows readily from existing work on \emph{reduced plabic graphs} \cite{Postnikov,IntroCA7}.

We need one additional definition, 
based on \cite[Section 2]{speyer2007perfect} and \cite[Eq. (3.11)]{di2014t} 
which gave a weighting to matchings in the context of the Octahedron Recurrence (equivalently $T$-systems). 

\begin{definition}\label{def:h-vec}
    Let $G$ be a plane graph and let $M$ be a subset of the edges of $G$.
    We define the vector $\hvec^G_M = (h_f)_{f \in \faces(G)}$ 
    by 
    \begin{equation}\label{eq:h-vec}
     h_f:=|f|/2 -|M \cap f| -1. 
    \end{equation}
\end{definition}
We will often, but not always, take $M$ to be a matching of $G$; sometimes $M$ will instead be a matching of a subgraph of $G$.

We now state the main theorem of this section: the dimer face polynomial $D_G$ is an $F$-polynomial for the cluster algebra with initial quiver $Q_G$. If $\mathbf{r}=(i_1, \dots, i_q)$, we write $\mu_{\mathbf{r}} (\Sigma)$ for the seed $\mu_{i_q} \circ \cdots \circ \mu_{i_2}\circ \mu_{i_1}(\Sigma)$, so that the first entry of $\mathbf{r}$ gives the first mutation performed, the second entry the second, etc. 

\begin{theorem}\label{thm:dimer-poly-is-F-poly}
    Let $G$ be a graph with property $(*)$ and let $\mathbf{r}=(f_1, \dots, f_q)$ be a reduction sequence for $G$. Consider a seed $\Sigma=(\mathbf{x}, Q_G)$. Let $z$ be the cluster variable of $\mu_{\mathbf{r}}(\Sigma)$ labeling vertex $f_q$. 
    We have
    \[F^{Q_G}_z(y_1, \dots, y_n)= D_G(y_1, \dots, y_n).\]
    Further, the $g$-vector $\mathbf{g}_z^{Q_G}$ of $z$ is equal to $\mathbf{h}_{\hat{0}}^G$, 
    where $\hat{0}$ denote the bottom matching in our dimer lattice for graph $G$.
\end{theorem}

See \cref{fig:big-ex} for an illustration of \cref{thm:dimer-poly-is-F-poly}.

\begin{figure}
    \centering
    \includegraphics[height=0.9\textheight]{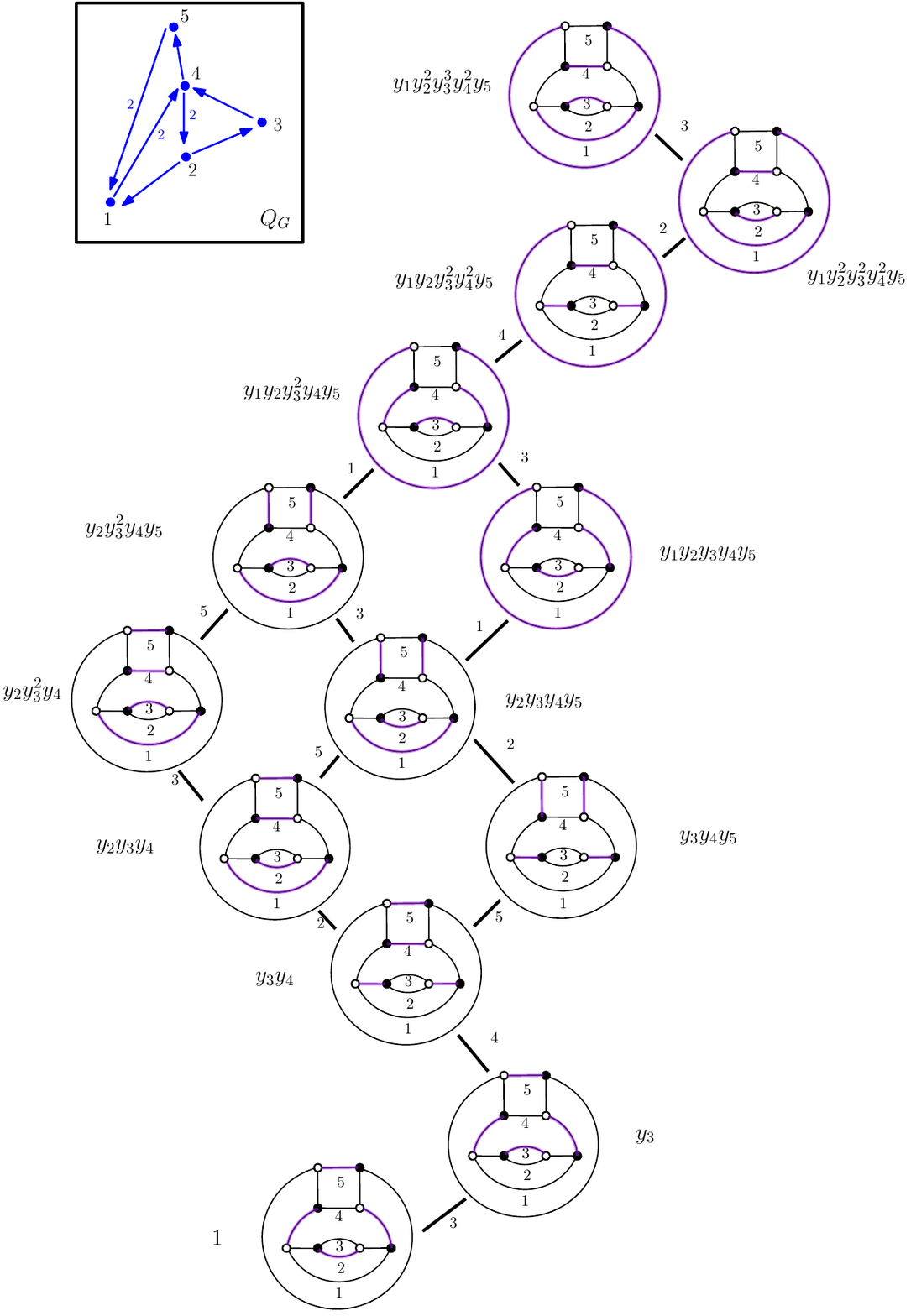}
    \caption{An example of $\mathcal{D}_G$ and $Q_G$. Next to each matching $M$ is its multivariate rank $\mrk(M)=\mathbf{y}^{\height(M)}$. In $Q_G$, the numbers next to arrows are multiplicities. A reduction sequence for $G$ is $(3,2,4,5,1)$. The cluster variable labeling vertex $1$ in the seed $\mu_1 \circ \mu_5 \circ \mu_4 \circ \mu_2 \circ \mu_3 (\mathbf{x}, Q_G)$ has $F$-polynomial $D_G$ and $g$-vector $\mathbf{h}_{\hat{0}}^G=(0,0,-1,0,0)$.}
    \label{fig:big-ex}
\end{figure}

\begin{remark}
    \cref{thm:dimer-poly-is-F-poly} implies that the dimer face polynomial of an arbitrary plane graph $G$ is the $F$-polynomial of a cluster monomial in some cluster algebra. If $G$ has no perfect matchings, then $D_G=1$. If $G$ has at least one perfect matching, we use \cref{rmk:dimer-lattice-arbitrary-graph} to endow $\mathcal{D}_G$ with a distributive lattice structure. Let $G' = G_1 \sqcup \cdots \sqcup G_r$ be the graph obtained by deleting all edges of $G$ which are not in any perfect matching. We have $D_G= D_{G_1} \cdots D_{G_r}$, so $D_G$ is a product of compatible $F$-polynomials for the cluster algebra with initial quiver $Q_{G'}=Q_{G_1} \sqcup \cdots \sqcup Q_{G_r}$. We note that $Q_{G'}$ may not be an induced subquiver of $Q_G$.
\end{remark}

Using \cref{thm:dimer-poly-is-F-poly}, we may readily give the cluster expansion of $z$, 
extending results in \cite{speyer2007perfect, di2014t, vichitkunakorn2016solutions, eager2012colored} 
to cases of arbitrary graphs with property $(*)$.  
We delay our proof of the cluster expansion
until after the proof of \cref{thm:dimer-poly-is-F-poly}.  
The formula here is reminiscent of formulas for twisted Pl\"ucker coordinates in \cite{marsh2016twists,MS-twist} and we discuss the relationship between our results and theirs in Section \cref{sec:positroid-var-applications}. 

\begin{corollary}\label{cor:easy-cluster-expansion}
    Let $G$, $\Sigma$, and $z$ be as in \cref{thm:dimer-poly-is-F-poly}. The Laurent polynomial expression for $z$ in terms of the initial seed $\Sigma$ is
    \[z= \sum_{M \in \mathcal{D}_G} \mathbf{x}^{\mathbf{h}_M^{G}}
    = \sum_{M \in \mathcal{D}_G} \prod_{f \in \faces(G)} x_f^{|f|/2 -|M \cap f| -1}  \]
    and the denominator vector of $z$ (see \cref{thm:cluster-Laurent-positivity-denom}) has all coordinates equal to $1$.
\end{corollary}

We obtain a corollary on the Newton polytopes of the $F$-polynomials in \cref{thm:dimer-poly-is-F-poly} by applying \cref{prop:face-lattice-newton,cor:newton-lattice-points}.

\begin{corollary}\label{cor:F-poly-newton-poly}
    Let $G$ and $z$ be as in \cref{thm:dimer-poly-is-F-poly}. Then every integer point in the Newton polytope of 
    $F^{Q_G}_z=D_G$ is a vertex, and the face lattice of the Newton polytope is isomorphic to the lattice of elementary subgraphs of $G$.
\end{corollary}

There has been substantial interest in Newton polytopes of $F$-polynomials and Laurent polynomial expressions for cluster variables \cite{Kalman-Newton,Fei-Newton,MSB-Newton,Li-Pan}. By \cite[Theorem 4.4]{Li-Pan}, the Newton polytope $\Newton(F)$ of any $F$-polynomial is \emph{saturated}, meaning that all lattice points of $\Newton(F)$ are exponent vectors of terms in $F$, and a lattice point is a vertex if and only if the corresponding term has coefficient 1. We remark that \cref{cor:F-poly-newton-poly} recovers these results when $F=D_G$, as by \cref{cor:01} all coefficients of $D_G$ are 1.

\subsection{Results needed for the proof of \cref{thm:dimer-poly-is-F-poly}}

In this subsection, we build up to the proof of \cref{thm:dimer-poly-is-F-poly}, proving some preparatory results and then comparing how the $F$-polynomial changes when a move is performed to $G$ to how the dimer face polynomial changes. We deal separately with (E), (S), and (B) in \cref{lem:e-moves-dont-change-dimers}, \cref{thm:sq-move-step}, and \cref{thm:bigon-step}, respectively. In the next subsection, we use these results and induction to complete the proof of \cref{thm:dimer-poly-is-F-poly}.

We start with a helpful lemma on the structure of graphs with property $(*)$. 
See also \cite[Lemma 3.7]{MS-twist}.

\begin{lemma}\label{lem:edges-dif-face-each-side}
    If $G$ is a graph with property $(*)$ and is not a single edge, then no edge of $G$ has the same face on both sides.
\end{lemma}
\begin{proof}
    Suppose for the sake of contradiction that some edge $e$ of $G$ has a (finite or infinite) face $f$ on both sides. Choose a short line segment $l$ perpendicular to $e$, with endpoints $q^+$ and $q^-$ in $f$. As $f$ is path-connected, there is a path $p$ from $q^+$ to $q^-$ in $f$, which in particular does not intersect $G$. Now, the closed curve $p \cup l$ encloses some vertices of $G$. Let $G'$ be the induced subgraph of $G$ on these vertices. Note that $G'$ includes exactly one vertex of $e$, say $v$, and is a connected component of $G \setminus e$.

    \begin{center}
        \includegraphics[width=0.25\textwidth]{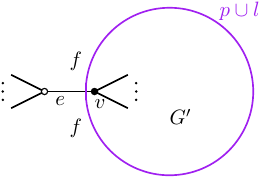}
    \end{center}

     By assumption, $e$ is in some matching $M$ of $G$. The matching $M$ restricts to a perfect matching of $G' \setminus v$, so $G' \setminus v$ has an even number of vertices and $G'$ has an odd number of vertices. Since $G$ is not a single edge and is connected, $e$ is adjacent to another edge $e'$. By assumption, there is a matching $N$ of $G$ which uses $e'$ and thus does not use $e$ (and this is true whether or not $e'$ is incident to $v$ or the other endpoint of $e$).  Such a matching $N$ would have to restrict to a perfect matching of $G'$. But $G'$ has an odd number of vertices, a contradiction.
\end{proof}

Now, we verify that the hypothesis of \cref{thm:dimer-poly-is-F-poly} is preserved under the three moves. 

\begin{proposition}\label{prop:moves-preserve-*}
    The set of graphs with property $(*)$ is closed under moves (E), (S), (B).
\end{proposition}

\begin{proof}
    It is straightforward to verify that if $G$ has property $(*)$, then applying (E) or (B) results in another graph with property $(*)$. Notice that on the left-hand side of (B) either both vertices have degree 2 or both have degree at least 3, or $G$ would not have property $(*)$.

    For move (S), let $G'$ be the graph on the left and $G$ be the graph on the right, and $f$ the face at which the square move is performed. We show that if $G$ has property $(*)$, then so does $G'$. The reverse argument is identical, as one can use (E) moves to assume all vertices of $f$ are trivalent in $G'$.  
    For every matching $M$ of $G$, there is a matching of $G'$ which agrees with $M$ on all edges of $G'$ that are not in $f$ (see \cref{fig:square-preserves-*}). 
    So, we only need to show that every edge of $f$ is in some matching of $G'$. By assumption, for $i=1, 2,3,4$, $G$ has a matching $M_i$ which uses $e_i$. If some $M_i$ uses all of $e_1, e_2, e_3, e_4$ as shown on the right of \cref{fig:square-preserves-*}, then $G'$ has a matching which admits an up-flip at $f$, shown on the left of \cref{fig:square-preserves-*}, and every edge of $f$ is either in this matching or its up-flip.

    \begin{figure}
        \centering
        \includegraphics[width=0.5\textwidth]{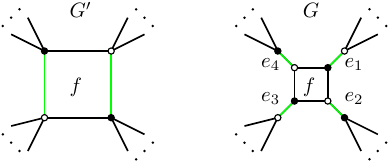}
        \caption{Graphs $G'$ and $G$ related by a square move at $f$, with the edges of $G$ labeled as in the proof of \cref{prop:moves-preserve-*}. In green, a matching $M$ of $G$ and a matching of $G'$ which agrees with $M$ away from $f$.}
        \label{fig:square-preserves-*}
    \end{figure}    

    We will, in fact, show that property $(*)$ guarantees there is a matching of $G'$ that contains all four edges $e_1, e_2, e_3,$ and $e_4$.  Assume otherwise.  By assumption, for $i=1, 2,3,4$, $G$ has a matching $M_i$ which uses $e_i$.  If no $M_i$ uses all of the $e_i$, there are two possibilities for $M_i$ around $f$, shown in \cref{fig:options-for-Mi}. 
    Note that $M_i$ must also use either $e_{i-1}$ or $e_{i+1}$, taking subscripts modulo $4$.  

    \begin{figure}
        \centering
        \includegraphics[width=0.8\textwidth]{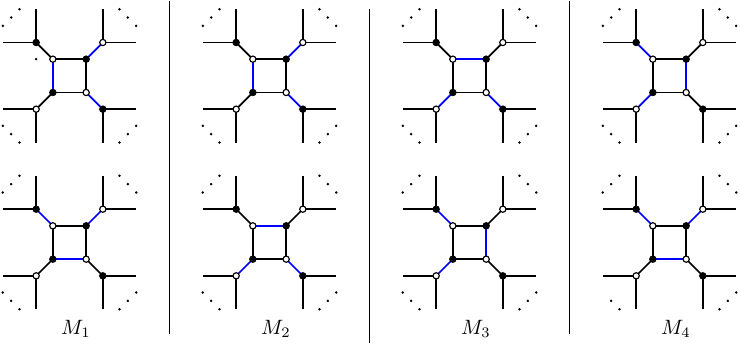}
        \caption{In the proof of \cref{prop:moves-preserve-*}, the options for the matching $M_i$ around the face $f$.}
        \label{fig:options-for-Mi}
    \end{figure}

    Regardless of how you choose the $M_i$ from the four columns of \cref{fig:options-for-Mi}, you will always obtain two matchings $M_i, M_j$ containing opposite edges of $f$. In particular, if $M_1$ uses $e_1$ and $e_2$ without loss of generality, then either (i) $M_2$ or $M_3$ uses $e_3$ and $e_4$ or (ii) $M_2$ uses $e_2$ and $e_3$ while $M_3$ uses $e_1$ and $e_4$.  The union $M_i \cup M_j$ is a \emph{double dimer} of $G$, and consists of a disjoint union of even cycles, with alternating edges in $M_i$ and $M_j$, and paths, where every edge is in both $M_i$ and $M_j$. In particular, each vertex of $G$ has degree exactly 2 in $M_i \cup M_j$. Around $f$, the union $M_i \cup M_j$ looks like the left of \cref{fig:dbl-dimer}, with one edge of $f$ in the cycle $C_1$ and the other edge in the cycle $C_2$.
    
    The cycles $C_1$ and $C_2$ are not the same. To see this, walk along $C_2$ starting at $v_1$, then going to the edge in $f$, then to $v_2$, and so on; call this the ``counterclockwise walk" on $C_2$. An edge $e$ of $C_2$ is in $M_i$ precisely if one encounters first the black vertex of $e$ and then the white vertex in the counterclockwise walk. If $C_2=C_1$, then the counterclockwise walk must leave $v_2$, encounter $v_4$, then the blue edge in $f$, then $v_3$, then finally return to $v_1$. Since $M_i \cup M_j$ is planar, this means some vertex of $G$ has degree 4 in $M_i \cup M_j$, which is impossible.
    
    As $C_1 \neq C_2$, the matching $N:= (M_i \setminus C_2) \cup (M_j \cap C_2)$ is as pictured in \cref{fig:dbl-dimer}, and contains $e_1, e_2, e_3, e_4$ as desired.
  \begin{figure}
      \centering
      \includegraphics[width=0.5\textwidth]{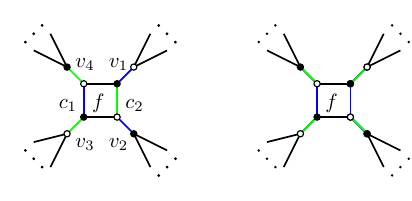}
      \caption{
      Left: the union $M_i \cup M_j$ in the proof of \cref{prop:moves-preserve-*}, with $M_i$ in green. The edge of $f$ in $M_j$ is in the cycle $C_1,$ while the edge of $f$ in $M_i$ is in the cycle $C_2$. Right: in green, the matching $N$ in the proof of \cref{prop:moves-preserve-*}. This is obtained from $M_i$ by removing all edges in $C_2 \cap M_i$ and adding all edges in $C_2 \cap M_j$.}
      \label{fig:dbl-dimer}
  \end{figure}    
\end{proof}

We next show that reduction sequences always exist, which is \cref{prop:reduc-seq-exists}.

\begin{proof}[Proof of \cref{prop:reduc-seq-exists}] We induct on the number of non-infinite faces of $G$. The base case is when $G$ is a single edge. The trivial sequence of no moves is a reduction sequence.

Now assume $G$ has at least one non-infinite face. It suffices to show that one can apply (E) and (S) moves to $G$ to obtain a graph $G'$ with a bigon face $f$. Indeed, if $G''$ is the graph obtained from $G'$ by applying (B) to $f$, then by induction, $G''$ has a reduction sequence $\mathbf{r}''$. A reduction sequence of $G$ is given by the sequence of moves to get from $G$ to $G'$, then the move (B) at $f$, then $\mathbf{r}''$. 

We now show that there is a sequence of (E) and (S) moves which turns $G$ into a graph with a bigon face. We use the theory of \emph{reduced plabic graphs}, as presented in \cite{IntroCA7}. To make $G$ into a plabic graph, choose any vertex $v$ on the infinite face of $G$, embed $G$ in a disk, and add an edge from $v$ to a vertex on the disk, which we label $1$. The resulting graph $H$ is a plabic graph with one boundary vertex, which is not reduced as it has at least one face which does not touch the boundary of the disk. 

By \cite[Proposition 7.4.9]{IntroCA7}, $H$ can be transformed by the local moves (M1), (M2), (M3) of \cite[Definition 7.1.3]{IntroCA7} into a bipartite plabic graph $H'$ with a bigon face\footnote{\cite[Proposition 7.4.9]{IntroCA7} guarantees that one can turn $H$ into a graph $H'$, not necessarily bipartite, with a ``hollow digon" face. If the vertices of the digon are the same color, one can contract one of the edges to get a loop, then add a degree 2 vertex of the opposite color to create a bigon face with vertices of opposite colors. Then one can apply (M2) moves to make $H'$ bipartite.}, without ever creating internal leaves in the (M3) moves. 
By adding in extra (M2) moves to make the intermediate graphs bipartite, we may also obtain $H'$ from $H$ by a sequence of (E) and (S) moves. If no trivalent vertex in an (S) move is adjacent to the vertex 1, we still get a sequence of valid (E), (S) moves when we delete the vertex 1. We apply this sequence of moves to $G$ gives a graph $G'$ with a bigon face. If a trivalent vertex in an (S) move is adjacent to vertex 1, then we apply the sequence of moves up until that square move to $G$. The resulting graph has a square face with one degree 2 vertex $w$. Applying (E) to contract the edges adjacent to $w$ gives a graph $G'$ with a bigon face.
\end{proof}

The next lemma establishes how the moves effect the dual quiver of $G$.

\begin{lemma}\label{lem:moves-effect-on-Q}
    Suppose $G$ has property $(*)$. If $G'$ is a graph obtained from $G$ by move 
  (E), move (S) at face $f$, or move (B) at face $f$ then 
       $$Q_{G'}=\begin{cases}
        Q_G\\
        \mu_f(Q_G)\\
        \mu_f(Q_G) \text{ with vertex } $f$ \text{ deleted}
    \end{cases}$$
    respectively.
\end{lemma}

\begin{remark}  For the cases of (E) and (S) moves, this has appeared previously in \cite[Theorem 2]{Scott} and \cite[Prop. 7.1.5]{IntroCA7}.  We note that the additional hypothesis (7.1.1) used in the latter is implied by the assumption that $G$ has property $(*)$ because of \cref{lem:edges-dif-face-each-side}.
\end{remark}

\begin{proof}[Proof of \cref{lem:moves-effect-on-Q}]
    \noindent \textbf{(E):} Performing move (E) either removes or introduces an oriented 2-cycle. Since oriented 2-cycles are deleted from $Q_G$, performing (E) does not change the dual quiver.

    \noindent \textbf{(S):} We may assume using (E) moves that the vertices of face $f$ are trivalent, as (E) does not change the dual quiver. \cref{lem:edges-dif-face-each-side} implies that among the four faces surrounding $f$, the consecutive ones are distinct; otherwise there would be an edge of $G$ with the same face on both sides. Notice that none of the arrows dual to edges of $f$ are in oriented 2-cycles, and so these four arrows are exactly the arrows incident to $f$ in $Q_G$. It is now straightforward to check that $Q_{G'}$ is obtained from $Q_G$ by adding an arrow $f' \to f''$ for each path $f' \to f \to f''$ in $Q_G$, reversing all arrows incident to $f$, and deleting oriented 2-cycles. This shows $Q_{G'}=\mu_f(Q_G)$.

    \noindent \textbf{(B):} Property $(*)$ implies that the vertices of the bigon face $f$ are either both degree 2 or both degree at least 3. (If exactly one vertex were degree 2, an edge adjacent to the other vertex would never be used in a matching.) In the former case, $G$ is just the bigon $f$, $Q_G$ is a quiver with one vertex and no arrows, and $Q_{G'}$ is the empty quiver, and the claim holds. 
    
    So we may suppose both vertices of $f$ are degree at least 3. Let $f', f''$ be the faces adjacent to $f$. These faces are distinct, as otherwise, one could use (E) to create an edge with the same face on both sides, contradicting \cref{lem:edges-dif-face-each-side}. The arrows adjacent to $f$ in $Q_G$ are, say, $f' \to f \to f''$. The quiver $Q_{G'}$ is obtained from $Q_G$ by replacing the path $f' \to f \to f''$ with the edge $f' \to f''$, deleting $f$ and then possibly deleting a 2-cycle. This is exactly the same as $\mu_f(Q_G)$ with the vertex $f$ deleted.
    
\end{proof}

We now turn to how the $F$-polynomials in \cref{thm:dimer-poly-is-F-poly} and the dimer face polynomials $D_G$ change under each of the moves.

We begin with (E) moves. The $F$-polynomials in \cref{thm:dimer-poly-is-F-poly} do not depend on the (E) moves in a reduction sequence for $G$. The next lemma, which is well known (see e.g. \cite[Figure 23]{GK}), shows that the dimer face polynomials do not either. 

\begin{lemma}\label{lem:e-moves-dont-change-dimers}
    Let $G$ be a graph with property $(*)$ and suppose that $G'$ is related to $G$ by move (E). Then, identifying faces of $G$ and $G'$, $D_G=D_{G'}$.  
\end{lemma}
\begin{proof}
    There is a natural bijection $\epsilon$ between the dimer sets $\mathcal{D}_G$ and $\mathcal{D}_{G'}$, given by
\begin{center}
    \includegraphics[width=0.7\textwidth]{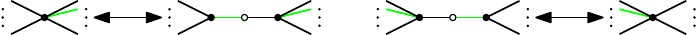}
\end{center}
where the dimers agree on all edges that are not pictured (and the color of the vertices may be opposite).
This map and its inverse are order-preserving: one can up-flip a face $f$ in a dimer $M$ of $G$ if and only if one can up-flip the face $f$ in the dimer $\epsilon(M)$ of $G'$. 
\end{proof}

For square and bigon moves, we will need the following result, relating $F$-polynomials computed with respect to adjacent initial seeds. This result is a combination of \cite[Proposition 2.4]{DWZ10} and \cite[(9.1)]{DWZ10}; in particular, \cite[Proposition 2.4]{DWZ10} involves two additional parameters $h_b, h_b'$, which are shown in \cite[(9.1)]{DWZ10} to equal $\min(0, g_b)$ and $\min(0, g_b')$ respectively. 
We note that \cite[Proposition 2.4]{DWZ10} is essentially \cite[(6.15),(6.28)]{FZ07}, though in somewhat different notation.
 
\begin{proposition}\label{prop:adjacent-F}
    Let $\Sigma=(\mathbf{x}, Q)$ be a seed and $\Sigma':= \mu_b(\Sigma)$. Choose $z$ a mutable cluster variable in $\mathcal{A}(\Sigma)$, with $g$-vector $\mathbf{g}^{\Sigma}=(g_1, \dots, g_n)$ with respect to $\Sigma$ and $g$-vector $\mathbf{g}^{\Sigma'}=(g'_1, \dots, g'_n)$ with respect to $\Sigma'$. Let $q_{ab}= \#\{a \to b \text{ in }Q\} - \#\{b \to a \text{ in }Q\}$. We set 
   
    \begin{equation}\label{eq:y-mut}
    y_a':= \begin{cases}
        y_b^{-1} & \text{if } a=b\\
        y_a(1+y_b)^{q_{ab}} & \text{if } q_{ab}\geq 0\\
        y_a (1+y_b^{-1})^{q_{ab}} & \text{if } q_{ab}\leq 0
    \end{cases}.
    \end{equation}
    Then 
    \begin{equation}\label{eq:adjacent-F}
   (1+y_b)^{\min(0, g_b)}F^{\Sigma}_z(y_1, \dots, y_n) = (1+y_b')^{\min(0, g'_b)} F^{\Sigma'}_z(y_1', \dots, y_n')     
    \end{equation}
  and 
    \begin{equation}\label{eq:g-mut}
        g_a'=\begin{cases}
            -g_b & \text{if } a=b\\
            g_a+q_{ab} g_b - q_{ab}\min(0, g_b) & \text{if } q_{ab} \geq 0\\
            g_a - q_{ab}\min(0, g_b) & \text{if } q_{ab} \leq 0\\
        \end{cases}.
    \end{equation}
\end{proposition}

Since mutation is an involution, we may also reverse the roles of $\Sigma, \Sigma'$ in \cref{prop:adjacent-F}, and the same result will hold.

Now we turn to square moves. Recall from \cref{def:h-vec} the definition of the vector $\hvec^G_{\hat{0}}$. The next theorem shows that the conclusion of \cref{thm:dimer-poly-is-F-poly} is preserved under applying square moves to $G$.

\begin{theorem}\label{thm:sq-move-step}
    Suppose $G$ and $G'$ are graphs with property $(*)$ related by a square move at face $f$. Let $\mathbf{r}$ be a reduction sequence for $G$, and let $z$ be as in \cref{thm:dimer-poly-is-F-poly}. The sequence $\mathbf{r}'=(f, \mathbf{r})$ is a reduction sequence for $G'$; define $z'$ analogously to $z$, using $G'$ instead. Let $\mathbf{g}:=\mathbf{g}^{Q_G}_z$ and $\mathbf{g}':=\mathbf{g}^{Q_{G'}}_{z'}$ denote the $g$-vector of $z$ and $z'$, respectively.  Then
    
    \vspace{1em}
    
  \begin{itemize}
      \item The resulting cluster variables $z$ and $z'$ are equal.
      \item We have $\mathbf{g}= \hvec^G_{\hat{0}}$ if and only if $\mathbf{g}'= \hvec^{G'}_{\hat{0}}$.
      \item Assuming $\mathbf{g}= \hvec^G_{\hat{0}}$ and $\mathbf{g}'= \hvec^{G'}_{\hat{0}}$, then $D_G= F^{Q_G}_{z}$ if and only if $D_{G'}= F^{Q_{G'}}_{z'}$.
  \end{itemize}
    
\end{theorem}

\begin{proof}
We first describe the relationship between the two $F$-polynomials and the two $g$-vectors. By \cref{lem:moves-effect-on-Q}, $Q_{G'}= \mu_f(Q_G)$. The $F$-polynomial depends only on the initial quiver, not the initial cluster variables,
so we may choose to compute $F^{Q_{G}}_{z}$ by framing the seed $\Sigma= (\mathbf{x}, Q_G)$ and to compute $F^{Q_{G'}}_{z'}$ by framing the seed $\mu_f(\Sigma)$. With this choice, the mutation sequences to obtain $z$ and $z'$ differ by one entry and the two initial seeds are related by that same mutation, hence we obtain the equality $z =z'$.  This shows our first claim and ensures that  
we are in exactly the situation of \cref{prop:adjacent-F}. So the $g$-vectors are related by \eqref{eq:g-mut} and we obtain $F^{Q_{G}}_{z}=F^{\Sigma}_z$ from $F^{Q_{G'}}_{z'}=F^{\mu_f(\Sigma)}_z$ by the substitution $y_i \mapsto y_i'$, where $y_i'$ is defined in \eqref{eq:y-mut}, and then multiplying by 
\[\frac{(1+ y_f^{-1})^{\min(0, -g_f)}}{(1+y_f)^{\min(0, g_f)}}.\]

For the second item of the theorem, recall that $\mathbf{g}'$ is obtained from $\mathbf{g}$ by \eqref{eq:g-mut} with $b=f$, and swapping the roles of $\Sigma$ and $\Sigma'$ (so swapping $\mathbf{g}$ with $\mathbf{g'}$ and $Q$ with $\mu_b(Q)$) in \eqref{eq:g-mut}, the equality still holds. Suppose $\hvec^{G}_{\hat{0}}= \mathbf{g}$. If $\hvec^{G'}_{\hat{0}}$ is obtained from $\hvec^{G}_{\hat{0}}$ by \eqref{eq:g-mut}, or if $\hvec^{G}_{\hat{0}}$ is obtained from $\hvec^{G'}_{\hat{0}}$ by \eqref{eq:g-mut} with $Q_{G'}$ in place of $Q_G$, then $\hvec^{G'}_{\hat{0}}=\mathbf{g}'$. One may also reverse this argument by swapping $G$ and $G'$ everywhere. 
So it suffices to show that either $\hvec^{G'}_{\hat{0}}$ is obtained from $\hvec^{G}_{\hat{0}}$ by \eqref{eq:g-mut} with $b=f$, or that $\hvec^{G}_{\hat{0}}$ is obtained from $\hvec^{G'}_{\hat{0}}$ by \eqref{eq:g-mut} with $b=f$ and $Q_{G'}$ rather than $Q_G$. 

There are 4 cases,

depending on the behavior of the bottom matching of $G$ around $f$. See \cref{fig:sq-move-bottom-matching} for an illustration.

Cases (a) and (b) differ since $\hat{0}$ includes a white-black edge of the face $f$ in case (a) but this is a black-white edge in case (b).  Further, $90^\circ$ rotations of cases (c) and (d) are impossible because such a rotation of case (c) would yield a matching of $G$ that contains the black-white edges of $f$ and thus admits a down-flip at $f$, so cannot be the bottom matching $\hat{0}$.  Analogously, a $90^\circ$ rotation of case (d) would yield a matching of $G'$ containing the black-white edges of $f$ which would not be the bottom matching.

Note that we may have $i=k$ or $j=\ell$ (but not both), and \cref{lem:edges-dif-face-each-side} implies no other faces can be the same. It is straightforward to verify that in each of the 4 cases, if no faces of $G$ can be down-flipped in the matching using the edges in the bottom row of \cref{fig:sq-move-bottom-matching}, then no faces of $G'$ can be down-flipped in the matching using the edges in the top row. That is, in each case, the top row of \cref{fig:sq-move-bottom-matching} does in fact depict the bottom matching of $G'$.

\begin{figure}
    \centering
    \includegraphics[width=0.9\textwidth]{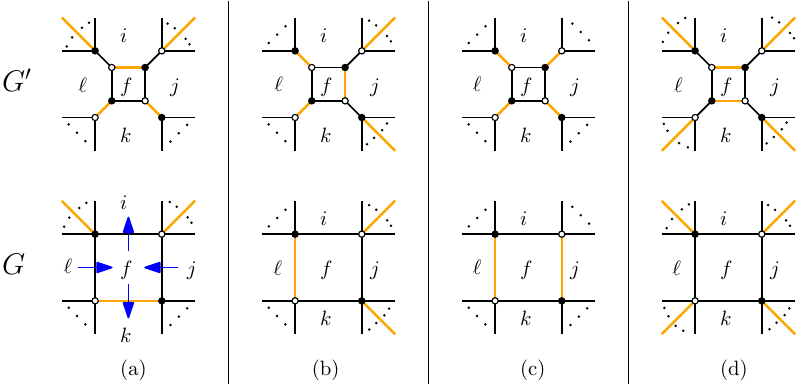}
    \caption{In orange, possible bottom matchings of $G'$ (top) and $G$ (bottom) in the proof of \cref{thm:sq-move-step}. We omit the 180 degree rotations of (a) and (b), since the arguments are identical for those cases. Matchings in the same column agree on all edges that are not shown. In Case (a), bottom, the arrows of $Q_G$ involving $f$ are in blue.} 
    \label{fig:sq-move-bottom-matching}
\end{figure}

In the arguments below, we use the notation $\hvec^G_{\hat{0}}=(h_a)_{a \in \faces(G)}$ and $\hvec^{G'}_{\hat{0}}=(h'_a)_{a \in \faces(G')}$.   
We follow the labeling of faces from \cref{fig:sq-move-bottom-matching}.  
In cases (a), (b),  
\[h_f=2-1-1=0 \qquad \text{and}\qquad h'_f=2-1-1=0.\]
We also see that we add two edges to each of the faces $i,j,k,\ell$ when going from $G$ to $G'$, and increase the number of edges per face in the bottom matching by one. No other faces are affected. Thus we have 
$\hvec^G_{\hat{0}}=\hvec^{G'}_{\hat{0}}$. By inspection, in this case $\hvec^{G'}_{\hat{0}}$ is obtained from $\hvec^G_{\hat{0}}$ by \eqref{eq:g-mut} with $b=f$. 

In case (c), we have $h_f=-1$ and $h'_f=1$. One can compute that
\begin{equation*}
    h'_i= \begin{cases}
        h_i -1 & \text{if }i \neq k\\
        h_i -2 & \text{if }i=k
    \end{cases}, \qquad \text{} \qquad
     h'_j=h_j,
     \qquad \text{} \qquad
    h'_k= \begin{cases}
        h_k -1 & \text{if }i \neq k\\
        h_k -2 & \text{if }i=k
    \end{cases} \qquad \text{and} \qquad
     h'_\ell=h_\ell.
\end{equation*}
For example, the first equality is because we add two (or four) edges to $i$, and increase the number of edges in the bottom matching by two (or four).
From the top left of \cref{fig:sq-move-bottom-matching}, we see $q_{i f}$ is either $-1$ or $-2$, depending on if $i =k$. So we have
\[h_i'= h_i - q_{if} \min(0, h_f).\]

We also have $h_j'=h_j = h_j + q_{jf} h_f - q_{jf} \min(0, h_f)$, since $h_f <0$. Analogous formulas hold for $h_k', h_\ell'$ and all other entries of $\hvec^G_{\hat{0}}$ and $\hvec^{G'}_{\hat{0}}$ agree. So $\hvec^{G'}_{\hat{0}}$ is obtained from $\hvec^G_{\hat{0}}$ by \eqref{eq:g-mut} with $b=f$ for case (c) as well.

Case (d) is identical to case (c) if you swap $G$ and $G'$ and rotate face labels by $90^\circ$, i.e. rename $i$ with $j$, $j$ with $k$, etc. So the argument for case (c) shows $\hvec^{G}_{\hat{0}}$ is obtained from $\hvec^{G'}_{\hat{0}}$ by \eqref{eq:g-mut} using $Q_{G'}$ instead of $Q_G$ and with $b=f$. 

We now turn to the third item in the theorem, on $F$-polynomials. It suffices to verify that $D_G$ can be obtained from $D_{G'}$ by the substitutions $y_f \mapsto y_f^{-1}$ and
\[ y_i\mapsto
    y_i\left(\frac{y_f}{1+y_f} \right)^{\#\{f \to i\}}, \quad y_j \mapsto y_j (1+y_f)^{\#\{j \to f\}}, \quad y_k\mapsto
    y_k\left(\frac{y_f}{1+y_f} \right)^{\#\{f \to k\}}, \quad y_\ell \mapsto y_\ell (1+y_f)^{\#\{\ell \to f\}},\]
(where the exponents are based on the arrows of $Q_G$, and thus are all equal to $1$ unless two faces coincide), followed by multiplication by 
\begin{equation}
\label{eq:multfactor} 
\frac{(1+ y_f^{-1})^{\min(0, -h_f)}}{(1+y_f)^{\min(0, h_f)}} = 
\begin{cases}
1 & \text{ in cases (a),(b) where }h_f=0\\
(1+y_f) & \text{ in case (c) where }h_f=-1\\
(1+y_f^{-1})^{-1} = y_f (1+y_f)^{-1} & \text{ in case (d) where }h_f=1.
\end{cases}
\end{equation}

We will use this below after we first analyze the effect of the substitutions $y_f, y_i, y_j, y_k,$ and $y_\ell$. 

We first consider terms of $D_{G'}$.  
Note that for each matching $M$ of $G'$, there are either one or two matchings of $G$ that agree with $M$ away from $f$, which we will call $N$ and $N'$. Similarly, for each matching $N$ of $G$, there are either one or two matchings, called $M$ and $M'$, which agree with $N$ away from $f$. See \cref{fig:sq-move-match-corresp} for an illustration.
Our method of proof is to show that for each of the possible bottom matchings in \cref{fig:sq-move-bottom-matching}, when we apply the specified substitutions and multiplication to $\mathbf{y}^{\height(M)}$ (or $\mathbf{y}^{\height(M)}+ \mathbf{y}^{\height(M')}$, as appropriate), we obtain exactly $\mathbf{y}^{\height(N)}$ (or $\mathbf{y}^{\height(N)} + \mathbf{y}^{\height(N')}$, as appropriate). We determine the exponents of $y_f, y_i, y_j, y_k, y_\ell$ in $\height(M)$ (resp. $\height(M')$, $\height(N)$, and $\height(N')$) using \cref{prop:ht_altitude}. In this calculation, each term of $D_{G'}$ appears precisely once, and each term of $D_{G}$ appears precisely once, so this verifies $D_G$ is obtained from $D_{G'}$ by the substitutions and multiplication above. 

\begin{figure}
    \centering
    \includegraphics[width=\textwidth]{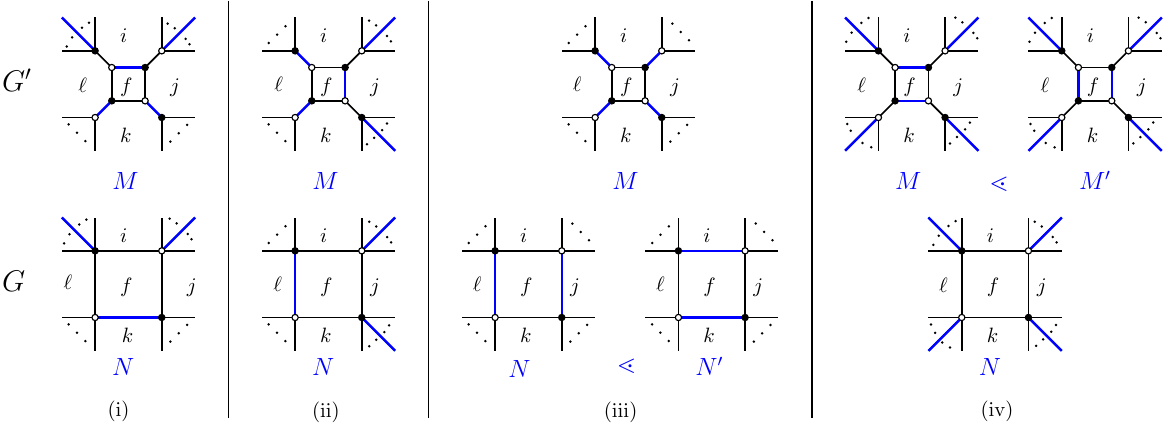}
    \caption{Correspondence between matchings of $G'$ (top) and matchings of $G$ (bottom) in the proof of \cref{thm:sq-move-step}. Matchings in the same column agree on the edges not pictured. For brevity, we do not picture the $180^\circ$ rotations of the first two columns.}
    \label{fig:sq-move-match-corresp}
\end{figure}

There are 16 cases to consider since, up to rotation, the bottom matching $\hat{0}$ of $G'$ can look locally around face $f$ like one of the four matchings illustrated in \cref{fig:sq-move-bottom-matching} (top row), and for each possible bottom matching, there are also four possible types of local configurations, up to rotation, for an arbitrary matching $M$, as illustrated in \cref{fig:sq-move-match-corresp} (top row).  Applying the square move to $\hat{0}$ of $G'$ yields the bottom matching $\hat{0}$ of $G$ (as well as another matching in case (c)), and also replaces $M$ (resp. $M$ and $M'$ in case (iv)) with $N$ (resp. $N$ and $N'$ in case (iii)), a matching of $G$.  The exponents of $y_f, y_i, y_j, y_k, y_\ell$ in $\height(N)$ are similarly computed using \cref{prop:ht_altitude}, and agree with the results of the aforementioned substitution and multiplication.

We begin by illustrating these arguments in the case when the bottom matching $\hat{0}$ of $G'$ looks locally like the configuration of case (a).

\noindent\textbf{Case (a), (i):} Considering \cref{fig:case-a}, if $M$ locally looks like case (i), then the local configuration around face $f$ involves doubled edges and there are walks from faces $i$, $j$, $k$, $\ell$ to $f$ which do not cross any edges of $\overrightarrow{M \triangle 0}$.  Thus, all 5 coordinates $\height(M)_i,\height(M)_j,\height(M)_k,\height(M)_\ell$, and $\height(M)_f$ are equal to each other.
Furthermore, for this example, $\hat{0}_G$ is in case (a) so the quantity \eqref{eq:multfactor} equals $1$, and under the specified substitution and multiplication, $\mathbf{y}^{\height(M)}$ is left invariant.

After applying the square move to get corresponding terms of $D_G$, we obtain $\overrightarrow{N \triangle 0}$ as in (i) of \cref{fig:case-a}.  It is evident that $\height(N)= \height(M)$ in this case, so $\mathbf{y}^{\height(N)} =  \mathbf{y}^{\height(M)}$, exactly as desired.

\begin{figure}[h]
\includegraphics[width=\textwidth]{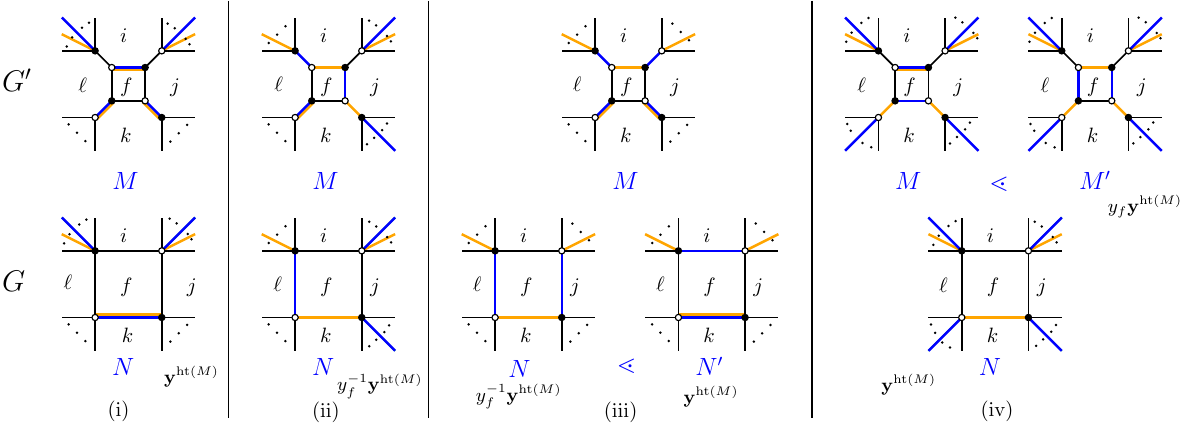}
\caption{\label{fig:case-a} Orange edges give the bottom matching of $G'$ (top) and $G$ (bottom) in case (a) from \cref{fig:sq-move-bottom-matching}. In blue, matchings of $G'$ (top) and the corresponding matchings of $G$ (bottom). Next to each matching different from $M$ is the corresponding term of the dimer polynomial, written in terms of $\mathbf{y}^{\height(M)}$.
}
\end{figure}

\noindent\textbf{Case (a), (ii):} We repeat the above analysis, continuing to assume that the bottom matching $\hat{0}$ of $G'$ is as in case (a), but letting $M$ be as in case (ii).  Then, a cycle in $\overrightarrow{M \triangle 0}$ separates faces $i$ and $j$ from the other three faces. Since the edge between $i$ and $f$ is in $\hat{0}$ and so oriented white-to-black in $\overrightarrow{M \triangle 0}$, $\height(M)_i= \height(M)_f -1$ by \cref{prop:ht_altitude}. We also have $\height(M)_j= \height(M)_f -1$ since there is a walk from $i$ to $j$ which passes through no edges of $\overrightarrow{M \triangle 0}$.
The other two coordinates $\height(M)_k,\height(M)_\ell$ are equal to $\height(M)_f$ because there are walks from faces $k$, $\ell$ to $f$ which do not cross any edges of $\overrightarrow{M \triangle 0}$.
Further, since $\hat{0}$ is in case (a), the quantity \eqref{eq:multfactor} equals $1$ just as above.

Under the desired substitutions, one may check that the binomial factors arising from $y_i'$ and $y_j'$ cancel each other out, as do those from $y_k'$ and $y_{\ell}'$, so we obtain 
\[y_f^r y_i^{r-1}y_j^{r-1}y_k^r y_\ell^r \mapsto 
y_f^{-r} ~ 
y_i^{r-1} \bigg(\frac{y_f}{1+y_f}\bigg)^{r-1} ~y_j^{r-1} (1+y_f)^{r-1} ~ 
y_k^r \bigg(\frac{y_f}{1+y_f}\bigg)^r ~ y_\ell^r (1+y_f)^r\]
\[ = y_f^{-1} (y_f^r y_i^{r-1}y_j^{r-1} y_k^r y_\ell^r).\]
Under the specified substitution and multiplication, $\mathbf{y}^{\height(M)} \mapsto y_f^{-1} \mathbf{y}^{\height(M)}$.

Applying the square move to get the corresponding matching $N$ of $G$, we see $\height(N)= \height(M)-\mathbf{e}_f$. In particular, the cycle in $\overrightarrow{M \triangle {0}}$ separating faces $i$, and $j$ from the rest of the local configuration now ``bends" around face $f$ in $\overrightarrow{N \triangle {0}}$ so that faces $i$, $j$ and $f$ are on the same side (see \cref{fig:case-a} (ii)).  
So $\mathbf{y}^{\height(N)}= y_f^{-1}\mathbf{y}^{\height(M)}$, exactly as desired. 

\noindent\textbf{Case (a), (iii):} If $M$ is as in case (iii), then a cycle in $\overrightarrow{M \triangle 0}$ separates $i$ from the other four faces. Since the edge between $i$ and $f$ is in $\hat{0}$ and so oriented white-to-black in $\overrightarrow{M \triangle 0}$, $\height(M)_i= \height(M)_f -1$ by \cref{prop:ht_altitude}. The other three coordinates $\height(M)_j,\height(M)_k,\height(M)_\ell$ are equal to $\height(M)_f$ because there are walks from faces $j$, $k$, $\ell$ to $f$ which do not cross any edges of $\overrightarrow{M \triangle 0}$.  In this case, under the substitutions above, the binomial factors do not fully cancel out each other, and one may check that 
\[y_f^r y_i^{r-1}y_j^{r}y_k^r y_\ell^r \mapsto 
y_f^{-r} ~ 
y_i^{r-1} \bigg(\frac{y_f}{1+y_f}\bigg)^{r-1} ~y_j^r (1+y_f)^r ~ 
y_k^r \bigg(\frac{y_f}{1+y_f}\bigg)^r ~ y_\ell^r (1+y_f)^r
= y_f^{r}y_f^{-1} (1+y_f) y_i^{r-1}y_j^r y_k^r y_\ell^r\]
\[= 
(1+y_f^{-1})(y_f^r y_i^{r-1}y_j^{r}y_k^r y_\ell^r). \]
Under the specified substitution and multiplication by \eqref{eq:multfactor} which equals $1$, we get $\mathbf{y}^{\height(M)} \mapsto \mathbf{y}^{\height(M)} (1+y_f^{-1})$.
We turn now to the corresponding terms of $D_G$. The corresponding matchings $N$ and $N'$ of $G$ are as in column (iii) of \cref{fig:case-a}. Then $\height(N')= \height(M)$ and $\height(N)=\height(N')-\mathbf{e}_f=\height(M)-\mathbf{e}_f$ since $N$ is obtained from $N'$ by a down-flip at face $f$ (cf. \cref{lem:ht_flip}). 
So $\mathbf{y}^{\height(N')} + \mathbf{y}^{\height(N)}= (1+y_f^{-1}) \mathbf{y}^{\height(M)}$, exactly as desired. 

\noindent\textbf{Case (a), (iv):} Finally, if $M$ and $M'$ are as in case (iv), then a cycle in $\overrightarrow{M \triangle 0}$ separates $k$ from the other four faces. Since the edge between $j$ and $k$ is in $\hat{0}$ and so oriented white-to-black in $\overrightarrow{M \triangle 0}$ and $\overrightarrow{M' \triangle 0}$, we have $\height(M)_k= \height(M)_f +1$  by \cref{prop:ht_altitude}. The other three coordinates $\height(M)_j,\height(M)_k,\height(M)_\ell$ are equal to $\height(M)_f$. Since $M'$ is obtained from $M$ by an up-flip at $f$, we have $\height(M')=\height(M) + \mathbf{e}_f$.

Again, under the substitutions above, the binomial factors do not fully cancel out each other, and one may check that $\mathbf{y}^{\height(M)}$ transforms as
\[y_f^r y_i^{r}y_j^{r}y_k^{r+1} y_\ell^r \mapsto 
y_f^{-r} ~ 
y_i^{r} \bigg(\frac{y_f}{1+y_f}\bigg)^{r} ~y_j^r (1+y_f)^r ~ 
y_k^{r+1} \bigg(\frac{y_f}{1+y_f}\bigg)^{r+1} ~ y_\ell^r (1+y_f)^r\]
\[= y_f(1+y_f)^{-1}(y_f^{r}y_i^{r}y_j^r y_k^{r+1} y_\ell^r) \]
while $\mathbf{y}^{\height(M')}$ transforms instead as
\[y_f^{r+1} y_i^{r}y_j^{r}y_k^{r+1} y_\ell^r \mapsto 
y_f^{-r-1} ~ 
y_i^{r} \bigg(\frac{y_f}{1+y_f}\bigg)^{r} ~y_j^r (1+y_f)^r ~ 
y_k^{r+1} \bigg(\frac{y_f}{1+y_f}\bigg)^{r+1} ~ y_\ell^r (1+y_f)^r\]
\[= (1+y_f)^{-1}(y_f^{r}  y_i^{r}y_j^r y_k^{r+1} y_\ell^r). \]
Under the specified substitution and multiplication by \eqref{eq:multfactor} which equals $1$, we get $\mathbf{y}^{\height(M)} + \mathbf{y}^{\height(M')} \mapsto 
\mathbf{y}^{\height(M)} y_f(1+y_f)^{-1} + 
\mathbf{y}^{\height(M)} (1+y_f)^{-1}
= \mathbf{y}^{\height{(M)}}$.
We turn now to the corresponding terms of $D_G$. The corresponding matching $N$ of $G$ is as in column (iv) of \cref{fig:case-a}. Then $\height(N)= \height(M)$, and so $\mathbf{y}^{\height(N)}= \mathbf{y}^{\height(M)}$, exactly as desired. 

If the bottom matching $\hat{0}$ is not as in case (a), we utilize variants of the above arguments.
For instance, if $\hat{0}$ is as in case (c), the analogous superpositions are illustrated in \cref{fig:case-c}.  The algebraic substitutions and height calculations proceed analogously to above. The only substantive difference in argument is that the multiplicative factor \eqref{eq:multfactor} is instead $(1+y_f)$ in this case.  This factor arises since the bottom matching $\hat{0}$ of $G'$ corresponds via square move to two (not one) matchings of $G$.  One of them is the corresponding bottom matching of $G$, this is the one illustrated in column (c) of \cref{fig:sq-move-bottom-matching}, but a $90^\circ$ rotation of this matching of $G$ also arises, and that corresponds a height of $y_f$ instead of $1$, the second term of \eqref{eq:multfactor} in this case.  The arguments for cases (b) and (d) are analogous to those for (a) and (c), although the multipicative factor \eqref{eq:multfactor} is $y_f(1+y_f)^{-1}$ in case (d) since two matchings of $G'$ correspond to the bottom matching of $G$.

\begin{figure}
\includegraphics[width=\textwidth]{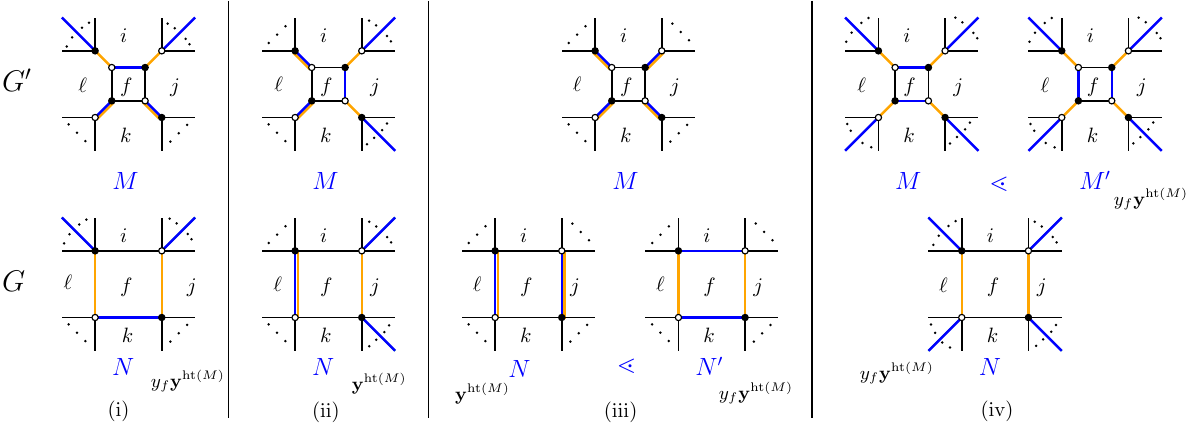}
\caption{\label{fig:case-c} Orange edges give the bottom matching of $G'$ (top) and $G$ (bottom) in case (c) from \cref{fig:sq-move-bottom-matching}. In blue, matchings of $G'$ (top) and the corresponding matchings of $G$ (bottom). Next to each matching different from $M$ is the corresponding term of the dimer polynomial, written in terms of $\mathbf{y}^{\height(M)}$.}

\end{figure}
\end{proof}

\begin{remark}\cref{thm:sq-move-step} asserts that under a square move, dimer face polynomial changes according to the substitutions in \eqref{eq:y-mut} followed by multiplication by \eqref{eq:multfactor}. Similar results have been proved for partition functions in various contexts. In the case of plabic graphs (see \cref{sec:positroid-var-applications}), Postnikov shows in \cite[Lemma 12.2]{Postnikov} that, under a square move, the images of the boundary measurement map are related by applying a substitution $z_e \mapsto z_e'$, which becomes precisely the substitution in \eqref{eq:y-mut} under the change of variables between $\{z_e\}$ and $\{y_f\}$ in \cref{rem:dimer-polyomial-edges}. Translating his results into the language of matchings (cf. the discussion of urban renewal in \cite[Section 3.4]{MS-twist}), for plabic graphs, under a square move, the partition functions for \emph{almost perfect matchings with a fixed boundary} are related by the substitution $z_e \mapsto z_e'$ in \cite[Figure 7]{MS-twist}, followed by simultaneous rescaling. There is an alternative indirect proof of \cref{thm:sq-move-step} using Postnikov's result and the transformations between partition functions and the dimer face polynomials in \cref{rem:dimer-polyomial-edges}. See also \cite[7.6]{marsh2016twists} where Marsh and Scott use shear weights and Pl\"ucker relations for the special case of reduced plabic graphs for Grassmannians.

A related argument appears in Speyer's work \cite[Sec 4.2]{speyer2007perfect} for the special case of graphs from the Octahedron recurrence and using a different definition of face weights.  Goncharov and Kenyon \cite[Theorem 4.7]{GK} approach a larger class of graphs associated to cluster integrable systems and use the Cluster-Poisson structure to demonstrate that under a square move, the partition function changes according to the substitutions $y_i \mapsto y_i'$. 
\end{remark}

We now examine the effect of bigon removal.

\begin{theorem}\label{thm:bigon-step}
       
        Suppose $G$ is a graph with property $(*)$ and $G'$ is obtained from $G$ by a bigon removal at face $f$.  Let $\mathbf{r}'$ be a reduction sequence for $G'$, and let $z'$ be as in \cref{thm:dimer-poly-is-F-poly}. The sequence $\mathbf{r}=(f, \mathbf{r}')$ is a reduction sequence for $G$; define $z$ analogously to $z'$, using $G$ instead. Let $\mathbf{g}:=\mathbf{g}^{Q_G}_z$ and $\mathbf{g}':=\mathbf{g}^{Q_{G'}}_{z'}$ denote the $g$-vector of $z$ and $z'$, respectively. If $F_{z'}^{Q_{G'}}= D_{G'}$ and $\mathbf{g}'= \hvec^{G'}_{\hat{0}}$, then $F_{z}^{Q_G}= D_G$ and $\mathbf{g}= \hvec^G_{\hat{0}}$.
        
\end{theorem}

\begin{proof}
Let $\Sigma=(\mathbf{x}, Q_G)$ and set $\Sigma'':=\mu_f(\Sigma)$. Note that $z$ is the cluster variable of $\mu_{\mathbf{r}}(\Sigma'')$ indexed by the final mutation in the reduction sequence. We have that $F_{z}^{\Sigma}$ and $F_{z}^{\Sigma''}$ are related as in \cref{prop:adjacent-F}.

By \cref{lem:moves-effect-on-Q}, the quiver $Q_{G'}$ differs from $Q''=\mu_f(Q_G)$ by deleting vertex $f$. In other words, $Q_{G'}$ is the induced subquiver of $Q''$ on all vertices besides $f$. Applying \cref{prop:F-g-adding-mutable} with $w=z'$ and $w'=z$, we have that $F_{z}^{\Sigma''}$ is equal to $F_{z'}^{Q_{G'}}$, and the $g$-vectors $\mathbf{g}'=\mathbf{g}_{z'}^{Q_{G'}}$ and $\mathbf{g}'':=\mathbf{g}_z^{\Sigma''}$ agree in all coordinates except the coordinate $g''_f$ of the latter vector.

\begin{figure}
    \centering
   \includegraphics[width=0.7 \textwidth]{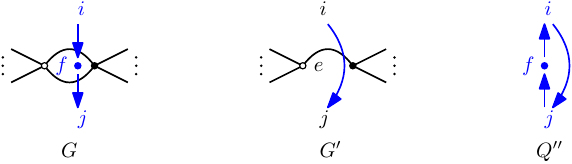}
    \caption{Left: The graph $G$ near the face $f$ and the arrows of $Q_G$ adjacent to $f$ in the proof of \cref{thm:bigon-step}. Center: The graph $G'$ and the arrow contributed by the faces $i$ and $j$. Right: The arrows of $\mu_f(Q_G)=Q''$ that differ from those of $Q_G$.
} 
    \label{fig:bigon-scenario}
\end{figure}

We first use the equalities $F_{z}^{\Sigma''}=F_{z'}^{Q_{G'}}= D_{G'}$ together with some properties of $F$-polynomials to deduce the value of $g_f''$. Suppose the faces adjacent to $f$ in $G$ are $i, j$, and let $e$ be the edge of $f$ which is not deleted in the bigon removal, so $e$ is an edge shared by $i,j$ in $G'$. Say $e$ is white-black in $j$ and black-white in $i$. See \cref{fig:bigon-scenario}. We claim
\begin{equation}\label{eq:g-for-bigon-removal}g_f''= \begin{cases}
    0 & \text{if } e \notin \hat{0}_{G'}\\
    1 & \text{if }e \in \hat{0}_{G'}.
\end{cases}\end{equation}

By \cref{thm:g-F-cluster-expansion}, a Laurent polynomial formula for $z$ in terms of the cluster $\mathbf{u}$ of $\Sigma''$ is 
\[F_z^{\Sigma''}(\hat{y}_1, \dots, \hat{y}_n) \cdot \mathbf{u}^{\mathbf{g}''}.\]
If you put this expression over a common denominator and write it in lowest terms, you obtain the expression from \cref{thm:cluster-Laurent-positivity-denom} (1). In particular, by \cref{thm:cluster-Laurent-positivity-denom} (2), the resulting expression does not have $u_f$ as a factor of the denominator, as $z$ is compatible with $u_f$. It also does not have $u_f$ as a factor of the numerator. Because of the arrow configuration in $Q''$ (see \cref{fig:bigon-scenario}, right) the variable $u_f$ is in the numerator of $\hat{y}_i$ and in the denominator of $\hat{y}_j.$ We next determine how $y_i$ and $y_j$ appear in terms of $F_z^{\Sigma''}$.

Suppose that $e \notin \hat{0}_{G'}$. Then for any term $\mathbf{y}^{\height(M)}$ of $D_{G'}= F_{z}^{\Sigma''}$, either $\height(M)_i = \height(M)_j$ or $\height(M)_i = \height(M)_j +1$. Indeed, as $i, j$ are adjacent faces, there is at most one cycle $C$ in $\overrightarrow{M \triangle {0}}$ which encircles one of $i,j$ but not the other. If there is such a cycle, then $C$ must contain $e$. This implies $e \in M$, so $e$ is oriented black-to-white. This means that $e$ is left-to-right with respect to the dual edge of the dual graph $G^*$ oriented from $i$ to $j$. Using \cref{prop:ht_altitude}, this implies $\height(M)_i = \height(M)_j +1$. A similar argument shows that if $e \in \hat{0}_{G'}$, then in any term $\mathbf{y}^{\height(M)}$ of $D_{G'}$, either $\height(M)_i = \height(M)_j$ or $\height(M)_i = \height(M)_j -1$. Using this, we will determine how $u_f$ appears in terms of $F_z^{\Sigma''}(\hat{y}_1, \dots, \hat{y}_n)$.

If $e \notin \hat{0}_{G'}$, the above paragraph implies $u_f$ only appears in numerators of terms in $F_z^{\Sigma''}(\hat{y}_1, \dots, \hat{y}_n)$. When we write $F_z^{\Sigma''}(\hat{y}_1, \dots, \hat{y}_n)$ over a (least) common denominator, $w_f$ still appears only in the numerator, and does not appear in all terms since $F$-polynomials have constant term 1. Thus, $g_f''=0$; otherwise, $u_f$ would be a factor of either the numerator or denominator of $z$.

If $e \in \hat{0}_{G'}$, then $w_f$ appears with exponent $0$ or $-1$ in terms of $F_z^{\Sigma''}(\hat{y}_1, \dots, \hat{y}_n)$. There is at least one term in which $u_f$ appears with exponent $-1$: there is some matching $M$ of $G'$ which uses an edge adjacent to $e$. Thus, $e$ is in a cycle of $\overrightarrow{M \triangle {0}}$, and is oriented white-to-black in this cycle. Using \cref{prop:ht_altitude}, we have $\height(M)_i = \height(M)_j-1$, so $u_f$ appears with exponent $-1$ in this term. Thus, when $F_z^{\Sigma''}(\hat{y}_1, \dots, \hat{y}_n)$ is written over a least common denominator, $u_f$ is a factor of the denominator, $u_f^2$ is not, and $u_f$ is not a factor of the numerator. This implies that $g_f''=1$; otherwise, $u_f$ would be a factor of the denominator or the numerator of $z$. This shows \eqref{eq:g-for-bigon-removal}.

\begin{figure}
    \centering
    \includegraphics[width=0.9\textwidth]{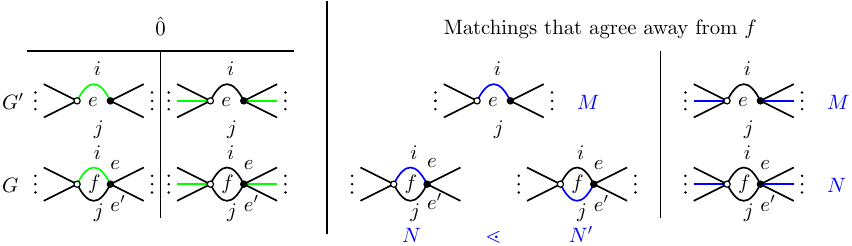}
    \caption{Left: The two possibilities for bottom matchings of $G'$ and $G$. Right: Correspondence between matchings of $G'$ (top) and $G$ (bottom). Matchings in the same column agree on the edges not pictured.}
    \label{fig:bigon-bottom}
\end{figure}

Now, we show that $\mathbf{g}=\hvec^G_{\hat{0}}$. All coordinates of $\mathbf{g}''$ which are not indexed by $f$ agree with the coordinates of $\mathbf{g}'=\hvec^{G'}_{\hat{0}}$. Also, $\mathbf{g}, \mathbf{g}''$ are related by \eqref{eq:g-mut}. The two possible bottom matchings of $G'$ and $G$ are given in \cref{fig:bigon-bottom}, left. It is easy to see that if $e \notin \hat{0}_{G'}$, then $h^G_f=0$ and all other coordinates of $\hvec^G_{\hat{0}}$ are the same as those of $\hvec^{G'}_{\hat{0}}$. This agrees with $\mathbf{g}$, computed using \eqref{eq:g-mut} with $\mathbf{g}$ taking the role of $\mathbf{g}'$, $\mathbf{g}''$ taking the role of $\mathbf{g}$, and $q_{ab}$ computed using $Q''$.
If $e \in \hat{0}_{G'}$, then $h^G_f=-1$, which is equal to $g_f$. We also have that $h^G_{j}=h^{G'}_j+1$, since in $G$, one fewer edge of $j$ is in the bottom matching. All other coordinates of $\hvec^G_{\hat{0}}$ and $\hvec^{G'}_{\hat{0}}$ agree. This again agrees with $\mathbf{g}$, computed using \eqref{eq:g-mut}.

Finally, we show that $F_z^{Q_G}=D_G$, using \cref{prop:adjacent-F} together with the assumption that $F_z^{\Sigma''}=F_{z'}^{Q_{G'}}= D_{G'}$. By \cref{prop:adjacent-F}, $F_z^{\Sigma}$ is obtained from $F_z^{\Sigma''}$ by the substitutions
\begin{equation}\label{eq:bigon-sub}
    y_f \mapsto y_f^{-1}, \quad y_i \mapsto y_i(1+y_f), \quad y_j \mapsto y_j (1+y_f^{-1})^{-1}= y_j \frac{y_f}{1+y_f}
\end{equation}
followed by multiplication by 
\[\begin{cases}
    1 & \text{ if }g_f=0\\
   {1+y_f} & \text{ if }g_f=-1.\\
\end{cases}\]
So it suffices to show that $D_G$ can be obtained from $D_{G'}$ by the same substitutions and multiplication.

As in the proof of \cref{thm:sq-move-step}, for each matching $M$ of $G'$ there are one or two matchings $N, N'$ of $G$ which agree with $M$ away from $f$. Distinct matchings of $G'$ correspond to distinct matchings of $G$, and every matching of $G$ corresponds to some matching of $G'$. See \cref{fig:bigon-bottom}, right. We will compare the term $\mathbf{y}^{\height(M)}$ in $D_{G'}$ to $\mathbf{y}^{\height(N)}$ or $\mathbf{y}^{\height(N)}+\mathbf{y}^{\height(N')}$ in $D_G$ to verify they are related by \eqref{eq:bigon-sub}
and the appropriate multiplication. 

We consider the case then the bottom matching of $G'$ uses $e$ (\cref{fig:bigon-bottom}, far left) and $g_f=-1$, as the other case is similar.

Suppose a matching $M$ of $G'$ does not use $e$, so $\overrightarrow{M \triangle {0}}$ has a cycle using $e$. As argued previously, this implies that $r:=\height(M)_j = \height(M)_i +1$. When the edge $e'$ is added, every cycle of $\overrightarrow{M \triangle {0}}$ which encircles $j$ becomes a cycle of $\overrightarrow{N \triangle {0}}$ which encircles $f$ and $j$. So $\mathbf{y}^{\height(N)}= y_f^{r+1} \mathbf{y}^{\height(M)}$. On the other hand, under \eqref{eq:bigon-sub}, $y_i^{r}y_j^{r+1} \mapsto (y_i^r y_j^{r+1}) y_f^{r+1}(1+y_f)^{-1}$, so under \eqref{eq:bigon-sub} and multiplication by $1+y_f$, $\mathbf{y}^{\height(M)} \mapsto \mathbf{y}^{\height(N)}$.

If $M$ does use $e$, then $\height(M)_i=\height(M)_j=:r$. We have $$\mathbf{y}^{\height(N')} + \mathbf{y}^{\height(N)}= y_f^r \mathbf{y}^{\height(M)} + y_f^{r+1} \mathbf{y}^{\height(M)}= (1+y_f) y_f^r \mathbf{y}^{\height(M)}.$$
On the other hand, under \eqref{eq:bigon-sub}, $y_i^r y_j^r \mapsto (y_i^r y_j^r) y_f^r$, so after multiplication by $1+y_f$, we have $\mathbf{y}^{\height(M)} \mapsto \mathbf{y}^{\height(N')} + \mathbf{y}^{\height(N)}.$ 
\end{proof}

\subsection{Proofs of \cref{thm:dimer-poly-is-F-poly,cor:easy-cluster-expansion}}
\begin{proof}[Proof of \cref{thm:dimer-poly-is-F-poly}]
We now prove \cref{thm:dimer-poly-is-F-poly}. By \cref{lem:e-moves-dont-change-dimers}, (E) moves do not change $D_G$; by \cref{lem:moves-effect-on-Q} (E) moves do not change $Q_G$ and thus do not change any $F$-polynomials. So we may assume, by replacing $G$ with another graph related by (E) moves, that one may apply a move at $f_1$ in $G$.

We proceed by induction on the length of the reduction sequence $\mathbf{r}$. If $\mathbf{r}=(f_1)$ consists of a single face, then $G$ is a bigon, and $Q_G$ is a single vertex. One can verify that $D_G=1+y_f= F_z^{Q_G}$, and that $\hvec^G_{\hat{0}}=(-1)=\mathbf{g}$.

Now, suppose $\mathbf{r}= (f_1, \mathbf{r}')$ is length at least 2. Let $G'$ be the graph obtained from $G$ by performing the appropriate move at $f_1$; it has reduction sequence $\mathbf{r}'$, of shorter length. If $f_1$ is a square face, then the inductive hypothesis together with \cref{thm:sq-move-step} (with the roles of $G$ and $G'$ swapped) imply that $\mathbf{g}=\hvec^G_{\hat{0}}$ and $D_G=F_z^{Q_G}$. If $f_1$ is a bigon, then the inductive hypothesis together with \cref{thm:bigon-step} imply that $\mathbf{g}=\hvec^G_{\hat{0}}$ and $D_G=F_z^{Q_G}$. This completes the proof.
\end{proof}

\begin{proof}[Proof of \cref{cor:easy-cluster-expansion}]
Let $\mathbf{g}$ and $F$ respectively denote the $g$-vector and $F$-polynomial of $z$ with respect to $\Sigma$. Recall from \cref{def:exchange-ratio} that for each face $a \in \faces(G)$, or equivalently every vertex $a$ of the quiver, we define the exchange ratio
\[\hat{y}_a:= \frac{\prod_{f \to a} x_f}{\prod_{a \to f} x_f}.\]

By \cref{thm:dimer-poly-is-F-poly}, we know that 
$\mathbf{x}^{\mathbf{g}}=\mathbf{x}^{\hvec^G_{\hat{0}}}.$

So, appealing to \cref{thm:g-F-cluster-expansion}, to show the desired cluster expansion, it suffices to show that 
\[D_G(\hat{y}_f) = \sum_{M \in \mathcal{D}_G} \mathbf{x}^{\hvec^G_M -\hvec^G_{\hat{0}}}.\]
We argue equality term by term, traveling up the dimer lattice. The term indexed by $\hat{0}$ on the right and left sides is 1. Assume we have equality for the term indexed by $N$ and suppose $N \lessdot M$ are related by an up-flip at face $a$. On the left-hand side, we multiply the term indexed by $N$ by $\hat{y}_a$. On the right-hand side, we multiply by
\[\mathbf{x}^{\hvec^G_N -\hvec^G_M}= \prod_{f \in \faces(G)} x_f^{|M \cap f| - |N \cap f|}.\]
Every white-black edge of $a$ which is in $f$ contributes $1$ to the power of $x_f$, while every black-white edge of $a$ which is in $f$ contributes $-1$ to the power of $x_f$. Since each white-black edge of $a$ which is in $f$ also contributes an arrow $f \to a$, and each black-white edge contributes an arrow $a \to f$, we see that 
\[\prod_{f \in \faces(G)} x_f^{|M \cap f| - |N \cap f|} = \hat{y}_a.\] This shows the terms indexed by $M$ on both sides are equal.

For the denominator vector statement, notice that $|f|/2-|M \cap f| -1 \geq -1$, and equality occurs precisely when $f$ can be flipped in $M$. So $x_f$ appears in the term $\mathbf{x}^{\hvec^G_M}$ with negative exponent if and only if $f$ can be flipped in matching $M$, in which case the exponent of $x_f$ is $-1$. By \cref{cor:prop*=all-faces-flipped}, every face can be flipped in some matching, so every $x_f$ appears in some term with exponent $-1$. So 
\[z=\frac{P(\mathbf{x})}{\prod_{f \in \faces(G)} x_f}\]
where $P(\mathbf{x})$ is not divisible by any $x_f$. This gives the desired statement about the denominator vector.
\end{proof}

\section{Further cluster algebra applications to links} 
In this section, we focus on a link diagram $L$ and the corresponding graphs $G_{L,i}$. We first show that all of the dimer face polynomials $\{D_{G_{L,i}}\}_i$ are $F$-polynomials in a single cluster algebra. We then discuss the relationship between certain moves, which appeared recently in \cite{BMS24} to prove a stronger statement, and the graph moves in \cref{def:moves}. Finally, we focus on 2-bridge links, whose Alexander polynomials were previously connected to type $A$ $F$-polynomials and were computed using matchings of \emph{snake graphs} \cite[Section 8]{B21} based on earlier work on snake graphs and continued fractions \cite{CS16} on the one hand, and on cluster algebras and Jones polynomials \cite[Sections 5 and 6]{LS-JonesPoly} on the other. 
In this case, we show that $G_{L,i}$ is not a snake graph, but does have the same dimer lattice.

\label{sec:further-applications-to-links}

\subsection{All multivariate Alexander polynomials in one cluster algebra}

In \cref{sec:alexander-poly-and-dimers}, we showed that the Alexander polynomial of a link with diagram $L$ can be obtained by choosing any segment $i$ of $L$ and specializing the polynomial $D_{G_{L,i}}$. The graphs $G_{L,i}$ satisfy property $(*)$, so \cref{thm:dimer-poly-is-F-poly} shows that $D_{G_{L,i}}$ is an $F$-polynomial for the cluster algebra $\mathcal{A}(\mathbf{x}, Q_{G_{L,i}})$. In this section, we show that all of the polynomials $\{D_{G_{L,i}}\}_i$ are $F$-polynomials of a single larger cluster algebra. This was conjectured in \cite{B21}. It was also proved, using categorical techniques, in \cite{BMS24}, 

while we were 
in the final stages of preparing this manuscript.

We first need a slightly extended notion of dual quiver $Q_G$.

\begin{definition}\label{def:extended-quiver}
    Let $G$ be a bipartite plane graph. The \emph{extended dual quiver} $\extQ_G$ is defined according to the same procedure as in \cref{def:dual-quiver}, except there is also a mutable vertex in the infinite face, and arrows are drawn across every edge separating two distinct faces.
\end{definition}

\begin{theorem}\label{thm:cluster-stuff-for-link-diag}
    
    Let $L$ be a connected prime-like link diagram with no nugatory crossings and $G_L$ the face-crossing incidence graph. For a segment $i$, let $\hat{0}_{L,i}$ denote the bottom matching of the graph $G_{L,i}$ and let $f_i$ denote the face containing $i$. Then for every segment $i$, there is a cluster variable $z_i$ in the cluster algebra $ \mathcal{A}(\Sigma_0)=\mathcal{A}(\mathbf{x}, \extQ_{G_L})$ such that 
    \[F^{\Sigma_0}_{z_i}(\mathbf{y})= D_{G_{L,i}}(\mathbf{y}).\]
    Further,
    \begin{enumerate}
        \item the $g$-vector of $z_i$ is equal to $\hvec^G_{\hat{0}_{L,i}}-\mathbf{e}_{f_i}$ and so has entries
        
    \[g_f= \begin{cases}
       0 & \text{if }f=f_i\\
       |f \cap \hat{0}_{L,i}|+1 & \text{ else};
   \end{cases}\]
    \item the formula for $z_i$ in terms of the initial cluster is
 
    \[z_i = \sum_{M \in \mathcal{D}_{G_{L, i}}} \mathbf{x}^{\hvec^G_M- \mathbf{e}_{f_i}}\]
    \item the denominator vector of $z_i$ has entries 
    \[d_f =\begin{cases}
        0 & \text{ if }f \text{ shares a white vertex with }f_i\\
        1 & \text{ else.} 
    \end{cases} \]
    \end{enumerate}
\end{theorem}

We note that the existence of the cluster variables $z_i$ and the formula for their denominator vectors were also obtained in \cite{BMS24}. For an illustration of \cref{thm:cluster-stuff-for-link-diag}, see \cref{fig:2-dimer-poly-for-whitehead}.

\begin{figure}
\centering
\includegraphics[width=0.9\textwidth]{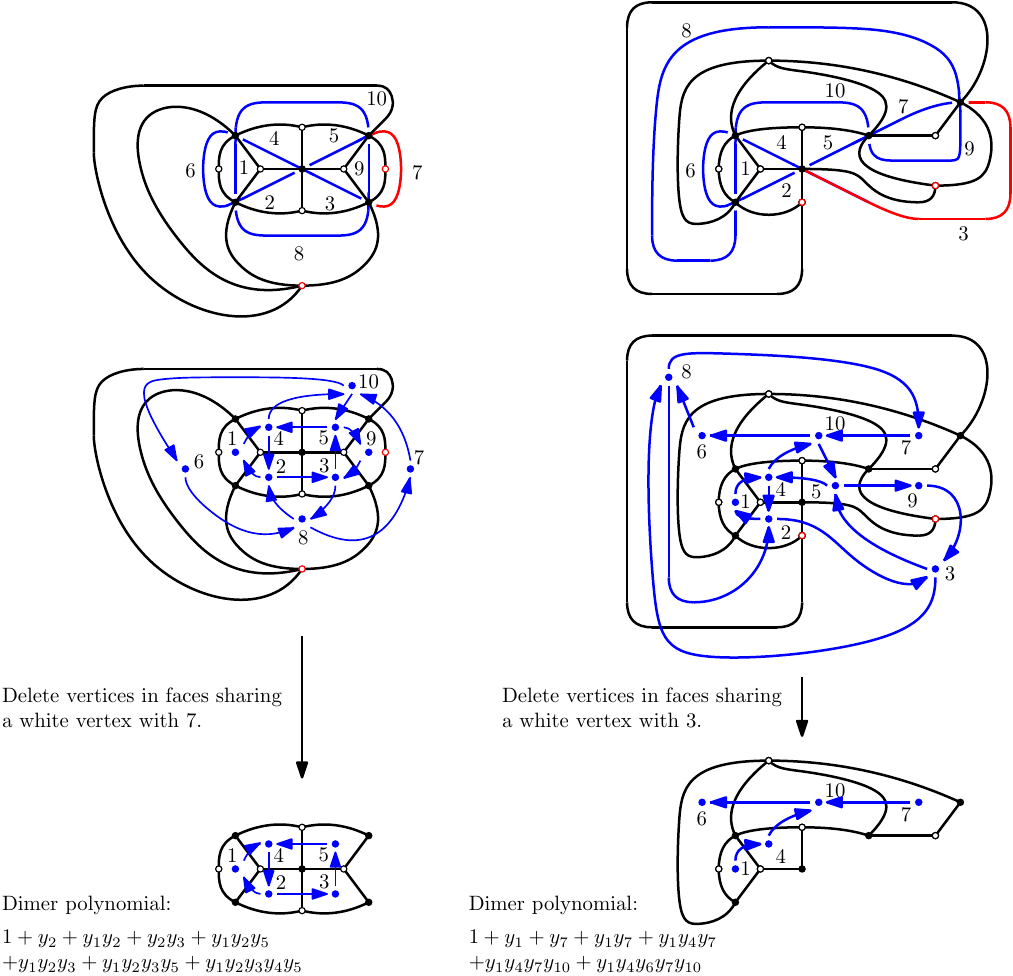}
\caption{Top: two different stereographic projections of the link diagram $L$ and graph $G_L$ from \cref{fig:whitehead_graph}. On the left, the infinite face of $G_L$ is $f_7$. On the right, it is $f_3$. Center: the corresponding stereographic projections of the quiver $\extQ_{G_L}$. Bottom: The graphs $G_{L,i}$ for $i=7$ (left) and $i=3$ (right), their dual quivers, and their dimer polynomials, which are $F$-polynomials of $\mathcal{A}(\extQ_{G_L})$ by \cref{thm:cluster-stuff-for-link-diag}. By \cref{thm:alexdimer}, both polynomials specialize to the Alexander polynomial of $L$.}
\label{fig:2-dimer-poly-for-whitehead}
\end{figure}

\begin{proof} Choose a segment $i$ of $L$. 

    For the existence of the cluster variable $z_i$, note that $G_L$ and $\extQ_{G_L}$ may both be drawn on the sphere. Stereographically project both $G_L$ and $\extQ_{G_L}$ to the plane so that the face $f_i$ is the infinite face. Deleting the white vertices of $f_i$ in $G_L$ yields the graph $G_{L,i}$. Deleting the vertices of $\extQ_G$ labeled by faces sharing a white vertex with $f_i$ yields the quiver $Q_{G_{L,i}}$. That is, the quiver $Q_{G_{L,i}}$ is an induced subquiver of $Q_{G_L}$. By \cref{thm:dimer-poly-is-F-poly}, $D_{G_{L,i}}$ is an $F$-polynomial in the cluster algebra $\mathcal{A}(Q_{G_{L,i}})$, obtained from the initial seed by mutation sequence $\mathbf{r}$. By \cref{prop:F-g-adding-mutable}, there is a cluster variable $z_i$ of $\mathcal{A}(\mathbf{x}, \mathbf{y}, \extQ_{G_L})$, again obtained by the mutation sequence $\mathbf{r}$, whose $F$-polynomial is $D_{G_{L,i}}$.

    Since $Q_{G_{L_i}}$ omits face $f_i$ and its neighbors, mutation sequence $\mathbf{r}$ cannot include these, hence 

    $z_i$ is compatible with the initial cluster variables
  \[
  \{x_f : f=f_i \text{ or } f \text{ adjacent to } f_i \}. 
  \]
    This, together with \cref{cor:easy-cluster-expansion}, also implies that the denominator vector of $z_i$ is as claimed in (3).

    We prove (1) and (2) together. First, for all faces $f$ of $G_{L}$ which are also faces of $G_{L,i}$, \cref{thm:dimer-poly-is-F-poly} and \cref{prop:F-g-adding-mutable} implies that 
    \[g_f= |f|/2 - |f \cap \hat{0}_{L,i}|-1.\]
    This is equal to the coordinate $h_f$ of $\hvec^G_{\hat{0}_{L,i}}$ by definition. So we only need to compute $g_{f}$ for $f$ containing a white vertex of $f_i$.
    Second, an identical proof as the proof of \cref{cor:easy-cluster-expansion} shows that 
    \[D_{G_{L,i}}(\hat{y}_f)= \sum_{M \in \mathcal{D}_{G_{L, i}}} \mathbf{x}^{\hvec^G_M-\hvec^G_{\hat{0}_{L,i}}}.
    \]
    We will utilize \cref{thm:cluster-Laurent-positivity-denom,thm:g-F-cluster-expansion} to determine the remaining entries of the $g$-vector. In particular, we use that
    \begin{equation}\label{eq:hard-cluster-expansion}
    \mathbf{x}^{\mathbf{g}} D_{G_{L,i}}(\hat{y}_f)=\frac{P_{z_i}^{\Sigma_0}(\mathbf{x})}{\prod_{f \in \faces(G_{L})} x_f^{d_f}}\end{equation}
    where $P_{z_i}^{\Sigma_0}(\mathbf{x})$ is as in \cref{thm:cluster-Laurent-positivity-denom}, and in particular is a polynomial which has no $x_f$ as a factor.

    Suppose $f$ is a face containing a white vertex of $f_i$. When we write $D_{G_{L,i}}(\hat{y}_f)$ as a rational expression in lowest possible terms, it is possible that it will have a factor of $x_f$. 
    The exponent of $x_f$ is 
    \begin{equation*}
        \min_{M \in \mathcal{D}_{G_{L,i}}} |\hat{0}_{L,i} \cap f| - |M \cap f|= |\hat{0}_{L,i} \cap f| - \max_{M \in \mathcal{D}_{G_{L,i}}} |M \cap f|.
    \end{equation*}
    Because $d_f=0$, \eqref{eq:hard-cluster-expansion} implies that multiplying by $x^{\mathbf{g}}$ must exactly cancel this factor of $x_f$. That is, 
    \begin{equation}\label{eq:obvious-eq-for-g}
        g_f = -|\hat{0}_{L,i} \cap f| + \max_{M \in \mathcal{D}_{G_{L,i}}} |M \cap f|.
    \end{equation}

    If $f=f_i$, then no edges of $f$ are in $G_{L,i}$, so \eqref{eq:obvious-eq-for-g} implies that $g_f=0$. In this case, the corresponding coordinate $h_f$ of $\hvec^G_{\hat{0}_{L,i}}$ is $h_f=|f|/2 -0 -1=1$. So $h_{f_i}-1=g_{f_i}$, as desired.

    If $f$ contains a white vertex of $f_i$ but is not equal to $f_i$, then the other white vertex of $f$, say $v$, is not in $f_i$. Otherwise, one can show that $L$ is not prime-like. 
    So exactly two edges of $f$, the ones adjacent to $v$, are in $G_{L,i}$. Since every edge of $G_{L_i}$ is in some matching, there is some matching that contains one of these edges; since these edges are adjacent, no matching contains both. Thus 
    \[\max_{M \in \mathcal{D}_{G_{L,i}}} |M \cap f|=1 \quad \text{ and } \quad g_f = 1-|\hat{0}_{L,i} \cap f|= |f|/2 - |\hat{0}_{L,i} \cap f| -1 = h_f.\] 

    This completes the proof of (1). Item (2) now follows by multiplying $D_{G_{L,i}}(\hat{y}_f)$ by $\mathbf{x}^\mathbf{g}$.
\end{proof}

Combining \cref{thm:alexdimer} and \cref{thm:cluster-stuff-for-link-diag}, we obtain the following corollary.

\begin{corollary}\label{cor:alex-poly-specialization-F-poly}
Let $L$ be a connected prime-like link diagram with no nugatory crossings and with segments numbered $1, \dots, s$. Let $G_L$ be the face-crossing incidence graph, set $Q_L:=\extQ_{G_L}$ and for any segment $i$ of $L$, let $z_i$ be as in \cref{thm:cluster-stuff-for-link-diag}. Then the Alexander polynomial $\Delta_L$ of $L$ can be obtained from any one of the $F$-polynomials $F^{Q_L}_{z_1}(\mathbf{y}), \dots,F^{Q_L}_{z_s}(\mathbf{y})$ in $\mathcal{A}(Q_L)$ using the specialization of \cref{thm:alexdimer}.
\end{corollary}

\subsection{Comparison of moves}

In \cite{BMS24}, Bazier-Matte and Schiffler show that all cluster variables $z_i$ from \cref{thm:cluster-stuff-for-link-diag} are in fact compatible, using an explicit mutation sequence from the initial quiver\footnote{Technically their initial quiver is $\extQ_{G_L}$ with the arrows reversed. They also show that the mutation sequence to the seed with cluster $\{z_i\}$ returns to the opposite of the initial quiver.} $\extQ_{G_L}$. Their mutation sequence involves two moves on the link diagram: a ``diagram Reidemeister III" move and a ``diagram bigon reduction" move, not to be confused with the graph bigon removal moves from \cref{def:moves}! In this section, we summarize how to reinterpret these moves as compositions of bigon removals and square moves in the face-incidence graph $G_L$.

See the top rows of \cref{fig:diagram-bigon-in-graphs,fig:diagram-R3-in-graphs} for a depiction of diagram bigon reduction and diagram Reidemeister III moves. We emphasize that these are operations on a link diagram without crossing data, as they ignore the crossing information. In \cite{BMS24}, diagram bigon reduction corresponds to the mutation sequence $a, b$, where segments $a,b$ are as in \cref{fig:diagram-bigon-in-graphs}. Diagram RIII corresponds to mutation sequence $a,b,c,a$, where segments $a,b,c$ are as in \cref{fig:diagram-R3-in-graphs}.

The bottom rows of \cref{fig:diagram-bigon-in-graphs,fig:diagram-R3-in-graphs} show the effect of the diagram moves on the face-crossing incidence graph $G_L$. Each diagram move corresponds to a sequence of (B) or (S) moves. The sequence of moves gives exactly the mutation sequence that \cite{BMS24} assign to each diagram move. In summary, \cref{fig:diagram-bigon-in-graphs,fig:diagram-R3-in-graphs} shows that the mutation sequence of \cite{BMS24} can be interpreted in terms of $G_L$, the face-incidence graph of the link diagram.

\begin{figure}[h]
    \centering
    \includegraphics[width=0.7\textwidth]{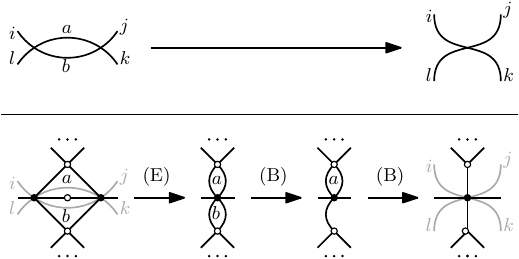}
    \caption{Top: diagram bigon reduction on a link diagram $L$. Bottom: the effect of diagram bigon reduction on the graph $G_L$, which is a sequence of bigon removals.}
    \label{fig:diagram-bigon-in-graphs}
\end{figure}

\begin{figure}[h]
    \centering
    \includegraphics[width=\textwidth]{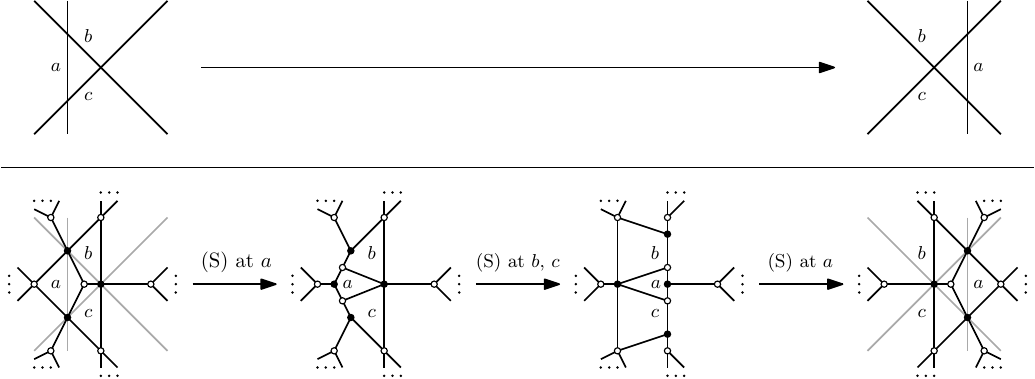}
    \caption{Top: diagram Reidemeister III on a link diagram $L$. Bottom: the effect of diagram Reidemeister III on the graph $G_L$. Assuming that there are no nugatory crossings either before or after the move, all faces of $G_L$ depicted here are square, and the move can be written as a sequence of four square moves. This sequence of four square moves is the ``$Y$-$\Delta$" move of Goncharov and Kenyon \cite[Figure 30]{GK}.} 
\label{fig:diagram-R3-in-graphs}    
\end{figure}

\subsection{Example: 2-bridge links, snake graphs, and type $A$ $F$-polynomials.} 
\label{sec:2_bridge}

In this section, we focus on the special case of 2-bridge links. A number of papers obtained the Alexander polynomials of 2-bridge links by specializing type $A$ $F$-polynomials \cite{B21,Wataru-Yuji,LS-JonesPoly}. By \cite{mussch,MSW}, type $A$ $F$-polynomials are dimer face polynomials for \emph{snake graphs}. 
    (Since snake graphs have property $(*)$ and exhibit type $A$ Dynkin quivers as their dual quivers, see \cref{lem:snake-quiver}, this identification of their dimer face polynomials as type $A$ $F$-polynomials can also be deduced from \cref{thm:dimer-poly-is-F-poly}. 
See also \cite{rabideau2018f}, which provides $F$-polynomials for such snake graphs in the context of their connection to continued fractions.) This also implies that the Alexander polynomials of 2-bridge links are specializations of snake graph dimer face polynomials. Here, we give an alternate proof of these specializations,
 and to do so we first consider a different but related family of graphs, which arise as face-crossing incidence graphs of 2-bridge links. 
In particular, when $L$ is (the standard regular diagram of) a 2-bridge link, and $i$ is a distinguished segment we call the \textit{lower segment} of $L$, we give a concrete description of $G_{L,i}$ and show that $Q_{G_{L,i}}$ is a type $A$ quiver. The graphs $G_{L,i}$ have previously appeared in work of Propp \cite[Section 4]{P05}. We recall Propp's proof that $G_{L,i}$ has the same dimer lattice, and thus dimer face polynomial, as a snake graph.

We define 2-bridge links using their Conway notation, following \cite{M96}. 

\begin{definition}[\cite{M96}]\label{def:standard_regular}
    Let $\alpha_1,\dots,\alpha_m$ be positive integers. The diagram $C(\alpha_1,\dots,\alpha_m)$ is the alternating diagram of the form shown in \cref{fig:standard-regular-2-bridge}.
    \begin{figure}
        \includegraphics[width=0.7\textwidth]{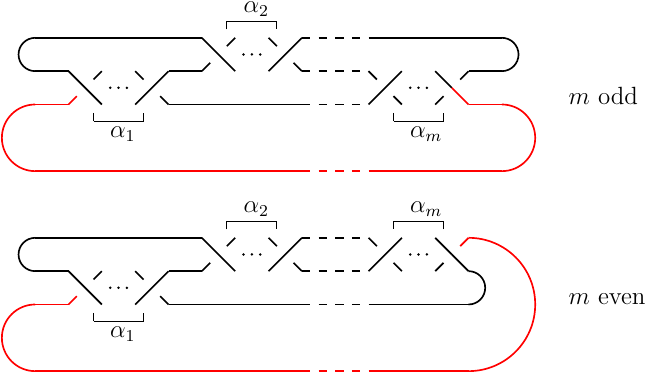}
        \caption{The diagram $C(\alpha_1,\dots,\alpha_m)$ of a 2-bridge link. The lower segment is in red. Note that the crossing information depends on the parity of each $\alpha_i$.}
        \label{fig:standard-regular-2-bridge}
    \end{figure}
    The red highlighted segment in \cref{fig:standard-regular-2-bridge} is the \textbf{lower segment} of the standard regular diagram. A link is a \textbf{2-bridge link} if it has a diagram of the form $C(\alpha_1,\dots,\alpha_m)$. 
    \end{definition}

For $L=C(\alpha_1,\dots,\alpha_m)$ and $i$ the lower segment, the graph $G_{L,i}$ is easy to describe. It is obtained by gluing together subgraphs of the following type.

\begin{definition}
Let $k$ be a positive integer. The \textbf{downward $k$-birdwing} is the bipartite plane graph shown in \cref{fig:birdwing}. The \textbf{upward $k$-birdwing} is the downward $k$-birdwing reflected across the horizontal axis. The downward $0$-birdwing is a single vertical edge, with white vertex on top.
\begin{figure}[h]
        \includegraphics[width=0.15\textwidth]{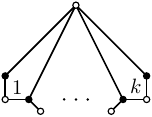}
        \caption{A downward $k$-birdwing.}
        \label{fig:birdwing}
    \end{figure}
\end{definition}

\begin{lemma}\label{lem:2-bridge-gluing-birdwing}
      For $L=C(\alpha_1,\dots,\alpha_m)$ a 2-bridge link and $i$ the lower segment, $G_{L,i}$ can be obtained as follows. In a line from right to left, place a downward $(\alpha_1 -1)$-birdwing, then a disjoint upward $(\alpha_2 -1)$-birdwing, a disjoint $(\alpha_3 -1)$-birdwing, etc. For $j=1, \dots m-1$, draw a \textbf{connecting edge} from the rightmost white (resp. black) vertex of the $j^\text{th}$ birdwing to the leftmost black (resp. white) vertex of the $(j+1)^\text{th}$ birdwing.
\end{lemma}

We call graphs constructed as in \cref{lem:2-bridge-gluing-birdwing} \textbf{flock graphs}. See \cref{fig:flock-graph-on-2-bridge} for an example of \cref{lem:2-bridge-gluing-birdwing}.

\begin{proof}
    This is clear from the definition of $C(\alpha_1, \dots, \alpha_m)$ and of the graph $G_{L,i}$. Indeed, by inspection, the downward $(\alpha_1 -1)$-birdwing is precisely the subgraph of $G_{L,i}$ generated by the white vertices corresponding to the regions enclosed by the first $\alpha_1$ crossings and the region above them, as well as the black vertices corresponding to those crossings. The $(\alpha_j-1)$-birdwing is similarly the subgraph of $G_{L,i}$ generated by the white vertices of regions enclosed by the $\alpha_j$ crossings and the region above or below them, and the black vertices of those crossings. The connecting edges surround the segment of $C(\alpha_1, \dots, \alpha_m)$ which goes from the $\alpha_j$ crossings to the $\alpha_{j+1}$ crossings.
\end{proof}

\begin{figure}[h]
    \centering
    \includegraphics[width=0.8\linewidth]{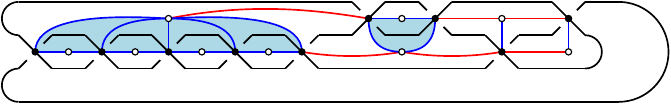}
    \caption{The diagram $L=C(5,2,1,1)$ and the graph $G_{L,i}$, where $i$ is the lower segment. The birdwings are shown in blue, while the connecting edges are in red.} 
    \label{fig:flock-graph-on-2-bridge}
\end{figure}

\begin{lemma}\label{lem:2-bridge-quiver}
For $L=C(\alpha_1,\dots,\alpha_m)$ a 2-bridge link and $i$ the lower segment, the quiver $Q_{G_{L,i}}$ is a type $A$ quiver (i.e. it is an orientation of a path).
\end{lemma}

\begin{proof}
The leftmost and rightmost non-infinite faces of $G_{L,i}$ are adjacent to a unique non-infinite face of $G_{L,i}$. All other non-infinite faces are adjacent to two others. Thus, $Q_{G_{L,i}}$ is an orientation of a path.
\end{proof}

In \cite[Section 4]{P05}, Propp shows that flock graphs have the same dimer lattices as \emph{snake graphs.}
A snake graph is a connected bipartite plane graph obtained by gluing together squares along one edge so that they ``snake" up and to the right in the plane. See \cref{fig:snake_}, right, for an example. A more precise definition following \cite[Sec. 2]{CS16} appears below.

\begin{definition}
\label{def:snake_graphs} 
Let $(\alpha_1, \dots, \alpha_m)$ be a sequence of positive integers whose sum is $d+1$. The associated \textbf{sign sequence} $(s_1,\dots, s_{d+1}) \in \{+,-\}^{d+1}$ consists of $\alpha_1$ minuses, then $\alpha_2$ pluses, then $\alpha_3$ minuses, and so on. The \textbf{snake graph} $\mathcal{G}[\alpha_1, \dots, \alpha_m]$ consists of $d$ boxes $B_1, \dots, B_d$ glued together as follows.
\begin{itemize}
    \item Place $B_1$ in the plane so that its edges are parallel to the axes. Label the bottom edge with $s_1=-$. 
    \item Suppose $B_1, \dots, B_{j}$ are already placed, and either the bottom or left edge of $B_j$ is labeled with $s_j$. If $s_{j+1} \neq s_j$, glue $B_{j+1}$ to $B_j$ along the edge of $B_{j+1}$ parallel to the labeled edge. If $s_{j+1} = s_j$, glue $B_{j+1}$ to $B_j$ along the unique edge of $B_j$ which shares a vertex with the labeled edge of $B_j$ and is either the top or the right edge. In both cases, label the shared edge of $B_{j+1}$ and $B_{j}$ with $s_{j+1}$. 
    \item Although there is no $B_{d+1}$ to glue to $B_d$, we label an edge of $B_{d}$ with $s_{d+1}$ according to the rules above.
\end{itemize}
 We properly color the vertices of $\mathcal{G}[\alpha_1, \dots, \alpha_m]$ so that the bottom left vertex is white. If two parallel edges of $B_j$ are labeled with $s_j$ and $s_{j+1}$, we call the other two edges \textbf{connecting edges} and draw them in red. 
\end{definition}

\begin{remark}\label{rem:snake-graph-zig-zag}
    If the connecting edges of $\mathcal{G}[\alpha_1, \dots, \alpha_m]$ are removed, we are left with a disjoint union of $m$ \emph{zig-zags}, consisting of $\alpha_j-1$ boxes for $j=1, \dots, m$. Said differently, $\mathcal{G}[\alpha_1, \dots, \alpha_m]$ consists of $m$ zig-zags attached using connecting edges. See \cref{fig:snake_}, right for an example; the edges of the zig-zags are in blue.
\end{remark}

\begin{figure}
    \centering
    \includegraphics[width=0.4\linewidth]{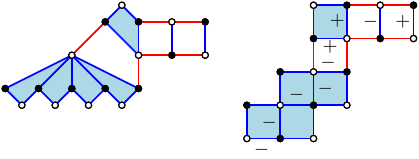}
    \caption{Left: The graph $G_{C(5,2,1,1), i}$. Right, the snake graph $\mathcal{G}(5,2,1,1)$. The sign sequence for the snake graph is $(-,-,-,-,-,+,+,-,+)$.} 
    \label{fig:snake_}
\end{figure}

As a companion to \cref{lem:2-bridge-quiver}, we also note the following result describing the dual quivers for snake graphs, using the same proof.

\begin{lemma}\label{lem:snake-quiver}
For a snake graph $G = \mathcal{G}[\alpha_1, \dots, \alpha_m]$ defined as in \cref{def:snake_graphs}, the quiver $Q_G$ is a type $A$ quiver (i.e. it is an orientation of a path).
\end{lemma}

The following proposition was shown in \cite[Section 4]{P05}; we include the proof here for the convenience of the reader.

\begin{proposition}\label{prop:flock-is-snake}
    Let $L=C(\alpha_1,\dots,\alpha_m)$ be a 2-bridge link and let $i$ be the lower segment. The dimer lattices of $G_{L,i}$ and the snake graph $\mathcal{G}[\alpha_1,\dots,\alpha_m]$ are isomorphic. 
\end{proposition}

\begin{proof}
Let $G$ be the graph constructed as follows: place two rows of $m$ vertices $v_1, \dots, v_m$ and $u_1, \dots u_m$ in the plane, and color $v_1$ white, $u_1$ black, and have the colors alternate along the rows. Add $\alpha_j$ edges between $v_j$ and $u_j$. Then add connecting edges from $v_{j}$ to $v_{j+1}$ and from $u_{j}$ to $u_{j+1}$.

We will show that using (E) moves, both $G_{L,i}$ and $\mathcal{G}:=\mathcal{G}[\alpha_1,\dots,\alpha_m]$ can be transformed into $G$. The result will then follow by \cref{lem:e-moves-dont-change-dimers}.

For the graph $G_{L,i}$, each $k$-birdwing may be contracted to two vertices with $k+1$ edges between them (see \cref{fig:birdwing-zigzag-collapse} for an example). So by \cref{lem:2-bridge-gluing-birdwing}, it is clear that $G_{L,i}$ can be transformed to $G$ using (E) moves.

\begin{figure}[h]
    \includegraphics[width=0.8\textwidth]{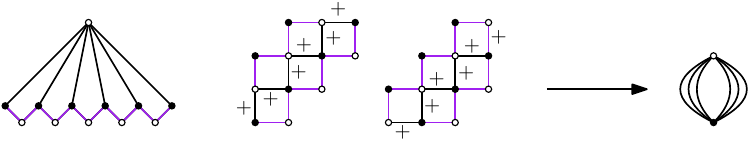}
    \caption{Using (E) moves to turn a birdwing or a zig-zag into two vertices connected by multiple edges. }
    \label{fig:birdwing-zigzag-collapse}
    \end{figure}

We now turn to the snake graph $\mathcal{G}$. Note that any edge of $\mathcal{G}$ which is in a zig-zag but is not labeled by a sign can be contracted away using (E) moves, since one of the vertices in such an edge is degree 2. Performing these moves collapses a zig-zag with $k-1$ boxes into a pair of vertices connected by $k$ edges (see \cref{fig:birdwing-zigzag-collapse} for an example). So, using \cref{rem:snake-graph-zig-zag}, $\mathcal{G}$ can be transformed into a properly colored graph $G'$ consisting of edges of multiplicity $\alpha_1, \dots, \alpha_m$ connected via connecting edges in a line. Rotating 45 degrees to the right, we see that the first of these edges has the white vertex on top, since the bottom left vertex of $\mathcal{G}$ is white. Thus, $G'=G$.
\end{proof}

We summarize the content of this section in the following proposition.

\begin{proposition}
\label{prop:TypeAF}
Let $L=C(\alpha_1, \dots, \alpha_m)$ be a 2-bridge link and $i$ its lower segment. Let $\mathcal{G}:=\mathcal{G}(\alpha_1, \dots, \alpha_m)$ be the corresponding snake graph.  The multivariate Alexander polynomial $D_{G_{L,i}}$ is equal to dimer face polynomial $D_{\mathcal{G}}$ of the snake graph. Moreover, it is a type $A$ $F$-polynomial $F_z^{Q_{G_{L,i}}}$. In particular, the Alexander polynomial $\Delta_L$ may be obtained from any of these polynomials by the specialization of \cref{thm:alexdimer}.
\end{proposition}
\begin{proof}
That $D_{G_{L,i}}$ is a type $A$ $F$-polynomial is a consequence of \cref{thm:dimer-poly-is-F-poly} and \cref{lem:2-bridge-quiver}. The equality of $D_{G_{L,i}}$ and $D_{\mathcal{G}}$ follows from \cref{prop:flock-is-snake}. The last sentence is implied by \cref{thm:alexdimer}.
\end{proof}

\section{Applications to cluster structures on positroid varieties} \label{sec:positroid-var-applications}

In this section, we relate our results to another common occurrence of bipartite plane graphs in cluster algebras: reduced plabic graphs and the cluster structure on open positroid varieties. In particular, we use our results to show for an open positroid variety, every Pl\"ucker coordinate and every twisted Pl\"ucker coordinate is a cluster monomial.

\subsection{Background on plabic graphs and open positroid varieties}
In this section, we denote by $\Gr_{k,n}$ the Grassmannian of $k$-planes in $\mathbb{C}^n$, and denote by $P_I$ the Pl\"ucker coordinate on $\Gr_{k,n}$ indexed by the $k$-element subset $I \subset [n]$. We assume familiarity with \cite{MS-twist,GL-positroid}, and point the reader to those sources for the geometric and combinatorial details omitted here.

Plabic graphs were originally defined by Postnikov \cite{Postnikov}, and in that formulation were not required to be bipartite. Here we largely follow the conventions of \cite[Section 3.1]{MS-twist}, adding Postnikov's \emph{leafless} assumption \cite[Definition 12.5]{Postnikov} for simplicity.

\begin{definition}
    A \textbf{plabic graph} of type $(k,n)$ is a planar bicolored graph $G$ embedded in a disk with $n$ black boundary vertices $1, \dots, n$ on the boundary of the disk such that
    \begin{itemize}
        \item each internal (i.e. non-boundary) vertex is adjacent only to vertices of the opposite color;
        \item each boundary vertex is adjacent to exactly zero or one internal vertex;
        \item any vertex of degree $1$ is adjacent to a boundary vertex;
        \item $G$ has $n-k$ more black vertices than white vertices. 
    \end{itemize}
    A plabic graph $G$ is \textbf{reduced} if no sequence of (E) or (S) moves applied to $G$ produces a bigon face. 
    The \textbf{dual quiver with frozens} $\frzQ_G$ of $G$ is the quiver obtained from $Q_G$ by freezing vertices corresponding to faces adjacent to the infinite face.
\end{definition}

Plabic graphs were introduced by Postnikov \cite{Postnikov} to parametrize positroid cells, which stratify the totally nonnegative Grassmannian. Inspired by his work, Knutson--Lam--Speyer introduced \emph{open positroid varieties}, which stratify $\Gr_{k,n}$ \cite{KLS}. Each reduced plabic graph $G$ of type $(k,n)$ corresponds to an open positroid variety $\Pi_G^{\circ} \subset \Gr_{k,n}$. Reduced plabic graphs related by (E) and (S) moves correspond to the same open positroid variety. Each open positroid variety $\Pi_G^{\circ}$ is cut out by the vanishing of certain Pl\"ucker coordinates and the non-vanishing of others (see \cref{lem:nonvanishing-pluckers-matching}).

For any reduced plabic graph $G$, there is a combinatorial procedure to label its faces with Pl\"ucker coordinates $\{P_{I_f}\}_{f \in \faces(G)}$, called the \emph{source}\footnote{One may also label the faces of $G$ with its \emph{target} Pl\"ucker coordinates. This will produce the target seed rather than the source seed. The target and source seeds are related by two applications of the twist map, follows by rescaling by Laurent monomials in frozen variables.} Pl\"ucker coordinates of $G$. These are non-vanishing on $\Pi_G^\circ$. The \textbf{source seed} $\Sigma_G$ associated to $G$ is the seed whose quiver is $\frzQ_G$ and whose cluster is $\{P_{I_f}\}_{f \in \faces(G)}$. 

The following theorem is due to Scott \cite{Scott} in the case of the top-dimensional positroid variety, Serhiyenko--Sherman-Bennett--Williams \cite{SSBW} in the case of open Schubert varieties, and Galashin--Lam \cite{GL-positroid} for arbitrary open positroid varieties.

\begin{theorem}\label{thm:positroid-cluster-struc}
    Let $G$ be a reduced plabic graph and $\Sigma_G$ the corresponding source seed. Then 
    \[\mathbb{C}[\Pi_G^{\circ}]= \mathcal{A}(\Sigma_G).\]
    Further, if $G, G'$ are related by a square move, then $\Sigma_G$ and $\Sigma_{G'}$ are related by mutation. 
\end{theorem}

In \cite{MS-twist}, Muller--Speyer introduced an automorphism $\tau$ on $\Pi_G^{\circ}$ called the \emph{twist}\footnote{They define the left and right twists, which are inverses of each other.  We use the left twist here.}. (Their definition was motivated by previous work of Marsh--Scott \cite{marsh2016twists} who defined a twist map in the case of Grassmannians, i.e. the top-dimensional positroid variety. The Marsh--Scott twist differs from the Muller--Speyer twist for the top-dimensional positroid variety by rescaling.) While we do not need its precise definition here, Muller--Speyer provided a formula for ``twisted Pl\"ucker coordinates" $P_J \circ \tau$ in terms of the source Pl\"uckers of $G$. This formula involves almost perfect matchings of $G$.

\begin{definition}
\label{def:almost-perfect}
    Let $G$ be a plabic graph of type $(k,n)$. An \textbf{almost-perfect matching} $M$ of $G$ is a subset of edges which covers each internal vertex exactly once and does not include any edges between boundary vertices. The \textbf{boundary} of $M$ is
    \[\partial M := \{i \in [n]: i \text{ is covered by an edge of }M\}.\]
    As $G$ is type $(k,n)$, $\partial M$ has size $k$.
\end{definition}

We have the following lemma, which will be useful later and follows from \cite[Lemma 11.10]{Postnikov} and \cite[Corollary 5.12]{KLS}.

\begin{lemma}\label{lem:nonvanishing-pluckers-matching}
    Let $G$ be a reduced plabic graph of type $(k,n)$. The Pl\"ucker coordinate $P_J$ is not identically zero on the positroid variety $\Pi^{\circ}_{G}$ if and only if $G$ has an almost perfect matching with boundary $J$.
\end{lemma}

We may flip a face in an almost-perfect matching, just as in \cref{def:dimer-poset}. We define a binary relation $\leq$ on the almost perfect matchings of $G$ with a fixed boundary using flips, exactly as in \cref{thm:dimer-lattice}.

\begin{theorem}[{\cite[Theorem B.1]{MS-twist}}]\label{thm:dimer-lattice-plabic}
    Let $G$ be a reduced plabic graph. Then the set of almost perfect matchings of $G$ with boundary $I$, endowed with the binary relation $\leq$, is a distributive lattice.
\end{theorem}

In what follows,  if $G$ is a plabic graph, let $\hat{0}_J$ denote the minimal matching with boundary $J$. We let $x_f$ denote the source Pl\"ucker coordinate $P_{I_f}$. For $f \in \faces(G)$, let $w_f$ denote the number of white vertices in $f$ and let
\[\bd_f(J):= |\{j \in J\colon j \text{ is not adjacent to degree 1 vertex and the boundary edge } (j-1, j) \text{ is in } f\}|.\]
We use the notation
\[\mathbf{x}^{\bd(J)}:= \prod_{f \in \faces(G)} x_f^{\bd_f(J)}.\]
If $G$ is a plabic graph and $M$ a subset of edges which are not in the boundary of the disk, we define the vector $\hvec^G_M=(h_f)_{f\in \faces(G)}$ by
\[h_f:= w_f - |M \cap f| -1.\]
Note that if $f$ is not adjacent to the boundary of the disk, then $w_f=|f|/2.$ 

\begin{proposition}[{\cite[Proposition 7.10]{MS-twist}}] \label{prop:twist-in-dimers}
    Let $G$ be a reduced plabic graph. Then for any Pl\"ucker coordinate $P_J$, 
        \begin{align*}
        P_J \circ \tau &= \mathbf{x}^{\bd(J)} \sum_{M: \partial M =J} \mathbf{x}^{\hvec^G_M}.
        \end{align*}
\end{proposition}

We note that the equalities in \cref{prop:twist-in-dimers} are consistent with \cref{lem:nonvanishing-pluckers-matching}, since if there are no almost perfect matchings with boundary $J$, the sum in \cref{prop:twist-in-dimers} is zero.

\begin{remark} 
For the special case they consider, Marsh--Scott prove a formula \cite[Theorem 1.1]{marsh2016twists} which, suitably rescaled, is algebraically equivalent to that of \cref{prop:twist-in-dimers}, see \cite[Remark 3.7, Proposition 3.8]{elkin2023twists} for the derivation.
\end{remark}

Finally, we need a result relating the twist map to the cluster structure on $\mathbb{C}[\Pi_G^{\circ}]$, which was proved independently by \cite{P23,CLSBW23}. Recall that a cluster monomial is a monomial of compatible cluster variables times a Laurent monomial in frozen variables.

\begin{theorem}\label{thm:twist-quasi-cluster}
    Let $G$ be a reduced plabic graph. Then 
    \begin{align*}\tau^*:\mathbb{C}[\Pi_G^{\circ}] &\to \mathbb{C}[\Pi_G^{\circ}] \\
    P_J &\mapsto P_J \circ \tau
    \end{align*}
    is a quasi-cluster homomorphism. In particular, $\tau^*$ and its inverse take cluster monomials to cluster monomials.
\end{theorem}
 
\subsection{All Pl\"uckers are cluster monomials}
In this section, we will utilise \cref{thm:dimer-poly-is-F-poly} to show the following theorem.

\begin{theorem}\label{thm:pluckers-cluster-monomials}
    Let $\Pi_{G}^{\circ}$ be an open positroid variety, and let $P_J$ be a Pl\"ucker coordinate which is not identically zero on $\Pi_{G}^{\circ}$. Then $P_J$ and its twist $P_J \circ \tau$ are both cluster monomials in $\mathbb{C}[\Pi_{G}^{\circ}]= \mathcal{A}(\Sigma_G)$.
\end{theorem}

\begin{remark}
    The statement of \cref{thm:pluckers-cluster-monomials} also holds for the target cluster structure on $\mathbb{C}[\Pi_{G}^{\circ}]$, since by \cite{P23,CLSBW23} the source and target cluster structure have the same cluster monomials.
\end{remark}

\begin{remark}\label{rem:history-plucker-cluster-monomials}
In the special case when $P_J$ is a source (or target) Pl\"ucker coordinate for some reduced plabic graph $G'$ for $\Pi^{\circ}_G$, then $P_J$ is a cluster variable in $\mathbb{C}[\Pi_{G}^{\circ}]$ by \cite{GL-positroid} (see \cref{thm:positroid-cluster-struc}) and its twist is a cluster monomial by \cite{P23,CLSBW23} (see \cref{thm:twist-quasi-cluster}). Similarly, if $P_J$ is a unit in $\mathbb{C}[\Pi_{G}^{\circ}]$, then by \cite[Theorem 1.3 (1)]{GLS-factorial} it is a Laurent monomial in frozen variables and in particular is a cluster monomial. So in these special cases, \cref{thm:pluckers-cluster-monomials} was already known. In particular, when $\Pi^{\circ}_G$ is the unique top-dimensional positroid variety, \cref{thm:pluckers-cluster-monomials} was known for all Pl\"ucker coordinates, since all Pl\"ucker coordinates are source Pl\"ucker coordinates for some plabic graph (in this case, by \cite[Proposition 8.10]{marsh2016twists}, twists of Pl\"ucker coordinates are cluster variables times a Laurent monomial in frozen variables). 

However, for most positroids, only a small fraction of the nonzero Pl\"ucker coordinates are source (or target) Pl\"uckers. For example, some positroid varieties $\Pi^\circ_G$ have only one reduced plabic graph and only $|\faces(G)|$ source Pl\"ucker coordinates, but have many more nonzero, non-unit Pl\"ucker coordinates.
\end{remark}

To prove \cref{thm:pluckers-cluster-monomials}, we would like to connect \cref{prop:twist-in-dimers} to \cref{thm:dimer-poly-is-F-poly}. The latter theorem involves \emph{perfect matchings} of graphs with property $(*)$, while the twist formula uses \emph{almost perfect matchings} of plabic graphs. We first need some technical results to relate these two scenarios.

For $G$ a plane graph, let $\leq_G$ be the binary relation on $\mathcal{D}_G$ where $M \leq_G M'$ if $M$ is obtained from $M'$ by a sequence of down-flips on simply connected faces. Recall from \cref{rmk:dimer-lattice-arbitrary-graph} that there is a refinement $\leq$ of $\leq_G$ which endows $\mathcal{D}_G$ with the structure of a distributive lattice.  The following lemma gives a sufficient condition for this refinement $\leq$ to be equal to $\leq_G$.

\begin{proposition}\label{prop:arbitrary-graph-dimer-lattice-good}
    Let $G$ be a plane graph such that every non-infinite face is simply connected. Let $H = H_1 \sqcup \cdots \sqcup H_r$ be the graph obtained from $G$ by deleting all edges which are not in any perfect matching. If $(\mathcal{D}_G, \leq_G)$ is a poset with a unique minimum or maximum, then every non-infinite face of $H$ is a face of $G$. Further, $(\mathcal{D}_G, \leq_G)$ is the Cartesian product of the distributive lattices $(\mathcal{D}_{H_i}, \leq_{H_i})$, and is also equal to $(\mathcal{D}_H, \leq_H)$.
\end{proposition}

\begin{proof}
Note that $\mathcal{D}_G, \mathcal{D}_H$ and the product $ \prod_i \mathcal{D}_{H_i}$ are all in natural bijection. So we identify the matchings in all three sets. We also identify a face with the edges in its boundary.
    We let $\leq$ denote the partial order on the product $\prod_{i} \mathcal{D}_{H_i}$. 
    
    If $M \lessdot_G M'$ are related by a flip at face $f$ of $G$, then all edges of $f$ are in a perfect matching of $G$. Thus, $f$ is also a face of some $H_i$, which implies $M \lessdot M'$. This means that if $M \leq_G M'$, then $M \leq M'$. In particular, $\leq_G$ is a coarsening of $\leq$. This implies that if $(\mathcal{D}_G, \leq_G)$ has a unique minimum (resp. maximum), it  agrees with the minimum (resp. maximum) of $(\mathcal{D}_G, \leq)$.

    Now, suppose for the sake of contradiction that some non-infinite face of $H$ is not a face of $G$. This face may not be simply connected. One connected component of the boundary of this face is the boundary of a non-infinite face $f$ of some $H_i$. The face $f$ is not a face of $G$. Since $H_i$ has property $(*)$ by construction, \cref{cor:prop*=all-faces-flipped} implies $f$ can be flipped in some matching of $G_i$. That is, there exist a matching of $G$ containing all white-black edges of $f$, and there exists a matching of $G$ containing all black-white edges of $f$. However, since $f$ is not a face of $G$, we cannot perform a flip at $f$.

    The set 
    \[WB=\{M \in \mathcal{D}_G: M \text{ contains the white-black edges of }f\}\]
    is an order filter in $(\mathcal{D}_G, \leq_G)$. Indeed, if $M \in WB$, then for any face $F'$ adjacent to $f$, either $M$ contains a black-white edge of $f'$ or $M$ does not contain some white-black edge of $f'$. So if $M'$ is obtained from $M$ by an up-flip at some face $f'$, $f'$ is not adjacent to $f$ and $M'$ is also in $WB$. By the definition of $\leq_G$, this implies that if $M \in WB$, then so are all $M'$ with $M' \geq_G M$.

    Similarly, the set 
     \[BW=\{M \in \mathcal{D}_G: M \text{ contains the black-white edges of }f\}\]
     is an order ideal in $(\mathcal{D}_G, \leq_G)$.

    In the case that $(\mathcal{D}_G, \leq_G)$ has a unique minimial element $\hat{0}$, this implies $\hat{0}$ is in $BW$. As we already argued, $\hat{0}$ is also the minimum element of $(\mathcal{D}_G, \leq)$. But this is impossible; since $f$ is a face of $G_i$ and $\hat{0}$ contains the black-white edges of $f$, we can perform a down-flip at $f$ in $H_i$. So $\hat{0}$ cannot be minimal in $(\mathcal{D}_G, \leq)$, a contradiction. An identical argument leads to a contradiction in the case that $(\mathcal{D}_G, \leq_G)$ has a unique maximal element, exchanging $BW$ for $WB$.

    This completes the argument that every non-infinite face of $H$ is a face of $G$. In particular, every non-infinite face of $H$ is simply connected. This, in turn, implies that every non-infinite face of each $H_i$ is a face of $H$ and thus of $G$. So, if a face of $G_i$ can be flipped in some matching in $\prod_i \mathcal{D}_{H_i}$, it is also a face of $G$ and of $H$ and can be flipped in those graphs as well. This shows that every cover relation in $(\mathcal{D}_G, \leq)$ is also a cover relation in $(\mathcal{D}_G, \leq_G)$ and $(\mathcal{D}_H, \leq_H)$ and so the posets coincide.
\end{proof}

\begin{lemma}\label{lem:setup-for-twist-cluster-mono}
    Let $G$ be a reduced plabic graph, and let $J$ be the boundary of some almost perfect matching of $G$. Let $H= H_1 \sqcup \cdots \sqcup H_r$ be the plane graph obtained from $G$ by deleting
    \begin{enumerate}
        \item all boundary vertices;
        \item all internal vertices adjacent to boundary vertices $\{j \in J\}$;
        \item all edges which are not in any almost perfect matching of $G$ with boundary $J$
    \end{enumerate}
    so that every connected component $H_i$ of $H$ has property $(*)$.
    Then under the natural identification of edges of $H$ with edges of $G$, the lattice of almost-perfect matchings of $G$ with boundary $J$ (see \cref{thm:dimer-lattice-plabic}) is isomorphic to $(\mathcal{D}_H, \leq_H)$ and the product of lattices $\prod_i (\mathcal{D}_{H_i}, \leq_{H_i})$.

    Further, every non-infinite face of $H$ is a face of $G$, and the dual quiver $Q_H$ is an induced subquiver of $\frzQ_G$ consisting only of mutable vertices.
\end{lemma}

\begin{proof}
Reducedness implies that $G$ is connected and so its faces are simply connected. Once we perform the deletions in steps (1) and (2), we have a plane graph $G'$ whose non-infinite faces are simply connected. Note that $G'$ has at least one perfect matching, since any almost-perfect matching of $G$ with boundary $J$ restricts to a perfect matching of $G'$. Also, every non-infinite face of $G'$ is a face of $G$, and the lattice of almost-perfect matchings of $G$ with boundary $J$ is clearly isomorphic to $(\mathcal{D}_{G'}, \leq_{G'}).$ By \cref{thm:dimer-lattice-plabic}, this implies that $(\mathcal{D}_{G'}, \leq_{G'})$ has a unique minimal element.
    Applying \cref{prop:arbitrary-graph-dimer-lattice-good} to $G'$, we get an isomorphism between $(\mathcal{D}_{G'}, \leq_{G'})$, $(\mathcal{D}_{H}, \leq_{H})$ and the product $\prod_i (\mathcal{D}_{H_i}, \leq_{H_i})$. By the same proposition, we also have that every non-infinite face of $H$ is a face of $G$.

    The statement about the dual quiver follows from the fact that every non-infinite face of $H$ is a face of $G$, and the definition of the dual quiver.
\end{proof}

\begin{remark}
Consider the graph $G'$ from the proof of \cref{lem:setup-for-twist-cluster-mono}. As mentioned therein, the dimers of $G'$ are in natural bijection with the almost perfect matchings of $G$ with boundary $J$. In \cite[Appendix B]{MS-twist}, they explicitly avoid using $G'$ to prove \cref{thm:dimer-lattice-plabic}. We point out that we make crucial use of \cref{thm:dimer-lattice-plabic} to deduce properties of the dimers of $G'$, rather than the other way around. In particular, we use \cref{thm:dimer-lattice-plabic} when we invoke \cref{prop:arbitrary-graph-dimer-lattice-good}, to ensure that deleting the edges of $G'$ that are not in any dimer does not create any new faces or any new flips.
\end{remark}

\begin{lemma}\label{cor:plabic-graph-f-poly-product}
    Let $G$ be a reduced plabic graph, and let $J$ and $H=H_1 \sqcup \cdots \sqcup H_r$ be as in \cref{lem:setup-for-twist-cluster-mono}. Then 
    \[D_H(\mathbf{y}) = \prod_i D_{H_i}(\mathbf{y})\]
    is a product of compatible $F$-polynomials for the cluster algebra $\mathcal{A}(\mathbf{x}, Q_G)$.
\end{lemma}

\begin{proof}
If $H_i$ is a single edge, then $D_{H_i}=1$, so we may ignore those components.

    For each $i$ where $H_i$ is not a single edge, the quiver $Q_{H_i}$ is a connected component of the quiver $Q_H$. By \cref{lem:setup-for-twist-cluster-mono} the quiver $Q_H$ is an induced subquiver of $Q_G$, implying that $Q_{H_i}$ is also an induced subquiver of $Q_G$. Since $H_i$ has property $(*)$ by construction, $D_{H_i}(\mathbf{y})$ is the $F$-polynomial of a cluster variable $z'_i$ in $\mathcal{A}(\mathbf{x}, Q_{H_i})$ by \cref{thm:dimer-lattice}. By \cref{prop:F-g-adding-mutable}, $D_{H_i}(\mathbf{y})$ is also the $F$-polynomial of a cluster variable $z_i$ in $\mathcal{A}(\mathbf{x}, Q_{G})$, and in particular can be obtained from the initial seed just by mutating at the subquiver $Q_{H_i}$. Since the quivers $Q_{H_i}$ and $Q_{H_j}$ are disjoint for $i \neq j$, $z_i$ and $z_j$ are compatible. 
\end{proof}

\begin{proof}[Proof of \cref{thm:pluckers-cluster-monomials}]
Fix a reduced plabic graph $G$. Fix $P_J$ which is not identically zero on $\Pi_G^{\circ}$ or equivalently, $J$ such that $G$ has an almost-perfect matching with boundary $J$. 

We will show that $P_J \circ \tau$ is a cluster monomial. By \cref{thm:twist-quasi-cluster}, this implies that $P_J$ is a cluster monomial as well.

Recall from \cref{prop:twist-in-dimers} that the formula for $P_J \circ \tau$ in terms of the seed $\Sigma_G=(\{x_f\}, \frzQ_G)$ is
    \begin{align*}
    P_J \circ \tau &= \mathbf{x}^{\bd(J)} \sum_{N: \partial N =J} \mathbf{x}^{\hvec^G_N}
        .\end{align*}
        
        If $G$ has a unique matching with boundary $J$, then this matching is $\hat{0}_J$ and the sum above is trivial. So in this case, 
        \[P_J \circ \tau= \mathbf{x}^{\bd(J)} \mathbf{x}^{\hvec^G_{\hat{0}_J}}.\]
        Notice that since no faces can be flipped, any face of $G$ not adjacent to the boundary of the disk satisfies $w_f-1 \ge |\hat{0}_J \cap f|$, so the entry $h_f$ of $\hvec^G_{\hat{0}_J}$ is nonnegative. We also have $\bd_f(J) =0$ for such faces. Thus, for any mutable vertex $f$ of $\frzQ_G$, $x_f$ appears in $P_J \circ \tau$ with a nonnegative power. So $P_J \circ \tau$ is a cluster monomial.

        If $G$ has more than one almost-perfect matching with boundary $J$, let $H= H_1 \sqcup \cdots \sqcup H_r$ be as in \cref{lem:setup-for-twist-cluster-mono}. Say $H_1, \dots, H_p$ have at least one non-infinite face, and $H_{p+1}, \dots, H_r$ are single edges. By \cref{cor:plabic-graph-f-poly-product}, the dimer face polynomial $D_H(\mathbf{y})= \prod_i D_{H_i}$ is a product of compatible $F$-polynomials in $\mathcal{A}(\Sigma_G)$. For $i \in [p]$, let $z_i$ denote the cluster variable in $\mathcal{A}(\Sigma_G)$ whose $F$-polynomial is $D_{H_i}(\mathbf{y})$. 
        
        By an argument similar to the proof of \cref{cor:easy-cluster-expansion}, 
        \[z_i = \mathbf{x}^{\mathbf{g}_i} \sum_{M \in \mathcal{D}_{H_i}} \mathbf{x}^{\hvec^G_M-\hvec^G_{\hat{0}_{H_i}}}.
        \]
        By \cref{thm:dimer-poly-is-F-poly} and \cref{prop:F-g-adding-mutable}, for $f \in \faces(H_i)$, the $g$-vector entry is $g_f=w_f - |\hat{0}_{H_i} \cap f|-1$. For other faces of $G$ corresponding to mutable vertices of $\frzQ_G$, the $g$-vector entry is 
        \[g_f= \max_{M \in \mathcal{D}_{H_i}} |M \cap f| - |\hat{0}_{H_i} \cap f|.\]
        This follows by a very similar argument to the proof of \cref{thm:cluster-stuff-for-link-diag}. In particular, $g_f$ is zero unless $f$ is a face of $H_i$ or is adjacent to a face of $H_i$. 

        Now, we consider the formula for the cluster monomial $z_1 z_2 \dots z_p$ in terms of the initial cluster. Using that the dimer lattice on $H$ is the product of dimer lattices on $H_1, \dots, H_r$ and that the components $H_{p+1}, \dots, H_r$ do not contribute to the exponents for terms of the sum, we obtain
        \[z_1 \cdots z_p = \left(\prod_{i=1}^p \mathbf{x}^{\mathbf{g}_i} \right) \sum_{M \in \mathcal{D}_{H}} 
        \mathbf{x}^{\hvec^G_M-\hvec^G_{\hat{0}_{H}}}.\]

        If $f\in \faces(H)$, then at most one $z_i$ has $g$-vector with nonzero $g_f$. So the exponent of $x_f$ in $\prod_{i=1}^p \mathbf{x}^{\mathbf{g}_i}$ is again 
        \[w_f - |\hat{0}_{H_i} \cap f|-1 = w_f - |\hat{0}_{H} \cap f|-1.\]
        Otherwise, if $f \notin \faces(H)$ and corresponds to a mutable vertex in $\frzQ_G$, the exponent of $x_f$ in $\prod_{i=1}^p \mathbf{x}^{\mathbf{g}_i}$ is
        \begin{equation}\label{eq:exponent}\sum_{i=1}^p \max_{M \in \mathcal{D}_{H_i}} |M \cap f| - |\hat{0}_{H_i} \cap f| =  \sum_{i=1}^r \max_{M \in \mathcal{D}_{H_i}} |M \cap f| - |\hat{0}_{H_i} \cap f|
        = \max_{M \in \mathcal{D}_{H}} |M \cap f| - |\hat{0}_{H} \cap f|.\end{equation}
        The first equality holds because for $i=p+1, \dots, r$, the term corresponding to $H_i$ is zero, and the second holds because $H= H_1 \sqcup \cdots \sqcup H_r$.
        Note that the rightmost expression in \eqref{eq:exponent} is at most $w_f-1 - |\hat{0}_H \cap f|$, because if it were $w_f- |\hat{0}_H \cap f|$, then the face $f$ could be flipped in some matching and thus would be a face of $H$, contradicting our assumption.

        In summary, by multiplying $z_1 \cdots z_p$ by some monomial $Z$ in $\{x_f: f \notin \faces(H)\}$, we can obtain
        \begin{equation}\label{eq:easy-cluster-mono}
        Z \cdot z_1 \cdots z_p = \mathbf{x}^{\hvec^G_{\hat{0}_H}} \sum_{M \in \mathcal{D}_H} \mathbf{x}^{\hvec^G_M - \hvec^G_{\hat{0}_H}}= \sum_{M \in \mathcal{D}_H} \mathbf{x}^{\hvec^G_M}.
        \end{equation}
As all cluster variables $\{x_f: f \notin \faces(H)\}$ are compatible with $z_1, \dots, z_p$, the left hand side of \eqref{eq:easy-cluster-mono} is a cluster monomial. 

Finally, a matching $M$ of $H$ differs from a matching $N$ of $G$ with boundary $J$ only at the edges adjacent to a boundary vertex $j \in J$. So $|M \cap f|$ differs from $|N \cap f|$ only if $x_f$ is a frozen variable in $\Sigma_G$, and in this case, 
\[|M \cap f|= |N \cap f| - |\{j \in J\colon \text{vertex }j \in f\}|.\]
So we may rewrite \eqref{eq:easy-cluster-mono} using almost perfect matchings of $G$ with boundary $J$ as
\[ Z \cdot z_1 \cdots z_p =\left(\prod_{f \in \faces(G)} x_f^{|\{j \in J\colon j \in f\}|}\right)
\sum_{N: \partial N=J} \mathbf{x}^{\hvec^G_N}.\]
By \cref{prop:twist-in-dimers}, 
\[\left(\frac{\mathbf{x}^{\bd(J)}}{\prod_{f \in \faces(G)} x_f^{|\{j \in J\colon \text{vertex }j \in f\}|}}\right) Z \cdot z_1 \cdots z_p = P_J \circ \tau. \]
The left hand side is also a cluster monomial. By definition, $\bd_f(J)=0$ if $f$ is not adjacent to the boundary of the disk, so $\mathbf{x}^{\bd(J)}$ is a monomial in frozen variables. Also, the denominator above is a product of frozen variables. So the left hand side is a cluster monomial times a Laurent monomial in frozens, and thus is a cluster monomial.
This shows that $P_J \circ \tau$ is a cluster monomial, as desired.
\end{proof}

\section*{Acknowledgements} The authors are grateful to Mario Sanchez and Jim Propp for inspiring conversations. KM is supported by the National Science Foundation under Award No.~DMS-1847284 and Award No.~DMS-2348676. GM is supported by the National Science Foundation under Award No.~DMS-1854162.
 MSB is supported by the National Science Foundation under Award No.~DMS-2103282. 
 Any opinions, findings, and conclusions or recommendations expressed in this material are
those of the author(s) and do not necessarily reflect the views of the National Science
Foundation.

\bibliographystyle{amsplain}
\bibliography{refs}

\end{document}